\documentclass[12pt]{amsart}
\usepackage{times}
\usepackage{nomencl}
\usepackage{anysize}
\marginsize{2.5cm}{2cm}{1.7cm}{2.9cm}
 \usepackage[latin1]{inputenc}
 \usepackage[dvips]{graphicx}
 \usepackage{wrapfig}
 \usepackage{amsmath}
 \usepackage{amsthm}
 \usepackage{amsfonts}
 \usepackage{amssymb}
 \usepackage{layout}
 \usepackage{verbatim}
 \usepackage{alltt}
\usepackage{yfonts}
\usepackage{stmaryrd}
 \usepackage{indentfirst}

\usepackage[all]{xy}
\usepackage{setspace}
\newtheorem*{ithm}{Theorem}
\newtheorem*{thma}{Theorem A}

\newtheorem*{conj}{Conjecture}

\newcounter{thm}[section]

\renewcommand{\thethm} {\arabic{section}.\arabic{thm}}

\newenvironment{thm}{\par\medskip\noindent\refstepcounter{thm}
\bgroup{\hspace*{-0.15 cm}\bf{Theorem}
\thethm.}\bgroup\it}{\egroup \egroup\par\medskip}

\newenvironment{lemma}{\par\medskip\noindent\refstepcounter{thm}
\bgroup{\hspace*{-0.15 cm}\bf{Lemma} \thethm.}\bgroup\it}{\egroup
\egroup\par\medskip}

\newenvironment{prop}{\par\medskip\noindent\refstepcounter{thm}
\bgroup{\hspace*{-0.15 cm}\bf{Proposition}
\thethm.}\bgroup\it}{\egroup \egroup\par\medskip}

\newenvironment{cor}{\par\medskip\noindent\refstepcounter{thm}
\bgroup{\hspace*{-0.15 cm}\bf{Corollary}
\thethm.}\bgroup\it}{\egroup \egroup\par\medskip}

\newenvironment{define}{\par\medskip\noindent\refstepcounter{thm}
\bgroup{\hspace*{-0.15 cm}\bf{Definition}
\thethm.}\bgroup}{\egroup \egroup\par\medskip}

\newenvironment{rem}{\par\medskip\noindent\refstepcounter{thm}
\bgroup{\hspace*{-0.15 cm}\bf{Remark} \thethm.}\bgroup}{\egroup
\egroup\par\medskip}

\newcounter{Athm}[section]

\renewcommand{\theAthm} {\arabic{Athm}}

\newenvironment{Athm}{\par\medskip\noindent\refstepcounter{Athm}
\bgroup{\hspace*{-0.15 cm}\bf{Theorem}
A.\theAthm.}\bgroup\it}{\egroup \egroup\par\medskip}

\newenvironment{Aprop}{\par\medskip\noindent\refstepcounter{Athm}
\bgroup{\hspace*{-0.15 cm}\bf{Proposition}
A.\theAthm.}\bgroup\it}{\egroup \egroup\par\medskip}

\newenvironment{Adefine}{\par\medskip\noindent\refstepcounter{Athm}
\bgroup{\hspace*{-0.15 cm}\bf{Definition}
A.\theAthm.}\bgroup}{\egroup \egroup\par\medskip}

\newenvironment{Arem}{\par\medskip\noindent\refstepcounter{Athm}
\bgroup{\hspace*{-0.15 cm}\bf{Remark} A.\theAthm.}\bgroup}{\egroup
\egroup\par\medskip}

\long\def\symbolfootnote[#1]#2{\begingroup%
\def\thefootnote{\fnsymbol{footnote}}\footnote[#1]{#2}\endgroup} 

\newcommand{\LL}{\Lambda}
\newcommand{\QQ}{\mathbb{Q}}
\newcommand{\FF}{\mathcal{F}}

\newcommand{\oo}{\mathcal{O}}
\newcommand{\FFc}{\mathcal{F}_{\textup{\lowercase{can}}}}
\newcommand{\lra}{\longrightarrow}
\newcommand{\ZZ}{\mathbb{Z}}
\newcommand{\PP}{\mathcal{P}}
\newcommand{\NN}{\mathcal{N}}
\newcommand{\ra}{\rightarrow}	
\newcommand{\xx}{\mathbf{X}}
\newcommand{\be}{\begin{equation}}
\newcommand{\ee}{\end{equation}}
\newcommand{\sss}{\scriptsize}

\newcommand{\XX}{\mathcal{X}}
\newcommand{\mm}{\hbox{\frakfamily m}}

\newcommand{\KS}{\textbf{\textup{KS}}}
\newcommand{\length}{\textup{length}}
\newcommand{\tk}{\tilde{K}}
\newcommand{\tH}{\tilde{H}}
\newcommand{\tgamma}{\tilde{\gamma}}
\newcommand{\tlambda}{\tilde{\Lambda}}
\newcommand{\tGamma}{\tilde{\Gamma}}
\newcommand{\fontaineO}{\mathcal{O}_{\widehat{\varepsilon^{\textup{nr}}}}}
\newcommand{\fontaineOK}{\mathcal{O}_{\varepsilon(K)}}
\newcommand{\htam}{$\mathbf{\mathbb{H}.T}$}
\newcommand{\hsez}{$\mathbf{\mathbb{H}.sEZ}$}
\newcommand{\hez}{$\mathbf{\mathbb{H}.EZ}$}

\newcommand{\unr}{\textup{unr}}
\newcommand{\Gal}{\textup{Gal}}
\newcommand{\Fr}{\textup{Fr}}
\newcommand{\TT}{\mathcal{T}}

\newcommand{\hone}{$\mathbf{H.1}$}

\newcommand{\hthree}{$\mathbf{H.3}$}
\newcommand{\hfour}{$\mathbf{H.4}$}

\newcommand{\xiv}{\mbox{\boldmath$\xi$}}
\newcommand{\FFl}{\mathcal{F}_{\mathcal{L}}}
\newcommand{\LLL}{\mathbb{L}_{\infty}}
\newcommand{\hsezF}{$\mathbf{\mathbb{H}.sEZ}_{/F}$}
\newcommand{\htamF}{$\mathbf{\mathbb{H}.T}_{/F}$}
\newcommand{\Ll}{\mathcal{L}}
\newcommand{\kk}{\mathcal{K}}

\def\addsymbol #1: #2#3{$#1$ \> \parbox{5in}{#2 \dotfill \pageref{#3}}\\}
\def\newnot#1{\label{#1}} 

\begin{document}
\title{$\Lambda$-adic Kolyvagin Systems}

\author{K\^az\i m B\"uy\"ukboduk}

\address{Kazim Buyukboduk \hfill\break\indent
Max Planck Institut f\"ur Mathematik\\
\hfill\break\indent Vivatsgasse 7 \hfill\break\indent 53111 Bonn
\hfill\break\indent Germany} \email{{\tt kazim@math.stanford.edu}}

\curraddr{
\hfill\break\indent Ko\c{c} University, Mathematics \hfill\break\indent Rumeli Feneri Yolu \hfill\break\indent 34450 Sar\i yer / \.Istanbul
\hfill\break\indent Turkey}


\begin{abstract}

In this paper we study Kolyvagin Systems, as defined by Mazur and Rubin, over the cyclotomic $\ZZ_p$-tower for a Gal$(\overline{\QQ}/\QQ)$-representation $T$. We prove, under certain hypotheses, that the module of $\LL$-adic Kolyvagin Systems for the \emph{cyclotomic deformation} $T\otimes\LL$ is free of rank one over the cyclotomic Iwasawa algebra. We link our result with a web of conjectures due to Perrin-Riou and Rubin; and we relate the $\LL$-adic Kolyvagin Systems we prove to exist to (conjectural) $p$-adic $L$-functions. We also study the Iwasawa theory of Rubin-Stark elements via the perspective offered by our main theorem, and outline a strategy to deduce the main conjectures of Iwasawa theory for totally real number fields assuming the Rubin-Stark conjecture.
\end{abstract}

\maketitle
\tableofcontents
\section{Introduction}
Kolyvagin's machinery of Euler Systems is a powerful tool to bound the order of Selmer groups. It takes an Euler System (which is a collection of cohomology classes  with certain coherence relations) as an input, and produces, using a descent argument, the so called derivative classes. These classes give the bounds we seek for.

There is of course a considerable cost for such a useful tool: Finding an Euler System is an extremely difficult task. There are only a few Euler systems known to exist, all of which with great importance, and yet we fail to know if there should be Euler Systems in any other setting than these few number of instances.

However, one may answer an existence problem for the derivative classes. In~\cite{mr02}, Mazur and Rubin isolate the notion of a \emph{Kolyvagin system}, which has the same formal properties that Kolyvagin's derivative classes have, except that the {Kolyvagin Systems} enjoy some other local conditions than utilized before. This new property adds to the level of rigidity, and one  gains  control over the size of the module of Kolyvagin Systems.

Fix a finite extension $\Phi$\newnot{symbol:ff} of $\QQ_p$, and let $\oo$\newnot{symbol:coeff} be its ring of integers, with $\mm$\newnot{symbol:maxidealO} its maximal ideal, $\varpi$\newnot{symbol:uniformizerO} its uniformizer, and $\mathbb{F}=\oo/\mm$ its residue field. For any field $K$ with a fixed separable closure $\overline{K}$, write $G_{K}=\textup{Gal}(\overline{K}/K)$\newnot{symbol:Gal}. Let $T$\newnot{symbol:T} be a free $\oo$-module of finite rank, endowed with a continuous $G_{\QQ}$-action which is unramified outside finitely many primes. Let $\FF$  denote a Selmer structure on $T$ and $\Sigma(\FF)$ a finite set of places of $\QQ$ that contains the infinite place, the prime $p$ and all primes where $T$ is ramified (see Definition 2.1.1 of~\cite{mr02} for a definition of a \emph{Selmer structure}) and let $\PP$ denote a set of rational primes disjoint from $\Sigma(\FF)$ which we will occasionally refer to as \emph{Kolyvagin primes}.  Let also $\chi(T,\FF)$\newnot{symbol:coreselmer} denote the core Selmer rank of the Selmer structure $\FF$ on $T$, as in \cite[Definition 4.1.11]{mr02}.

The following result is proved in \cite{mr02} (Corollary 4.5.1 and Corollary 4.5.2) under suitable hypotheses:

\begin{ithm}[Howard, Mazur, Rubin]
Let $\mathbf{KS}(T/\mm^kT,\FF,\PP)$ be the module of Kolyvagin Systems over $\oo/\mm^k$, with $k \in \ZZ^+$. Then
\begin{itemize}
\item[(i)] If $\chi(T,\FF)=0$, then $\mathbf{KS}(T/\mm^kT,\FF,\PP)=0$.
\item[(ii)] If  $\chi(T,\FF)=1$, then $\mathbf{KS}(T/\mm^kT,\FF,\PP)$ is free of rank one as an $\oo/\mm^k$-module.
\item[(iii)] If $\chi(T,\FF)>1$, then for every $r$, the $\oo/\mm^k$-module $\mathbf{KS}(T/\mm^kT,\FF,\PP)$ contains a free submodule of rank $r$.
\end{itemize}

\end{ithm}

Therefore, we have an answer for when one could hope to apply Kolyvagin's machinery, at least for the residual representations $T/\mm^kT$. Furthermore, when $\chi(T,\FF)=1$, one may go further and use the result above to obtain Kolyvagin systems for the representation $T$ itself (and not only for its quotients $T/\mm^kT$). These Kolyvagin systems proved to exist essentially have the same use as the derivative classes obtained from an Euler  System.

Mazur and Rubin also show ({c.f.}, \cite[Theorem 5.3.3]{mr02}) how to obtain Kolyvagin Systems over the cyclotomic Iwasawa algebra starting from an Euler System (in the sense of~\cite{r00}). Using these Kolyvagin systems over the cyclotomic tower, one \emph{expects} to prove one divisibility in main conjectures of Iwasawa Theory\footnote{These Kolyvagin system which we also prove to exist in this article without assuming the existence of Euler systems, have a priori no relation with the $L$-values. Conjecturally, the classes that we prove to exist here should be related to the relevant $p$-adic $L$-function via an independent recipe due in full generality to Perrin-Riou~\cite{pr}. See Section~\ref{applications} below for further details.}. Once again, one may wonder if there should exist Kolyvagin Systems over an Iwasawa algebra in general. This is what we answer in this paper. Before we describe our main results, we introduce some more notation.

Let $T^*=\textup{Hom}(T,\Phi/\oo(1))$\newnot{symbol:T*}  be the Cartier dual of $T$ and let $A=T\otimes_{\oo}\Phi/\oo$\newnot{symbol:A}. Let also $\mathcal{I}_{\ell}$ denote a fixed inertia subgroup of $G_{\QQ}$ at $\ell$. Let $\QQ_{\infty}$\newnot{symbol:qinfty} be the cyclotomic $\ZZ_p$-extension of $\QQ$, and $\Gamma:=\Gal(\QQ_\infty/\QQ)$ be its Galois group. As usual, $\LL:=\oo[[\Gamma]]$ is the cyclotomic Iwasawa algebra.

 Let $\FFc$ denote the canonical Selmer structure on $T\otimes \Lambda$ (see Section~\ref{prelim} below), and let $\chi(T)=\chi(T,\FFc)$ denote the core Selmer rank of the canonical Selmer structure $\FFc$ on $T$ (see \cite[Definition 4.1.11 and Theorem 5.2.15]{mr02} for a definition).  Let $\overline{\mathbf{KS}}(T\otimes\Lambda\,,\FFc,\PP)$ denote the $\LL$-module of Kolyvagin systems for $\FFc$ on $T\otimes \LL$ (defined as in \S\ref{KS(Lambda)}; see also Remark~\ref{comparison-KS} for a comparison of this definition to \cite[Definitions 3.1.3 and 3.1.6]{mr02}). Define also $\overline{\mathbf{KS}}(T,\FFc,\PP)$ to be the $\oo$-module of Kolyvagin systems for $\FFc$ on $T$ (which has been extensively  studied in~\cite[\S5.2]{mr02}, see particularly~\cite[Theorem 5.2.12]{mr02}).

The main technical result of this paper (Theorem~\ref{main})  shows that the Kolyvagin Systems exist over the cyclotomic Iwasawa algebra under certain technical hypotheses:
\begin{thma}
\label{lifting}
Assume the hypotheses in \S \ref{subsec:hypo} and that $\chi(T)=1$. Then:
\begin{enumerate}
\item[\textbf{(i)}] The $\LL$-module $\overline{\mathbf{KS}}(T\otimes\Lambda,\FFc,\PP)$ is free of rank one.
\item[\textbf{(ii)}] The specialization map
$$\overline{\mathbf{KS}}(T\otimes\Lambda,\FFc,\PP) \lra \overline{\mathbf{KS}}(T,\FFc,\PP)$$
 is surjective.
\end{enumerate}
\end{thma}

Any element of the module  $\overline{\mathbf{KS}}(T\otimes\Lambda,\FFc,\PP)$ will be called a \emph{$\LL$-adic Kolyvagin system}.

Before we explain our strategy to prove the theorem above, we first discuss its consequences in several aspects. When $T=\oo(1)\otimes \rho^{-1}$, where $\rho$ is an even Dirichlet character, the Kolyvagin system that Theorem~A proves to exist comes from the cyclotomic unit Euler system (via the Euler systems to Kolyvagin systems map of Mazur and Rubin, {c.f.}, \cite[Theorem 5.3.3]{mr02}). Our theorem further shows that the Kolyvagin system which arises from the cyclotomic unit Euler system is the ``best possible'' in this setting. See Proposition~\ref{prop:cycloprimitive} below for more details.

Consider now an elliptic curve $E_{/\QQ}$  without CM such that
\begin{itemize}
\item[(E1)] $p$ is coprime to all its Tamagawa factors,
\item[(E2)] $p$ is not anomalous for $E$ (in the sense of Mazur~\cite{mazur-anom}),
\end{itemize}
and let $T=T_p(E)$ be the $p$-adic Tate module.  In this case,  the $\LL$-adic Kolyvagin system Theorem~A proves to exist comes from Kato's~\cite{ka1} Euler system (again via the Euler systems to Kolyvagin systems map of Mazur and Rubin). And again our theorem shows (assuming the Birch and Swinnerton-Dyer conjecture) that the Kolyvagin system which arises from Kato's Euler system is the ``best possible'' in this setting. See Proposition~\ref{prop:katoprimitive} below for further details.

Besides these now-classical examples mentioned above, we also discuss several other arithmetic applications of Theorem~A. As the first application, we study in \S\ref{p-adicL-func} the relation of $\LL$-adic Kolyvagin systems with $p$-adic $L$-functions, using the rigidity offered by our main result. 
Rubin~\cite[\S8]{r00} (see also \S\ref{p-adicL-func} below) constructs an Euler system starting from Perrin-Riou's~\cite{pr2} conjectural $p$-adic $L$-functions. This conjectural Euler system gives rise to a $\Lambda$-adic Kolyvagin system using the Euler system to Kolyvagin system map of Mazur and Rubin \cite[Theorem 5.3.3]{mr02}  and the first part of Theorem~A implies  that the cohomology classes that we prove to exist in this article should be related (up to multiplication by an element of $\Lambda$) to the cohomology classes that Perrin-Riou~\cite[\S4.4]{pr2} and Rubin~\cite[Conjecture 8.2.6]{r00} predict to exist. Thus, $\LL$-adic Kolyvagin systems we construct here relate to Perrin-Riou's conjectural $p$-adic $L$-functions. On the other hand, one may also think of Theorem~A as an evidence (albeit quite weak) for the conjectures of Perrin-Riou  and Rubin; as it proves a consequence of these conjectures, namely the existence of $\LL$-adic Kolyvagin systems.

As a further application, we study in \S\ref{sec:stark-iwasawa} the Iwasawa theory of Rubin-Stark elements from the perspective offered by Theorem~A. Let $F$ be a totally real number field, and $r=[F:\QQ]$. Let $\rho$ be a totally even character of $G_F$ (i.e., it is trivial on all complex conjugations inside $G_F$) into $\oo^\times$ that has finite prime-to-$p$ order, and let $f_{\rho}$ be its conductor. We assume that $(p,f_{\rho})=1$ and $\rho(\wp)\neq1$ for any prime $\wp$ of $F$ lying over $p$. Assume also for notational simplicity that $p$ is unramified in $F/\QQ$. Let $T=\oo(1)\otimes\rho^{-1}$, but now considered as a $G_F$-representation. The main difficulty in this setting is that the core Selmer rank of the canonical Selmer structure is greater than one. We show first how to overcome this difficulty by constructing an auxiliary  Selmer structure (denoted by $\FF_{\LLL}$ in \S\ref{modified selmer}), by modifying the local conditions at $p$. For the construction of $\FF_{\LLL}$, we utilize a structure theorem for the semi-local cohomology group $H^1(F_p,T\otimes\LL)$ due to Perrin-Riou (see Appendix A). Once the auxiliary  Selmer structure is constructed, we apply Theorem~A to show the existence of $\LL$-adic Kolyvagin systems for these modified Selmer structures. On the other hand, Rubin~\cite[\S6]{ru96} shows that the Rubin-Stark elements give rise to an Euler system in this setting. If we apply Mazur and Rubin's {Euler systems to Kolyvagin systems map}, we \emph{only} obtain a $\LL$-adic Kolyvagin system for the canonical Selmer structure $\FFc$ on $T\otimes\LL$, which is coarser than our modified Selmer structure $\FF_{\LLL}$; and the techniques of~\cite{mr02} do not apply. We discuss how to obtain a  $\LL$-adic Kolyvagin system
 for the modified Selmer structure $\FF_{\LLL}$ starting with the Rubin-Stark elements, based on the methods of~\cite{kbbstark}, and how to deduce the main conjectures using these $\LL$-adic Kolyvagin systems.  In a forthcoming work~\cite{kbb-iwasawa}, we extend the methods of~\cite{kbbstark} to verify all the speculative claims in \S\ref{sec:stark-iwasawa}.

In a forthcoming work~\cite{kbb-stick}, the author gives another important arithmetic application of the first part of Theorem A (i.e., the rigidity of $\LL$-adic Kolyvagin systems) which we do not present here: He describes a relation between the Stickelberger elements and the Rubin-Stark elements for totally real fields. The existence of the Rubin-Stark elements is conjectural, whereas the existence of the Euler system of Stickelberger elements constructed in~\cite{kbb-stick} (see also~\cite{kurihara} for a construction of an Euler system using Stickelberger elements in a different way) rely only on a special case of Brumer's conjecture, which is proved in a variety of cases (c.f., \cite{wiles-brumer, kurihara, greither}) from Wiles' proof\footnote{The author has been warned that there might be a gap in Wiles' proof of Brumer's conjecture and his proof of the main conjecture for totally real fields in~\cite{wiles-mainconj}. He was also told that Buzzard and Taylor are able to mend this problem by making use of the recent developments in the theory of $p$-adic Hilbert modular forms.} of the main conjectures~\cite{wiles-mainconj}.

We would also like to point out that the rigidity phenomenon for Kolyvagin systems plays a central role in a recent work of Mazur and Rubin~\cite{mrdarmon}, where they prove an important portion of Darmon's conjecture~\cite{darmonconj}.

There are three main steps in proving Theorem~A. Let us give a brief outline of these in this paragraph. Let $\gamma$ be a topological generator of $\Gamma$ and let $R_{k,m}$\newnot{symbol:rkm} be the artinian ring $\LL/(\mm^k,(\gamma-1)^k)$ with $k,m \in \ZZ^+$. The main idea of the current article is to consider ``non-arithmetic specializations" (as Ochiai~\cite{ochiai-nearly-ordinary} also calls them) $T_{k,m}:=T\otimes R_{k,m}$ of the \emph{big} Galois representation $T\otimes\LL$. The first step is to choose (see \S\ref{primes}) a collection of \emph{Kolyvagin primes} $\PP_{k,m}$ for each $G_\QQ$-representation $T_{k,m}$. The second step is the construction (see \S\ref{subsec:core}) of what Mazur and Rubin call the \textit{core vertices} for the canonical Selmer structure $\FF_{\textup{can}}$ on $T_{k,m}$. This construction builds on the ideas developed in~\cite[\S4.1]{mr02}, however, the situation in this paper is technically more involved as one has to deal with two-dimensional artinian coefficient rings $R_{k,m}$, whereas Mazur and Rubin give a construction of core vertices for principal artinian coefficient rings. To achieve this, we use an Iwasawa theoretic control mechanism in order to give an upper bound on the size of certain Selmer groups (see Lemma~\ref{lem:applycart} and Proposition~\ref{lem:upper}) and an argument based on global duality to give a lower bound on the size of these Selmer groups (see Proposition~\ref{lem:lower}). 
In the final step, we prove (see \S\ref{KS1}) mostly following Howard's arguments in~\cite[Appendix B]{mr02} that there is a canonical isomorphism from the $R_{k,m}$-module  $\mathbf{KS}(T_{k,m},\FFc , \PP_{k,m})$  onto the module $H^1 _{\FFc(n)}(\QQ, T_{k,m})$, for any core vertex $n$. Furthermore, these isomorphisms are compatible with the base change maps $T_{k,m} \ra T_{k^{\prime}, m^{\prime}}$, which allows us to patch them together and construct (in \S\ref{KS(Lambda)}) the sought after classes inside $H^1 (\QQ, T\otimes\Lambda)$.

Theorem~A, particularly the second part, can be interpreted as a deformation theoretical result: We prove that a Kolyvagin system for the $\oo$-representation $T$ can be deformed to a $\LL$-adic Kolyvagin system for the \emph{cyclotomic deformation} $T\otimes\LL$ of $T$. Although we restrict our attention to the case of cyclotomic deformations of Galois representations  in this paper, one hopes that our formalism would apply in a similar way for other deformations of $T$. 
For example, consider an elliptic curve $E$ defined over $\QQ$ and let
$T=T_p(E)$ be its $p$-adic Tate module. In this setting, Kato has
constructed an Euler system, which gives rise to a Kolyvagin system
for the cyclotomic deformation of $T$. Ochiai~\cite{ochiai-nearly-ordinary} proved that this
Euler system may be further deformed to a \emph{big} Euler system  for
the nearly ordinary deformation of $T$ under certain technical
hypotheses (most notable of which is the assumption that the
deformation ring in question is a power series ring). An appropriate
generalization of Theorem A (and the work of Ochiai) should say
similarly that for a general Galois representation $T$, the Kolyvagin
systems which Mazur and Rubin~\cite{mr02} prove to exist may be deformed to
Kolyvagin systems for various types of deformations of $T$.

\section{Core Vertices over $\mathbf{\Lambda}$}
\label{sec:Core}

\subsection{Preliminaries}
\label{prelim}
Let $T$ be a free $\oo$-module of finite rank, on which the absolute Galois group $G_{\QQ}$ acts continuously, and let $V=T \otimes \Phi$\newnot{symbol:V}, and $A=V/T$.
We will denote $G_{\QQ_{\ell}}$ by $\mathcal{D}_{\ell}$\newnot{symbol:Dec} whenever we wish to identify this group with a closed subgroup of $G_{\QQ}$; namely with a particular decomposition group at $\ell$. We further define $\mathcal{I}_{\ell} \subset \mathcal{D}_{\ell}$\newnot{symbol:Inert} to be the inertia group and $\textup{Fr}_{\ell} \in \mathcal{D}_{\ell}/\mathcal{I}_{\ell}$\newnot{symbol:frob} to be the arithmetic Frobenius element  at $\ell$. We also write $\QQ_{\ell}^{\textup{unr}} \newnot{symbol:maxunr}\subset \overline{\QQ}_{\ell}$ for the maximal unramified subfield of $\overline{\QQ}_{\ell}$.

Let $\pmb{\mu}_{p^n} \subset \overline{\QQ}$ be the $p^n$-th roots of unity  and $\pmb{\mu}_{p^{\infty}}=\varinjlim_n \pmb{\mu}_{p^n}$. For any ring $R$, an ideal $\mathfrak{a}\subset  R$ and any $R$-module $M$, write $M[\mathfrak{a}]$ for the submodule killed by the elements of $\mathfrak{a}$.

Let $\LL := \oo[[\Gamma]]$\newnot{symbol:LL}, with $\Gamma= \hbox{Gal}(\QQ _{\infty}/ \QQ)$\newnot{symbol:Gamma} and $\QQ_\infty$ the cyclotomic $\ZZ_p$-extension of $\QQ$. Fixing a topological generator $\gamma$ of $\Gamma$, we may identify  $\Lambda$ with the power series ring $\oo[[\gamma-1]]$ in one variable over $\oo$, and we will occasionally do so, denoting $\gamma-1$ by $\mathbf{X}$. We will consider modules $T_f := T\otimes \Lambda / (f)$\newnot{symbol:Tf} and $T_{k,f} := T \otimes \Lambda / (\mm^k,f)$\newnot{symbol:Tkf}, where $k$ is a positive integer and $f$ is an element of $\LL$ such that the quotient $\LL/(f)$ is free of non-zero finite rank over $\oo$. Such an element $f\in\LL$ will be called a \emph{distinguished power series}. We will let $G_{\QQ}$ act on the tensor products above via acting on both factors.  For $\ell \neq p$, the inertia subgroup $\mathcal{I}_{\ell}$ acts trivially on $\Lambda$, hence $T_f$ (resp., $T_{k,f}$) is ramified at $\ell$ if and only if $T$ (resp., $T/\mm^kT$) is ramified at $\ell$.

We recall a definition from~\cite[\S2]{mr02}.
\begin{define}
\label{selmer structure}
Let $M$ be any $\oo[[G_{\QQ}]]$-module. A \emph{Selmer structure} $\FF$\newnot{symbol:selmerstr} on $M$ is a collection of the following data:
\begin{itemize}
\item A finite set $\Sigma(\FF)$ of places of $\QQ$, including $\infty$, $p$, and all primes where $M$ is ramified.
\item For every $\ell \in \Sigma(\FF)$, a local condition on $M$ (which we now view as a $\ZZ_p[[\mathcal{D}_{\ell}]]$-module), i.e., a choice of an $\oo$-submodule $H^1_{\FF}(\QQ_{\ell},M) \subset H^1(\QQ_{\ell},M).$
 \end{itemize}
\end{define}

\begin{define}\label{selmer triple}
A \emph{Selmer triple} is a triple $(T,\FF,\PP)$ where $\FF$ is a Selmer structure on $T$ and $\PP$ is a set of rational primes, disjoint from $\Sigma(\FF)$.
\end{define}
Define $\FFc$\newnot{symbol:FFc}, the \emph{canonical Selmer structure}, on $T \otimes \Lambda /(f)$ as follows:
\begin{itemize}
\item  $\Sigma(\FFc)=\{\ell: T \hbox{ is ramified at } \ell\} \cup \{p,\infty\}$.

\item $$H ^{1} _{\FFc}(\QQ _{\ell}, T _f):= \left \{
\begin{array}{ccl}
	 H ^{1} (\QQ _{p} \,, T_f)& ,& \hbox{if } \ell = p,    \\
          H ^{1}_{f}(\QQ _{\ell},T_f)& ,& \hbox{if } \ell \in \Sigma(\FFc)-\{p,\infty\}.
\end{array}
\right.$$
\end{itemize}
Here $H ^{1}_{f} (\QQ _{\ell}, T_f)\newnot{symbol:finite}:= \ker\left\{ H^{1} (\QQ _{\ell} , T_f) \lra H^{1} _{\unr} (\QQ _{\ell} , T _f \otimes \Phi)\right \}$, where, for any $\oo[[G_{\QQ_{\ell}}]]$-module $M$, we define
 $$H^1_{\textup{unr}}(\QQ_{\ell},M)\newnot{symbol:unr}:=\ker\left\{H^1(\QQ_{\ell},M) \lra H^1(\QQ_{\ell}^{\textup{unr}},M)\right)=H^1(\QQ_{\ell}^{\textup{unr}}/\QQ_{\ell},M^{\mathcal{I}_{\ell}}\}$$
  as in~\cite[Definitions 1.1.6 and 3.2.1]{mr02}.

We denote the Selmer structure on the quotients $T _{k,f}$ obtained by \emph{propagating} the local conditions given by $\FFc$ on $T_f$ to $T _{k,f}$ also by $\FFc$. See~\cite[Example 1.1.2]{mr02} for a definition of the \emph{propagation} of local conditions.

The main objective of this Section is to prove, under the assumptions listed below, the existence of \emph{core vertices}
for the Selmer structure $\FFc$ on $T_{k,m} := T _{k,\textbf{X} ^m}$ \newnot{symbol:Tkm} (see Theorem~\ref{thm1}).

\begin{define}
\label{my-transverse}
Let $\pmb{\mu}_\ell$ denote group $\ell$-th roots of unity inside $\overline{\QQ}_{\ell}$. For a $\oo[[\mathcal{D}_\ell]]$-module $M$ which is finite over $\oo$, the submodule
$$H^1_{\textup{tr}}(\QQ_{\ell},M\newnot{symbol:trans}):=H^1\left(\QQ_{\ell}(\pmb{\mu}_{\ell})/\QQ_{\ell},H^0(\QQ(\pmb{\mu}_{\ell}),M)\right) \subset H^1(\QQ_{\ell},M)$$
chosen as the local condition is called the \emph{transverse condition} on $M$ at $\ell$.
\end{define}

We will often make use of the following notation from~\cite{mr02}:
\begin{define}
\label{modified selmer-1}
Let $a,b,c \in \ZZ^+$ be pairwise coprime, and assume that $c$ is not divisible by any prime in $\Sigma(\FF)$. Then we write $\FF^a_b(c)$\newnot{symbol:modsel} for the Selmer structure on a $\oo[[G_{\QQ}]]$-module $M$ given by
\begin{itemize}
\item $\Sigma(\FF^a_b(c))=\Sigma(\FF) \cup \{\ell: \ell \mid abc\}$
\item $H^{1} _{\FF^a_b(c)}(\QQ _{\ell}, M)= \left \{
\begin{array}{ccl}
	H^{1}_{\FF}(\QQ _{\ell},M)& ,& \hbox{if } \ell \in   \Sigma(\FF) \hbox{ and } \ell\nmid ab \\
	H ^{1}(\QQ _{\ell},M) & ,& \hbox{if } \ell \mid a \\
   0 & ,& \hbox{if } \ell \mid b \\
   H^ {1} _{\textup{tr}}(\QQ_{\ell}, M) & ,& \hbox{if } \ell \mid c
\end{array}
\right.$
\end{itemize}
Whenever any of $a,b,c$ equals 1, we will suppress it from the notation.
\end{define}

\subsection{Hypotheses on $T$}
\label{subsec:hypo}
We will be assuming all the hypotheses introduced in \cite[\S3.5]{mr02}, except for $\mathbf{H.5}$ and $\mathbf{H.6}$:
\begin{itemize}
\item[($\mathbf{H.1}$)] $T/\mm T$ is an absolutely irreducible $\mathbb{F} [[G_{\QQ}]]$-representation.
\item[($\mathbf{H.2}$)] There is a $\tau \in G_{\QQ}$ such that $\tau=1 \hbox{ on } \pmb{\mu} _{p ^{\infty}}$ and the $\oo$-module $T/(\tau -1)T$ is free of rank one.
\item[($\mathbf{H.3}$)] $H^{1}(\QQ (T , \pmb{\mu}_{p ^{\infty}})/\QQ, T/\mm T)=H^{1}(\QQ  (T  , \pmb{\mu} _{p ^{\infty}})/\QQ, T ^{*}[\mm]) = 0$.

Here, $\QQ(T)$ is the smallest extension of $\QQ$ such that the $G_{\QQ}$-action on $T$ factors through $\Gal(\QQ(T)/\QQ)$ and $\QQ(T,\pmb{\mu}_{p^\infty}) = \QQ(T)(\pmb{\mu}_{p^\infty})$.
\item[($\mathbf{H.4}$)] Either $\textup{Hom}_{\mathbb{F}[[G_{\QQ}]]}(T/\mm T, T ^{*}[\mm])=0$, or
 $p>4$.
\end{itemize}

In addition to the hypothesis $\mathbf{H.1}$-$\mathbf{H.4}$, we will need the following assumptions for the main results of this paper:
\begin{itemize}
\item[($\mathbf{\mathbb{H}.T}$)] \emph{Tamagawa Condition}: The $\oo$-module $A^{\mathcal{I}_{\ell}}$ is divisible for every $\ell \neq p$.
\item[($\mathbf{\mathbb{H}.sEZ}$)] \emph{Strong Exceptional Zero-like Condition}: $H^0(\QQ_p,T^*)=0$.
\end{itemize}

We will also consider the following weaker version of $\mathbf{\mathbb{H}.sEZ}$:
\begin{itemize}
\item[($\mathbf{\mathbb{H}.EZ}$)] \emph{Exceptional Zero-like Condition}: $H^0(\QQ_p,T^*)$ is finite.
\end{itemize}

We will later choose a set of rational primes $\mathcal{P} _{k,m}$ (in \S\ref{primes} below) and verify that the analogues of hypotheses $\mathbf{H.5}$ and $\mathbf{H.6}$ of \cite[\S3.5]{mr02} hold for the collection $\{\PP _{k,m}\}$ and for the Selmer structures $\FFc(n)$, for every $n\in \mathcal{N}_{k,m} := \{\hbox{square free products of primes in } \mathcal{P}_{k,m}\}.$

\subsection{Cartesian Property for $\FFc$ on $T\otimes\LL$}
\label{subsec:cartesian}
Recall that an element $f \in \LL$ is called a distinguished power series if $\LL/(f)$ is a free $\oo$-module of finite non-zero rank. Let $\mathcal{T}=\{T_{k,f}\}$ denote the collection of quotients of $T\otimes \Lambda$, where $f$ ranges over distinguished power series and $k \in \ZZ^+ \cup \{+\infty\}$ (with the convention that $T_{k,f}=T_f$ when $k=+\infty$).

The main theorems of this paper will only concern the sub-collection
$$\mathcal{T}_0=\{T_{k,\xx^m}:k \in \ZZ^+\cup\{\infty\},m \in \ZZ^+\} \subset \mathcal{T},$$
 yet we still give the proofs of certain auxiliary facts in greater generality for the sake of clarity.
 \begin{define}
\label{iwasawa-cartesian}
A local condition $\FF$ is \emph{cartesian} on $\mathcal{T}$ if it satisfies the following conditions for every prime:
\begin{enumerate}
\item[\textbf{(C1)}]  (Functoriality) If $f$ and $g$ are distinguished polynomials and $k \leq k^{\prime}$ are positive integers then
\begin{itemize}
\item[\textbf{(C1.a)}] $H^1_{\FF}(\QQ_{\ell},T_{k,f})$ is the image of $H^1_{\FF}(\QQ_{\ell},T_{k^{\prime},f})$ under the canonical map $$H^1(\QQ_{\ell},T_{k^{\prime},f}) \lra H^1(\QQ_{\ell},T_{k,f}), \hbox{ and}$$
\item[\textbf{(C1.b)}] $H^1_{\FF}(\QQ_{\ell},T_{k,f})$ is the image of $H^1_{\FF}(\QQ_{\ell},T_{k,fg})$ under the canonical map $$H^1(\QQ_{\ell},T_{k,fg}) \lra H^1(\QQ_{\ell},T_{k,f}).$$
\end{itemize}
\item[\textbf{(C2)}] (Cartesian condition over power series) If $f$ and $g$ are distinguished power series and $k \in \ZZ^+$, then $H^1_{\FF}(\QQ_{\ell},T_{k,f})$ is the inverse image of $H^1_{\FF}(\QQ_{\ell},T_{k,fg})$ under the natural map $$H^1(\QQ_{\ell},T_{k,f}) \lra H^1(\QQ_{\ell},T_{k,fg}),$$ which is induced from the injection $T_{k,f}\stackrel{[g]}{\ra}T_{k,fg}$, where $[g]$ stands for multiplication by $g$.
\item[\textbf{(C3)}] (Cartesian condition over powers of $p$) If $f$ is a distinguished power series and $k \leq k^{\prime}$ are positive integers, then $H^1_{\FF}(\QQ_{\ell},T_{k,f})$ is the inverse image of $H^1_{\FF}(\QQ_{\ell},T_{k^{\prime},f})$ under the natural map $H^1_{\FF}(\QQ_{\ell},T_{k,f}) \lra H^1_{\FF}(\QQ_{\ell},T_{k^{\prime},f})$, which is induced from the injection  $T_{k,f} \stackrel{[p^{k^{\prime}-k}]}{\lra} T_{k^{\prime},f}$, where $[p^{k^{\prime}-k}]$ is the multiplication by $p^{k^{\prime}-k}$.
\end{enumerate}
\end{define}

Condition \textbf{C1} may sometimes be replaced by the weaker:
\begin{enumerate}
\item[\textbf{(C1$^{\prime}$)}]  (weak Functoriality) If $f$ and $g$ are distinguished power series and $k \leq k^{\prime}$ are positive integers then
\begin{itemize}
\item[\textbf{(C1$^{\prime}$.a)}] $H^1_{\FF}(\QQ_{\ell},T_{k,f})$ lies inside the image of $H^1_{\FF}(\QQ_{\ell},T_{k^{\prime},f})$ under the canonical map $$H^1(\QQ_{\ell},T_{k^{\prime},f}) \lra H^1(\QQ_{\ell},T_{k,f}),$$
\item[\textbf{(C1$^{\prime}$.b)}] $H^1_{\FF}(\QQ_{\ell},T_{k,f})$ lies inside the image of  $H^1_{\FF}(\QQ_{\ell},T_{k,fg})$ under the canonical map $$H^1(\QQ_{\ell},T_{k,fg}) \lra H^1(\QQ_{\ell},T_{k,f}).$$
\end{itemize}
\end{enumerate}
which is sufficient for most of our purposes. When a Selmer structure $\FF$ satisfies \textbf{C1{$^\prime$}}, \textbf{C2} and  \textbf{C3} we still say $\FF$ is \emph{weakly cartesian} on $\TT$.

In this section, we will check the cartesian properties for $\FFc$ and for $\FFc(n)$ on certain sub-collections of $\mathcal{T}$. We remark that the local condition $\FFc$ on $T_{k,f}$ is obtained by propagating $\FFc$ defined on $T_f$. Hence  the condition \textbf{C1.a} of Definition~\ref{iwasawa-cartesian} is automatically satisfied for $\FFc$.

Since we assume $p>2$, we have $H^1(\mathbb{R},T)=0$, hence the local conditions at the infinite place will be forced to be \emph{zero}. We therefore trivially have the cartesian property at $\infty$, regardless of the choice of a Selmer structure.

\subsubsection{Cartesian property at $\ell \neq p$}
Throughout this section, $\ell$ denotes a rational prime other than $p$ and we assume $\mathbf{\mathbb{H}.T}$. Note that if $\ell \notin \Sigma(\FFc)$, then $A^{\mathcal{I}_{\ell}}=A$ and therefore $A^{\mathcal{I}_{\ell}}$ is divisible. Hence, $\mathbf{\mathbb{H}.T}$ is satisfied for primes $\ell \notin \Sigma(\FFc)$.

\begin{lemma}
\label{all-divisible}For any distinguished power series  $f \in \Lambda$,  the $\oo$-module $A_f^{\mathcal{I}_{\ell}}$ is divisible.

\end{lemma}
\begin{proof}
Since $\ell \neq p$, the inertia group $\mathcal{I}_{\ell}$ acts trivially on $\Lambda/(f)$, hence $A_f^{\mathcal{I}_{\ell}}=A^{\mathcal{I}_{\ell}}\otimes \Lambda/(f)$.
\end{proof}
Let $V_f$ denote the $\Phi$-vector space $T_f \otimes \Phi$, and $A_f$ denote $V_f/T_f$. Note that there is a natural injection $T_{k,f} \hookrightarrow A_f$; in fact, we may identify the image of $T_{k,f}$ under this injection with the $\mm^k$-torsion subgroup $A_f[\mm^k]\subset A_f$.

Recall the definition of the unramified local cohomology group
$$H^1_{\textup{unr}}(\QQ_{\ell},V_f):=\ker\{H^1(\QQ_{\ell},V_f) \lra H^1(\mathcal{I}_{\ell},V_f)\}\cong H^1(\QQ_{\ell}^{\unr}/\QQ_{\ell},V_f^{\mathcal{I}_{\ell}}).$$
We define $H^1_{{f}}(\QQ_{\ell},A_{f})$ as the image of $H^1_{\textup{unr}}(\QQ_{\ell},V_f)$ under the natural map $$H^1(\QQ_{\ell},V_f) \lra H^1(\QQ_{\ell},A_f).$$ We also define $H^1_{{f}}(\QQ_{\ell},T_{k,f})$ as the inverse image of $H^1_{{f}}(\QQ_{\ell},A_{f})$ under the map $H^1(\QQ_{\ell},T_{k,f}) \ra H^1(\QQ_{\ell},A_f)$, which is induced from the injection $T_{k,f} \hookrightarrow A_f$.

\begin{lemma}\cite[Lemma 1.3.8(i)]{r00}
\label{compare-ffc-finite}
$H^1_{\FFc}(\QQ_{\ell},T_{k,f})=H^1_{{f}}(\QQ_{\ell},T_{k,f})$.
\end{lemma}

\begin{lemma} \cite[Lemma 1.3.5]{r00}
\label{compare-finite-unr-1}
For every $\ell \neq p$, the following sequences are exact:
\begin{enumerate}
\item[\textbf{(i)}] $0\lra H^{1}_{{f}}(\QQ_ {\ell},A_{f})\lra H^{1}_{\textup{unr}}(\QQ_ {\ell},A_{f})\lra \mathcal{W}_{f}/(\textup{Fr}_{\ell}-1)\mathcal{W}_{f}\lra 0$,
\item[\textbf{(ii)}] $
0\lra H^{1}_{\textup{unr}}(\QQ_ {\ell},T_{f})\lra H^{1}_{{f}}(\QQ_ {\ell},T_{f})\lra \mathcal{W}_{f}^{\textup{Fr}_{\ell}=1} \lra 0$.
\end{enumerate}
with $\mathcal{W}_{f}= A_f ^{\mathcal{I}_{\ell}}/(A_f ^{\mathcal{I}_{\ell}})_{\textup{div}}$ and $M_{\textup{div}}$ stands for the maximal divisible submodule of an $\oo$-module $M$.
\end{lemma}
As an immediate consequence of Lemma~\ref{compare-finite-unr-1}, we obtain:
\begin{cor}
\label{compare-finite-unr-2}
If $\mathbf{\mathbb{H}.T}$ holds, then  $H^{1}_{{f}}(\QQ_ {\ell},A_{f})= H^{1}_{\textup{unr}}(\QQ_ {\ell},A_{f})$ and $H^{1}_{\textup{unr}}(\QQ_ {\ell},T_{f})= H^{1}_{{f}}(\QQ_ {\ell},T_{f})$ for every $\ell \neq p$.
\end{cor}

\begin{cor}
\label{unr-functorial}
Assume  $\mathbf{\mathbb{H}.T}$ holds. Then  $H^{1}_{\textup{unr}}(\QQ_ {\ell},T_{k,f})$ is the image of $H^{1}_{\textup{unr}}(\QQ_ {\ell},T_{f})$ under the map $H^{1}(\QQ_ {\ell},T_{f}) \ra H^{1}(\QQ_ {\ell},T_{k,f})$ restricted to $H^{1}_{\textup{unr}}(\QQ_ {\ell},T_{f})$.
\end{cor}

\begin{proof} By Lemma~\ref{compare-ffc-finite}, Lemma~\ref{compare-finite-unr-1} and Corollary~\ref{compare-finite-unr-2} we have the following diagram:
$$\xymatrix @C=.15in@R=.12in{H^1_{{f}}(\QQ_{\ell},T_{k,f})\ar @{=}[d]\ar@{=}[r]&\textup{image}(H^1_{{f}}(\QQ_{\ell},T_f))\ar@{=}[r]&\textup{image}(H^1_{\textup{unr}}(\QQ_{\ell},T_f))\subset H^1_{\textup{unr}}(\QQ_{\ell},T_{k,f})\\
H^1_{{f}}(\QQ_{\ell},T_{k,f})\ar@{=}[r]&\iota^{-1}_{k,f}(H^1_{{f}}(\QQ_{\ell},A_f))\ar@{=}[r]&\iota^{-1}_{k,f}(H^1_{\textup{unr}}(\QQ_{\ell},A_f))\supset H^1_{\textup{unr}}(\QQ_{\ell},T_{k,f})
}$$ which shows that all the containments above are in fact equalities. This completes the proof.
\end{proof}
\begin{prop}
\label{prop:cart}
Assuming $\mathbf{\mathbb{H}.T}$, the canonical Selmer structure $\FFc$ is cartesian on $\mathcal{T}$ at $\ell \neq p$.
\end{prop}

\begin{proof}
As we have remarked earlier, \textbf{C1.a} is satisfied by the definition of $\FFc$ on $\mathcal{T}$.

 \textbf{C1.b} also follows from Lemma~\ref{compare-ffc-finite}, Lemma~\ref{compare-finite-unr-2} and Corollary~\ref{unr-functorial}.

We next check \textbf{C2}. Suppose $f$ and $g$ are any distinguished power series. Let $V_f$ and $A_f$ be as above. Then we have the following commutative diagram with exact rows:
\begin{equation}
\label{diag:1}
\vcenter{
\xymatrix@R=.2in{
0\ar[r] &H ^1 _{\FFc}(\QQ_{\ell} \,,T_{k,f})\ar[r]\ar[d]& H ^1 (\QQ_{\ell} \,,T_{k,f})\ar[r]\ar[d] & \frac{H ^1 (\QQ_{\ell} \,,\,A_f)}{H^{1}_{{f}}(\QQ_{\ell},\,A_f)}\ar[d]\\
0\ar[r] &H ^1 _{\FFc}(\QQ_{\ell} \,,T_{k,fg})\ar[r]& H ^1 (\QQ_{\ell} \,,T_{k,fg})\ar[r]& \frac{H ^1 (\QQ_{\ell} \,,\,A_{fg})}{H^{1}_{{f}}(\QQ_ {\ell} \,,\,A_{fg})}}
}
\end{equation}

The rows are exact by Lemma~\ref{compare-ffc-finite} and the vertical maps are all induced from the multiplication by $g$ map $[g]: \Lambda /(f) \ra \Lambda /(fg).$

Using the Hochschild-Serre spectral sequence and the fact that the cohomological dimension  of $\textup{Gal}(\QQ_{\ell}^{\textup{unr}}/\QQ_{\ell})\cong \hat{\ZZ}$ is one, we obtain the following diagram:
\begin{equation}
\label{diag:2}
\vcenter{
\xymatrix @C=.5in @R=.25in{
\frac{H ^1 (\QQ_{\ell} ,A_f)}{H^{1}(\QQ_ {\ell} ^{\textup{unr}}/\QQ_{\ell},A_f(\mathcal{I}_{\ell}))} \ar[d]\ar[r]^(.50)\sim& H^{1}(\QQ_ {\ell} ^{\unr},A_{f})^{\textup{Fr}_{\ell}=1}\ar @{=}[r]\ar[d]&H^{1}(\mathcal{I}_{\ell},A_{f})^{\textup{Fr}_{\ell}=1}\ar[d]\\
\frac{H ^1 (\QQ_{\ell} ,A_{fg})}{H^{1}(\QQ_ {\ell} ^{\textup{unr}}/\QQ_{\ell},A_{fg}(\mathcal{I}_{\ell}))}\ar[r]^(.50)\sim& H^{1}(\QQ_ {\ell} ^{\unr},A_{fg})^{\textup{Fr}_{\ell}=1}\ar @{=}[r]& H^{1}(\mathcal{I}_{\ell},A_{fg})^{\textup{ Fr}_{\ell}=1}
}}
\end{equation}
where $A_f(\mathcal{I}_{\ell}):=A_f^{\mathcal{I}_{\ell}}$. Further, the $\mathcal{I}_{\ell}$-cohomology of the exact sequence \be\label{exact seq-main} \xymatrix @R=.12in{
&A_f\ar @{=}[d]& A_{fg}\ar @{=}[d] & A_g \ar @{=}[d]\\
0 \ar[r] &A \otimes \Lambda/(f) \ar[r]^{[g]} &A \otimes \Lambda/(fg) \ar[r]& A \otimes \Lambda/(g) \ar[r] &0
}\ee
(which exists because of the division algorithm for distinguished polynomials, see \cite[\S7]{washington}) gives, using the proof of Lemma~\ref{all-divisible},
$$0 \lra A^{\mathcal{I}_{\ell}} \otimes \Lambda/(f) \stackrel{[g]}{\lra} A^{\mathcal{I}_{\ell}} \otimes \Lambda/(fg) \lra A^{\mathcal{I}_{\ell}} \otimes \Lambda/(g).$$
Since tensoring with $A ^{\mathcal{I}_{\ell}}$ is right exact, the very right map in the exact sequence above is in fact surjective, thus the following sequence is exact:
\be\label{diag:exact1}0 \lra A^{\mathcal{I}_{\ell}} \otimes \Lambda/(f) \stackrel{[g]}{\lra} A^{\mathcal{I}_{\ell}} \otimes \Lambda/(fg) \lra A^{\mathcal{I}_{\ell}} \otimes \Lambda/(g)\lra 0\ee
 Furthermore, the $\mathcal{I}_{\ell}$-cohomology of the exact sequence~(\ref{exact seq-main}) and the exact sequence~(\ref{diag:exact1}) give rise to the following diagram with exact rows:
$$\xymatrix @R=.15in{H^{0}(\mathcal{I}_{\ell}\,,A_{fg})\ar[r]\ar @{=}[d]&H^{0}(\mathcal{I}_{\ell}\,,A_{g})\ar[r]\ar @{=}[d]&H^{1}(\mathcal{I}_{\ell}\,,A_{f})\ar[r]&H^{1}(\mathcal{I}_{\ell}\,,A_{fg})\\
A ^{\mathcal{I}_{\ell}}\otimes \Lambda/(fg)\ar[r]&A^{\mathcal{I}_{\ell}}\otimes \Lambda/(g)\ar[r]&0&&
}$$
This shows that the map $H^{1}(\mathcal{I}_{\ell},A_{f})\ra H^{1}(\mathcal{I}_{\ell},A_{fg})$ is injective and therefore the map on the  right hand side in~(\ref{diag:2}) is also injective (as well as the map on the left hand side), hence
\begin{equation}
\label{eq:1}
H^{1}_{\unr}(\QQ_ {\ell} ,A_{f})=\ker\left(H^{1}(\QQ_ {\ell} ,A_{f}) \lra \frac{H^{1}(\QQ_ {\ell},A_{fg})}{H^{1}_{\unr}(\QQ_ {\ell},A_{fg})}\right).
\end{equation}
Since we assumed  $\mathbf{\mathbb{H}.T}$, it follows by Corollary~\ref{compare-finite-unr-2} that $H^{1}_{{f}}(\QQ_ {\ell},A_{f})=H^{1}_{\textup{unr}}(\QQ_ {\ell},A_{f})$ for every distinguished power series $f$. This in turn implies, using (\ref{eq:1}),
 $$H^{1}_{{f}}(\QQ_ {\ell},A_{f})=\ker\left(H^{1}(\QQ_ {\ell},A_{f}) \lra \frac{H^{1}(\QQ_ {\ell},A_{fg})}{H^{1}_{f}(\QQ_ {\ell},A_{fg})}\right).$$
  But this means that the rightmost map in (\ref{diag:1}) is injective, therefore $$H ^1 _{\FFc}(\QQ_{\ell},T_{k,f})= \ker\left(H ^1 (\QQ_{\ell},T_{k,f}) \lra \frac{H ^1 (\QQ_{\ell},T_{k,fg})}{H ^1 _{\FFc}(\QQ_{\ell},T_{k,fg})}\right)$$
   which is the property \textbf{C2} in Definition~\ref{iwasawa-cartesian}.

 To verify the property \textbf{C3} in Definition~\ref{iwasawa-cartesian}, note that by the definition of $H^1_{\FFc}(\QQ_{\ell},T_f)$, we have an injection $${H^1(\QQ_{\ell},T_f)/H^1_{\FFc}(\QQ_{\ell},T_f) \hookrightarrow  H^1(\QQ_{\ell},V_f)/H^1_{\FFc}(\QQ_{\ell},V_f)},$$ which shows that $H^1(\QQ_{\ell},T_f)/H^1_{\FFc}(\QQ_{\ell},T_f)$ is $\oo$-torsion free. Property \textbf{C3} now follows from \cite[Lemma 3.7.1(i)]{mr02}.
\end{proof}

\subsubsection{Cartesian property at $p$}
In this section we verify that  $\FFc$ is cartesian on the collection
$$\mathcal{T}_0 =\{T_{k,\xx^m}: k,m \in \ZZ^+ \} \subset \mathcal{T}.$$

We will write $T_{k,m}$ instead of $T_{k,\xx^m}$, and $T_m$ instead of $T_{\xx^m}$ for notational convenience. Define also $R_{k,m}:=\LL/(\mm^k,\xx^m)$ for every $k,m\in \ZZ^+$.

Throughout this section we assume that \hsez\, holds.

\begin{lemma}
\label{h2-vanishing}
$H^2(\QQ_p,T\otimes \Lambda)=0$.
\end{lemma}
\begin{proof}
By local duality, $H^0(\QQ_p,T^*)=0$ if and only if $H^2(\QQ_p,T)=0$. Since the cohomological dimension of $G_{\QQ_p}$ is 2 and since we assumed \hsez, it follows that
$$H^2(\QQ_p,T\otimes\Lambda)/(\gamma-1) \cong H^2(\QQ_p,T)=0.$$
 Proof of Lemma now follows by Nakayama's lemma.
\end{proof}

\begin{prop}
\label{cart-at-p}
The canonical Selmer structure $\FFc$ is cartesian on $\mathcal{T}_0$ (in the sense of Definition~\ref{iwasawa-cartesian}; but $\mathcal{T}$ replaced by $\mathcal{T}_0$) at the prime $p$.
\end{prop}

\begin{proof}
\textbf{C1.a} is satisfied by the definition of $\FFc$ on $\mathcal{T}_0$.

The exact sequence
$$H^1(\QQ_p,T\otimes\Lambda) \lra H^1(\QQ_p,T_m) \lra H^2(\QQ_p,T\otimes\Lambda)[(\gamma-1)^m]$$
 shows, by Lemma~\ref{h2-vanishing}, that $H^1_{\FFc}(\QQ_p,T_m)=H^1(\QQ_p,T_m)=\textup{image}(H^1(\QQ_p,T\otimes\Lambda)).$ Furthermore, the $G_{\QQ_p}$-cohomology of the exact sequence
 $$0 \lra T_m \stackrel{[\varpi^k]}{\lra}T_m \lra T_{k,m}\lra 0$$
  implies that $\textup{coker}\{H^1(\QQ_p,T_m) \ra H^1 (\QQ_p,T_{k,m})\}=H^2(\QQ_p,T_m)[\mm^k].$ Since the cohomological dimension of $G_{\QQ_p}$ is 2, it follows that
  $$H^2(\QQ_p,T_m)\cong H^2(\QQ_p,T\otimes\Lambda)/(\gamma-1)^mH^2(\QQ_p,T\otimes\Lambda),$$
  and this is \emph{zero} by Lemma~\ref{h2-vanishing}. We therefore see that
  \be \label{identify-at-p}H^1_{\FFc}(\QQ_p,T_{k,m}):=\textup{im}\{H^1(\QQ_p,T\otimes\Lambda) \ra H^1(\QQ_p,T_{k,m})\}=H^1(\QQ_p,T_{k,m}).\ee
  It now follows from (\ref{identify-at-p}) that $\FFc$ satisfies  \textbf{C1.b}, \textbf{C2} and \textbf{C3}.

\end{proof}

\begin{rem}
\label{hsez-necessary}
For $\FFc$ on $\TT_0$ to satisfy condition  \textbf{C2}, \hez\, alone is not sufficient; we indeed need to assume \hsez. This is what we explain in this paragraph: We prove that \hez\, together with \textbf{C2} for $\FFc$ imply \hsez.

We have the following diagram with exact rows:
$$\xymatrix@R=.19in{H^1(\QQ_p,T_1) \ar[r] \ar[d]^{\alpha} &H^1(\QQ_p,T_{k,1}) \ar[r]\ar[d]^{\beta}&H^2(\QQ_p,T_1)[\mm^k]\ar[d]^{\theta} \ar[r]&0\\
H^1(\QQ_p,T_{m+1}) \ar[r] &H^1(\QQ_p,T_{k,m+1}) \ar[r]&H^2(\QQ_p,T_{m+1})[\mm^k] \ar[r]&0
}$$
where $\alpha, \beta$ and $\theta$ are all induced from multiplication by $\xx^m$.
If \textbf{C2} is true, then this means $\theta$ is injective. Set $M=H^2(\QQ_p,T\otimes\Lambda)$. It is well known that $M$ is a finitely generated torsion $\Lambda$-module ({c.f.},~\cite[ Proposition 3.2.1]{pr}). Let $\textup{char}(M)$ denote the characteristic ideal of $M$. Then \hez\, is equivalent to saying that $\xx\nmid\textup{char}(M)$; which in return implies that $M/\xx^rM$ is finite for all $r \in \ZZ^+$.

On the other hand, since the cohomological dimension of $G_{\QQ_p}$ is 2,
$$H^2(\QQ_p,T_1) \cong M/\xx M\hbox{, and } H^2(\QQ_p,T_{m+1}) \cong M/\xx^{m+1}M,$$
 hence $\theta$ is given by
 $$\theta:\,\{M/\xx M\} [\mm^k] \lra \{M/\xx^{m+1}M\}[\mm^k].$$
 Further, since $M/\xx M$ is finite, one can choose $k$ large enough so that
 $$\{M/\xx M\} [\mm^k]=M/\xx M, \,\, \hbox{and  }\{M/\xx^{m+1}M\}[p^k]=M/\xx^{m+1}M.$$
For such $k$, the kernel of $\theta$ is $(M[\xx^m]+\xx M)/\xx M$, which is trivial since $\theta$ is injective. This means $M[\xx^m] \subset \xx M$ for all $m$. Let $M[\xx^{\infty}]:=\cup_{m \in \ZZ^+}M[\xx^m]$. Then $M[\xx^{\infty}] \subset \xx M$; which shows that $M[\xx^{\infty}]$ is in fact $\xx$-divisible. But since $M$ is finitely generated, this is impossible unless $M=0$. As in the proof of Lemma~\ref{h2-vanishing}, this is equivalent to \hsez.
\end{rem}

\subsection{Choosing a set of Kolyvagin primes}
\label{primes}
In this section, we choose the set of primes $\PP_{k,m}$ that  were mentioned in \S\ref{subsec:hypo}. Let $\tau$ be as in $\mathbf{H.2}$. Given  $k,m \in \ZZ^+$, we define
$$ \PP_{k,m}\newnot{symbol:pkm}:=\{\ell: \hbox{Fr}_{\ell} \hbox{ is conjugate to } \tau \hbox{ in } \textup{Gal}(\QQ(T/\mm^{k+m}T, \pmb{\mu}_{p^{k+m+1}})/\QQ)\},$$
 where $\QQ(T/\mm^{k+m}T, \pmb{\mu}_{p^{k+m+1}})$ is the fixed field in $\overline{\QQ}$ of the kernel of the map
 $$G_{\QQ} \lra \textup{Aut}(T/\mm^{k+m}T)\oplus\textup{Aut}(\pmb{\mu}_{p^{k+m+1}}).$$

Note that $\PP_{k,m}$ depends only on $k+m$; we may therefore define $\PP_{k+m}:=\PP_{k,m}$. We set $\PP:=\PP_2=\PP_{1,1}$\newnot{symbol:primes}. Note also that $\PP_j \subset \PP_i$ for all $j>i$.
\begin{rem}
\label{rem:l-cong-to-1}
If $\ell\in\PP_{k,m}$, then by our definition $\textup{Fr}_{\ell}$ is conjugate to $\tau$ in $\textup{Gal}(\QQ(\pmb{\mu}_{p^{k+m+1}})/\QQ)$,  hence we have $\ell \equiv 1 \,(\textup{mod} \, \mm^{k+m+1})$, since $\tau=1$ on $\pmb{\mu}_{p^\infty}$.
\end{rem}
\begin{lemma}
\label{lem:1}
For $\ell \in \PP _{k,m}$, the Frobenius  $\textup{Fr}_{\ell}$ acts trivially on $\Lambda/(\mm^k,\mathbf{X}^m)$.
\end{lemma}

\begin{proof}
Let $\gamma$ be the fixed topological generator of $\Gamma$, such that $\gamma -1$ corresponds to $\mathbf{X}$ under the identification of $\Lambda$ with $\oo[[\xx]]$.  The Frobenius $\textup{Fr}_{\ell}$ acts on $\Lambda$ through the natural surjection $\textup{Gal}(\overline{\QQ}/\QQ) \twoheadrightarrow \Gamma$. Let $\overline{\textup{Fr}}_{\ell}$ be the image of $\textup{Fr}_{\ell}$ under this map. Since $\textup{Fr}_{\ell}$ and $\tau$ are conjugate in the group $\textup{Gal}(\QQ(T/\mm^kT , \pmb{\mu} _{p^{k+m+1}})/\QQ)$ (so also in $\textup{Gal}(\QQ_{k+m}/\QQ)$, where $\QQ_{k+m}$ is the $(k+m)$-layer of $\QQ_\infty$), it follows from our assumption that $\tau=1$ on $\pmb{\mu}_{p^\infty}$ that $\overline{\textup{Fr}}_{\ell}=\gamma ^{\alpha}$, with $\alpha \in \ZZ_p$ and has $p$-adic valuation at least $k+m$. Note that
$$\gamma ^{\alpha}=(\gamma -1+1)^{\alpha}\equiv (\gamma -1)^{\alpha}+1 \equiv 1 \hbox{ mod }(\mm^k,(\gamma-1)^m),$$
 which completes the proof of Lemma.

\end{proof}
Recall that $R_{k,m}:=\LL/(\mm^k,\xx^m)$ for $k,m \in \ZZ^+$.
\begin{lemma}
For any $\ell \in \PP _{k,m}$, the $R_{k,m}$-module $T_{k,m}/(\textup{Fr}_{\ell}-1)T_{k,m}$ is free of rank one.
\end{lemma}
\begin{proof}
By Lemma~\ref{lem:1}, $(\textup{Fr}_{\ell}-1)T_{k,m}= \left[(\textup{Fr}_{\ell}-1)(T/\mm^kT)\right]\otimes \Lambda/(\xx^m)$. Since $(\hbox{Fr}_{\ell}-1)(T/\mm^kT)=(\mm^kT+(\textup{Fr}_{\ell}-1)T)/\mm^kT$, it follows that
\begin{equation}
\label{eqn:5}
(\textup{Fr}_{\ell}-1)T_{k,m}=(\mm^kT+(\textup{Fr}_{\ell}-1)T)/\mm^kT \otimes \Lambda/(\xx^m)
\end{equation}
Since $\Lambda/(\xx^m)$ is a free $\oo$-module, the functor $[-\otimes_{\oo}\Lambda/\xx^m]$ is exact. It therefore follows from (\ref{eqn:5}) that
\be
\label{eqn:6}
T_{k,m}/(\textup{Fr}_{\ell}-1)T_{k,m}= T/(\mm^kT+(\textup{Fr}_{\ell}-1)T) \otimes \Lambda/(\xx^m).
\ee
Since $\ell \in \PP_{k,m}$, the $\oo/\mm^k$-module $T/(\mm^kT+(\textup{Fr}_{\ell}-1)T)$ is free of rank one. Now using (\ref{eqn:6}), the proof of Lemma follows.

\end{proof}

Note that, by Remark~\ref{rem:l-cong-to-1} it follows that  $|\mathbb{F}_{\ell}^{\times}| \cdot T_{k,m}=(\ell-1)\cdot T_{k,m}=0,$ hence all the results proved in \cite[\S1.2]{mr02} hold with the choice $R=R_{k,m}$ and the free $R_{k,m}$-module $T_{k,m}$. We record these results for future reference:

\begin{lemma}\cite[Lemma 1.2.1]{mr02}
\label{lemma:fs-explicit}
For $\ell \in \PP_{k,m}$, there are canonical functorial isomorphisms
\begin{enumerate}
\item[\textbf{(i)}] $H^1 _{f}(\QQ_{\ell},T_{k,m})\cong T_{k,m}/(\Fr_{\ell}-1)T_{k,m}$,
\item[\textbf{(ii)}] $H^1 _{\textup{s}}(\QQ_{\ell},T_{k,m}) \otimes \mathbb{F}_{\ell}^{\times} \cong T_{k,m}^{\Fr_{\ell}=1}$
\end{enumerate}
\end{lemma}

\begin{lemma}\cite[Lemma1.2.3]{mr02}
\label{lem:ref}
For $\ell \in \PP_{k,m}$,  the finite-singular comparison map
$$\phi_{\ell}^{\textup{fs}}: H^1 _{f}(\QQ_{\ell},T_{k,m}) \lra H^1 _{\textup{s}}(\QQ_{\ell},T_{k,m}) \otimes \mathbb{F}_{\ell}^{\times}$$
(see \cite{mr02} Definition~1.2.2)  is an isomorphism. In particular, both $H^1 _{f}(\QQ_{\ell},T_{k,m})$ and $H^1_{\textup{s}}(\QQ_{\ell},T_{k,m})$ are free of rank one over $R_{k,m}$.
\end{lemma}

\subsection{The Transverse Condition}
\label{subsub:Trans}
In this section we study the properties of the \emph{transverse condition} (recall Definition~\ref{my-transverse} above) at primes $\ell \in \PP_{k,m}$. Compare the facts we record in this section to  Definition 1.1.6(iv), \S1.2 and Lemma 3.7.4 of \cite{mr02}.

As remarked earlier, if $\ell \in \PP_{k,m}$, then the results from \cite[\S1.2]{mr02} still apply. In particular:
\begin{lemma}
\label{lem:trans}
The transverse subgroup $H^1 _{\textup{tr}}(\QQ_{\ell},T_{k,m})\subset H^1(\QQ_{\ell},T_{k,m})$ projects isomorphically onto $H^1_{\textup{s}}(\QQ_{\ell},T_{k,m})$ in the exact sequence
$$\xymatrix{0\ar[r]&H^1_{f}(\QQ_{\ell},T_{k,m})\ar[r]&H^1 (\QQ_{\ell},T_{k,m})\ar[r]&H^1 _{\textup{s}}(\QQ_{\ell},T_{k,m})\ar[r]&0}$$
 In other words, the sequence above has a functorial splitting
 $$H^1 (\QQ_{\ell},T_{k,m})=H^1 _{f}(\QQ_{\ell},T_{k,m})\oplus H^1_{\textup{tr}}(\QQ_{\ell},T_{k,m}).$$
\end{lemma}

\subsection{Cartesian Property of $\FFc(\lowercase{n})$}
Recall $\mathcal{N}_{k,m}=\{\hbox{square free products of primes in } \mathcal{P}_{k,m}\}.$  In this section, we verify that the modified Selmer structure $\FFc(n)$ (given as in Definition~\ref{modified selmer-1}) is cartesian on the collection of $\Lambda$-modules $\TT_{k,m}:=\{T_{\alpha,\beta}\}_{\substack{ \alpha \leq k\\ \beta\leq m }}$, for every $n \in \mathcal{N}_{k,m}$.
\begin{prop}
\label{cart-tr} \textup{(Compare to \cite{mr02} Lemma 3.7.4)}
For every $\ell \in \PP_{k,m}$, the transverse condition at $\ell$ is cartesian  on the collection of quotients $\TT_{k,m}$.
\end{prop}

\begin{proof}
Let $n \leq N \leq m$ and $l \leq L \leq k$ be positive integers. For $\ell \in \PP_{k,m}$, we have the following commutative diagram by Lemma~\ref{lem:trans}:

$$
\vcenter{\xymatrix @C=.22in {0\ar[r]&H^1 _{\textup{tr}}(\QQ_{\ell},T_{L,N})\ar[r]&H^1 (\QQ_{\ell},T_{L,N})\ar[r]\ar[d]&H^1_{f}(\QQ_{\ell},T_{L,N})\ar[r]\ar[d]&0\\
0\ar[r]&H^1 _{\textup{tr}}(\QQ_{\ell},T_{l,n})\ar[r]&H^1 (\QQ_{\ell},T_{l,n})\ar[r]&H^1_{f}(\QQ_{\ell},T_{l,n})\ar[r]&0
}}
$$
where the vertical maps are induced by canonical surjection $T_{L,N} \twoheadrightarrow T_{l,n}.$ This diagram shows that
$H^1 _{\textup{tr}}(\QQ_{\ell},T_{L,N})=\ker\{H^1 (\QQ_{\ell},T_{L,N})\ra H^1_{f}(\QQ_{\ell},T_{L,N})\}$
 is mapped into
 $$H^1 _{\textup{tr}}(\QQ_{\ell},T_{l,n})=\ker\{H^1 (\QQ_{\ell},T_{l,n})\lra H^1_{f}(\QQ_{\ell},T_{l,n})\}$$
 under the map $H^1 (\QQ_{\ell},T_{L,N}) \ra H^1 (\QQ_{\ell},T_{l,n})$. This shows that the transverse condition satisfies \textbf{C1$^{\prime}$}. By Lemma~\ref{lemma:fs-explicit} and \ref{lem:ref}, we have that $T_{L,N}^{\textup{Fr}_{\ell}=1}$ (resp., $T_{l,n}^{\textup{Fr}_{\ell}=1}$)
 is a free $R_{L,N}$-module (resp., $R_{l,n}$-module) of rank one. Hence the functorial map below induced from $T_{L,N} \twoheadrightarrow T_{l,n}$ is surjective:
 $$T_{L,N}^{\textup{Fr}_{\ell}=1} \cong H^1 _{\textup{tr}}(\QQ_{\ell},T_{L,N}) \otimes\mathbb{F}_\ell^\times \lra H^1 _{\textup{tr}}(\QQ_{\ell},T_{l,n}) \otimes\mathbb{F}_\ell^\times \cong T_{l,n}^{\textup{Fr}_{\ell}=1}.$$
This shows that  \textbf{C1} holds as well.

By Lemma~\ref{lem:trans}, we have the following commutative diagram, where the isomorphisms on the right come from Lemma~\ref{lemma:fs-explicit}:
\be
\label{eqn:7}
\vcenter{\xymatrix @C=.22in {H^1 (\QQ_{\ell},T_{k,n})\ar@{->>}[r]\ar[d]^{[\xx^{N-n}]}&H^1_{f}(\QQ_{\ell},T_{k,n})\ar[r]^(.44){\sim}\ar[d]^{[\xx^{N-n}]}&T_{k,n}/(\Fr_{\ell}-1)T_{k,n} \ar[d]^{[\xx^{N-n}]}\\
H^1 (\QQ_{\ell},T_{k,N})\ar@{->>}[r]&H^1_{f}(\QQ_{\ell},T_{k,N})\ar[r]^(.44){\sim} &T_{k,N}/(\Fr_{\ell}-1)T_{k,N}
}}
\ee
Since $T_{k,n}/(\Fr_{\ell}-1)T_{k,n}$ (resp., $T_{k,N}/(\Fr_{\ell}-1)T_{k,N}$) is a free $R_{k,n}$-module (resp., $R_{k,N}$-module) of rank one by Lemma~\ref{lem:ref}, and since the map $[\xx^{N-n}]:R_{k,n} \ra R_{k,N}$ is injective, it follows that the the vertical map on the right is injective. So, if $c \in H^1 (\QQ_{\ell},T_{k,n})$ and $[\xx^{N-n}]$c projects to \emph{zero} in $H^1_{f}(\QQ_{\ell},T_{k,N})$, then $c$  projects to \emph{zero} in $H^1_{f}(\QQ_{\ell},T_{k,n})$, i.e.,
$$c \in \ker\{H^1(\QQ_{\ell},T_{k,n}) \lra H^1_{f}(\QQ_{\ell},T_{k,n})\}=H^1_{\textup{tr}}(\QQ_{\ell},T_{k,n})$$
 provided $[\xx^{N-n}]c \in \ker\{H^1(\QQ_{\ell},T_{k,N}) \ra H^1_{f}(\QQ_{\ell},T_{k,N})\}=H^1_{\textup{tr}}(\QQ_{\ell},T_{k,N}).$
This shows that
$$H^1_{\textup{tr}}(\QQ_{\ell},T_{k,n})=\ker\{H^1 (\QQ_{\ell},T_{k,n}) \stackrel{[\xx^{N-n}]}{\lra} H^1(\QQ_{\ell},T_{k,N})/H^1_{\textup{tr}}(\QQ_{\ell},T_{k,N})\},$$
 which is the property \textbf{C2} for the transverse condition on $\TT_{k,m}$.

One can similarly prove, for positive integers $l \leq L \leq k$, that
$$H^1_{\textup{tr}}(\QQ_{\ell},T_{l,n})=\ker\{H^1 (\QQ_{\ell},T_{l,n}) \stackrel{[\varpi^{L-l}]}{\lra} H^1 (\QQ_{\ell},T_{L,n})/H^1_{\textup{tr}}(\QQ_{\ell},T_{L,n})\}$$
 for all $n \leq m$, which is the property \textbf{C3} for the transverse condition on $\TT_{k,m}$.
\end{proof}
Using Propositions~\ref{prop:cart},~\ref{cart-at-p} and~\ref{cart-tr} we obtain the following:
\begin{cor}
\label{cor:cartes}
For every $n \in \NN_{k,m}$, the Selmer structure $\FFc(n)$ is cartesian on $\TT_{k,m}$.
\end{cor}
\begin{rem}
\label{application-cart}
Corollary~\ref{cor:cartes} has the following consequences for the Selmer modules:

Suppose $k\geq L \geq l$ and $m \geq S \geq s$ are positive integers, and $n \in \mathcal{N}_{k,m}$. Then:
\begin{enumerate}
\item[\textbf{(i)}] The image of $H^1_{\FFc(n)}(\QQ,T_{L,S})$ lies in $H^1_{\FFc(n)}(\QQ,T_{l,s})$ under the natural map $$H^1(\QQ,T_{L,S}) \lra H^1(\QQ,T_{l,s})$$
\item[\textbf{(ii)}] $H^1_{\FFc(n)}(\QQ,T_{L,s})$ is the pre-image of  $H^1_{\FFc(n)}(\QQ,T_{L,S})$ under the map   $$\xymatrix{H^1(\QQ,T_{L,s})\ar[rr]^{[\xx^{S-s}]} &&H^1(\QQ,T_{L,S})}\hbox{  and}$$
$H^1_{\FFc(n)}(\QQ,T_{l,S})$ is the pre-image of  $H^1_{\FFc(n)}(\QQ,T_{L,S})$ under the map
$$\xymatrix{H^1(\QQ,T_{l,S})\ar[rr]^{[\varpi^{L-l}]} &&H^1(\QQ,T_{L,S}).}$$
\end{enumerate}
\end{rem}
\subsection{Core Vertices for ($T_{\lowercase{k},\lowercase{m}}, \FFc,\PP _{\lowercase{k},\lowercase{m}}$)}
\label{subsec:core}
Throughout this section we assume \hone-\hfour\, as well as \htam\,and \hsez. Let $\overline{T}:=T/\mm T \cong T_{k,m}/(\mm,\xx)T_{k,m}=T_{1,1}$ be the residual representation. Fix a pair $k,m \in \ZZ^+$ until the end of \S\ref{subsec:core}.

\begin{define}
\label{define:selmergroup}
Suppose $M$ is an $\oo[[G_{\QQ}]]$-module which is finite over $\oo$. If $\FF$ is a Selmer structure on $M$, we define the Selmer group $H^1_{\FF} (\QQ,M ) \subset H ^1(\QQ, M)$ to be the kernel of the sum of restriction maps
$$H^1(\QQ_{\Sigma(\FF)}/\QQ,M) \lra \bigoplus_{\ell \in \Sigma(\FF)} H^1 (\QQ_\ell,M)/H^1_{\FF} (\QQ_\ell, M) $$
where $\QQ_{\Sigma(\FF)}$  denotes the maximal extension of $\QQ$ which is unramified outside $\Sigma(\FF)$.
\end{define}
See~\cite[\S2.3]{mr02} for a definition of the dual Selmer structure $\FF^*$\newnot{symbol:FF*} on $M^*$ and the dual Selmer group $H^1_{\FF^*}(\QQ,T^*)$.
\begin{define}
\label{def:core vertex}
We say that $n \in \NN_{k,m}$ is a \emph{core vertex} for the triple $(T_{\lowercase{k},\lowercase{m}},\FFc, \PP _{\lowercase{k},\lowercase{m}}$) if
\begin{itemize}
\item $H^1_{\FFc(n)^*}(\QQ,T_{k,m}^*)=0$,
\item $H^1_{\FFc(n)}(\QQ,T_{k,m})$ is a free $R_{k,m}$-module.
\end{itemize}
\end{define}
Thanks to~\cite[Corollary 4.1.9]{mr02}, there exists core vertices $n \in \NN_{k,m}$ for the Selmer triple $(\overline{T}=T_{1,1}, \FFc, \PP_{k,m})$.
\begin{define}\label{def:core rank}
The \emph{core Selmer rank} of the Selmer structure $\FFc$ on $T_{1,1}$ is the dimension of the $\mathbb{F}$-vector space $H^1_{\FFc(n)}(\QQ,\overline{T})$ for any core vertex $n \in \NN_{k,m}$. By~\cite[Theorem 4.1.10]{mr02} this definition makes sense. We denote the core Selmer rank for $\FFc$ on $T$ by $\chi(\overline{T})=\chi(\overline{T},\FFc)$.
\end{define}
We note that the results from~\cite{mr02} we quote above rely on the hypotheses $\mathbf{H1}$-$\mathbf{H6}$ of~\cite[\S3.5]{mr02}. These hypotheses are satisfied for the triple $(\overline{T},\FFc,\PP_{k,m})$ under our running assumptions and thanks to Propositions~\ref{prop:cart} and \ref{cart-at-p}.

We assume until the end of this paper that  $\chi(\overline{T})=1$ (except in \S\ref{sec:stark-iwasawa}, where we construct another Selmer structure with core Selmer rank one). Making use of~\cite[Corollary 4.1.9]{mr02}, we construct below \emph{core vertices} for the Selmer triple $(T_{k,m}, \FFc, \PP_{k,m})$:

\begin{thm}
\label{thm1}
Suppose $\chi(\overline{T})=1$ and $n$ is a core vertex for the triple $(\overline{T}, \FFc, \PP_{k,m})$. Then:
\begin{enumerate}
\item[\textbf{(i)}] The $R_{k,m}$-module $H^1_{\FFc(n)}(\QQ, T_{k,m})$ is free of rank one.
\item[\textbf{(ii)}] $H^1_{\FFc(n)^*}(\QQ, T_{k,m}^*)=0$.
\end{enumerate}
We say that $n$ is a core vertex  for the triple $(T_{k,m}, \FFc, \PP_{k,m})$.
\end{thm}
Theorem~\ref{thm1} is an extension of~\cite[Theorem 4.1.10]{mr02} to cover $G_\QQ$-representations over the coefficient rings $R_{k,m}$. We note that Mazur and Rubin prove their assertion~\cite[Theorem 4.1.10]{mr02} only for one-dimensional artinian coefficient rings.

Before proving Theorem~\ref{thm1}, we need several preliminary results. Recall that hypotheses \hone-\hfour\, along with \htam\,and \hsez\, are in effect throughout this section, as well as the assumption that $\chi(\overline{T})=1$.
\begin{lemma}
\label{lem:applycart}
Suppose $n \in \NN_{k,m}$. Then the injection
$$\xymatrix @C=.5in {[\varpi^{k-1}]:\,T_{1,m}=T\otimes R_{1,m}\ar[r] & T\otimes R_{k,m}=T_{k,m}}$$ induces isomorphisms
\begin{itemize}
\item[(i)] $[\varpi^{k-1}]: H^1(\QQ, T_{1,m})\stackrel{\sim}{\lra}H^1(\QQ, T_{k,m})[\mm],$
\item[(ii)] $[\varpi^{k-1}]: H^1_{\FFc(n)}(\QQ, T_{1,m})\stackrel{\sim}{\lra}H^1 _{\FFc(n)}(\QQ, T_{k,m})[\mm].$
\end{itemize}
\end{lemma}

\begin{proof}
Set $\mathbb{T}:=T\otimes\Lambda$ and let $\mathbf{\mathcal{M}} :=(\mm,\xx)$ be the maximal ideal of $\Lambda$. Then, $\mathbb{T}/\mathbf{\mathcal{M}}\mathbb{T}=T/\mm T$, and since we assumed \hthree, it follows that $(\mathbb{T}/\mathbf{\mathcal{M}}\mathbb{T})^{G_{\QQ}}=0$. By \cite[Lemma 2.1.4]{mr02}, we have $S^{G_{\QQ}}=0$ for every subquotient $S$ of $\mathbb{T}$ as well. In particular, $T_{k,m}^{G_{\QQ}}=0$ for every $k,m \in \ZZ^+$.

Now the $G_{\QQ}$-cohomology of the exact sequences
\begin{equation*}\vcenter{\xymatrix@R=.1in{
0\ar[r]&T_{1,m}\ar[r]^{\varpi^{k-1}}& T_{k,m}\ar[r] &T_{k-1,m}\ar[r]& 0\\
&&0\ar[r]&  T_{k-1,m} \ar[r]^{\varpi}&T_{k,m}
}}
\end{equation*}
gives (together with the vanishing of $S^{G_{\QQ}}$ for the relevant quotients $S$ of $\mathbb{T}$)
\begin{equation}
\label{eqn:starstar}
\vcenter{
\xymatrix @C=.35in@R=.1in {
0\ar[r]&H^1(\QQ,T_{1,m})\ar[r]^{[\varpi^{k-1}]}& H^1(\QQ,T_{k,m})\ar[r]^(.47){\mathfrak{R}} &H^1(\QQ,T_{k-1,m})\\
&&0\ar[r]&  H^1(\QQ,T_{k-1,m}) \ar[r]^(.52){[\varpi]}&H^1(\QQ,T_{k,m})
}
}
\end{equation}
This shows
\[
\begin{array}{ccl}
H^1(\QQ,T_{1,m})&=&\ker(\mathfrak{R})\\
									&=&\ker([\varpi]\circ\mathfrak{R}) \\
									&=&H^1(\QQ,T_{k,m})[\mm].
\end{array}
\]
The second equality is because $[\varpi]$ is injective on $H^1(\QQ,T_{k-1,m})$ thanks to the second exact sequence in~(\ref{eqn:starstar}), and the last equality is because $[\varpi]\circ \mathfrak{R}$ on $H^1(\QQ,T_{k,m})$ is multiplication by $\varpi$).

This proves (i). To finish the proof of Lemma, we need to show that $H^1_{\FFc(n)}(\QQ, T_{1,m})$ maps isomorphically onto $H^1_{\FFc(n)}(\QQ, T_{k,m})[\mm]$ under the isomorphism of (i). This follows at once from Remark~\ref{application-cart}.
\end{proof}

\begin{rem}
\label{rem:cartx} Starting from the exact sequences
 \begin{equation*}\vcenter{\xymatrix@R=.1in{
0\ar[r]&T_{k,1}\ar[r]^{\xx^{m-1}}& T_{k,m}\ar[r] &T_{k,m-1}\ar[r]& 0\\
&&0\ar[r]&  T_{k,m-1} \ar[r]^{\xx}&T_{k,m}
}}
\end{equation*}
one may similarly prove that $$[\xx^{m-1}]: H^1_{\FFc(n)}(\QQ, T_{k,1}) \stackrel{\sim}{\lra} H^1_{\FFc(n)}(\QQ, T_{k,m})[\xx]$$ 
for every $k,m \in \ZZ^+$.
\end{rem}

Write $\mathcal{M}_{k,m}=(\mm,\xx)$ for the maximal ideal of $R_{k,m}$. Since
 $$ H^1_{\FFc(n)}(\QQ, T_{k,m})[\mathcal{M}]=\{H^1_{\FFc(n)}(\QQ, T_{k,m})[\mm]\}[\xx],$$
  it follows from  Lemma~\ref{lem:applycart} and Remark~\ref{rem:cartx}  that
\be
\label{comparison}
H^1_{\FFc(n)}(\QQ,T_{k,m})[\mathcal{M}] \cong H^1_{\FFc(n)}(\QQ, \overline{T}).
\ee

\begin{prop}
\label{lem:upper}
Suppose $n \in \NN_{k,m}$ is a core vertex for $(\overline{T}, \FFc, \PP_{k,m})$ and $\chi(\overline{T},\FFc)=1$. Then the $R_{k,m}$-module $\textup{Hom}\left(H^1 _{\FFc(n)}(\QQ, T_{k,m}),\Phi/\oo \right)$ is cyclic.
\end{prop}

\begin{proof}
By the assumption on $n$, it follows that the $\mathbb{F}$-vector space $H^1_{\FFc(n)}(\QQ, T_{1,1})= H^1_{\FFc(n)}(\QQ,\overline{T})$ is  one-dimensional. Furthermore, by Remark~\ref{rem:cartx} applied with $k=1$ and using the fact that $R_{1,m}$ is a principal ideal ring, it follows that $H^1_{\FFc(n)}(\QQ, T_{1,m})$ is a cyclic $R_{1,m}$-module:
$$H^1 _{\FFc(n)}(\QQ, T_{1,m}) \cong \mathbb{F}[[\xx]]/\xx^{\alpha}, \hbox{  for some }  \alpha \leq m  \hbox{ (as an } R_{1,m}\hbox{-module)}.$$

By Lemma~\ref{lem:applycart}(ii), we have
\be
\label{eqn:upper1}
H^1 _{\FFc(n)}(\QQ, T_{k,m})[\mm] \cong H^1 _{\FFc(n)}(\QQ, T_{1,m}).
\ee

Taking the Pontryagin duals of both sides in~(\ref{eqn:upper1}), we obtain
\[
\begin{array}{rl}
\textup{Hom}\left(H^1 _{\FFc(n)}(\QQ, T_{k,m}),\Phi/\oo \right) \otimes\oo/\mm&\cong \textup{Hom}\left(H^1_{\FFc(n)}(\QQ, T_{1,m}),\mathbb{F} \right)\\
&\cong \textup{Hom}\left(\mathbb{F}[[\xx]]/\xx^{\alpha}, \mathbb{F} \right)\\
&\cong \mathbb{F}[[\xx]]/\xx^{\alpha}
\end{array}
\]
which shows that $\textup{Hom}\left(H^1 _{\FFc(n)}(\QQ, T_{k,m}),\Phi/\oo\right)\big{/}\mm\cdot\textup{Hom}\left(H^1 _{\FFc(n)}(\QQ, T_{k,m}),\Phi/\oo \right)$ is generated by one element as an $R_{1,m}$-module. By Nakayama's lemma, it now follows that the $R_{k,m}$-module $\textup{Hom}\left(H^1 _{\FFc(n)}(\QQ, T_{k,m}),\Phi/\oo\right)$ is generated by one element.
\end{proof}

\begin{prop}
\label{lem:lower}
Under the hypotheses and the notation of Proposition~\ref{lem:upper},
$$
\length_{\oo}\,H^1_{\FFc(n)}(\QQ, T_{k,m})- \length_{\oo}\,H^1_{\FFc(n)^*}(\QQ, T_{k,m}^*)=k\cdot m.
$$
\end{prop}

\begin{proof}(All the \emph{length}s in this proof are measured over $\oo$.)

Fix $m$, and let $R_k$ denote the ring $\oo/\mm^k$. Using~\cite[Corollary 2.3.6]{mr02}, it suffices to prove Lemma only for $n=1$, i.e., for $\FFc(n)=\FFc$. By \cite[Theorem 4.1.5]{mr02}, there are integers $r_k,\,s_k$, one of which can be taken to be zero, such that, there is an isomorphism
$$H^1 _{\FFc}(\QQ, T_{k,m})\oplus R_k^{r_k}\cong H^1_{\FFc^*}(\QQ, T_{k,m}^*) \oplus R_k^{s_k}.$$
 Here we regard $T_{k,m}$ as an $R_k$-representation. We note that proof of~\cite[Theorem 4.1.5]{mr02} still holds true in our case by Proposition~\ref{prop:cart} and Lemma~\ref{lem:applycart}. The proof of the same theorem in fact shows that $r_k,\,s_k$ do not depend on $k$, we denote this common value by $\chi(T_{m}^*),\,\chi(T_m)$, respectively. Hence, there is an isomorphism
\be
\label{eqn:structure}
H^1_{\FFc}(\QQ, T_{k,m})\oplus R_k^{\chi(T_{m}^*)}\cong H^1_{\FFc^*}(\QQ, T_{k,m}^*) \oplus R_k^{\chi(T_m)}.
\ee
Passing to inverse limit in (\ref{eqn:structure}) (by making use of Remark~\ref{application-cart}), we see that
$$\chi(T_m)-\chi(T_m^*)=\textup{rank}_{\oo}\,H^1 _{\FFc}(\QQ, T_m)-\textup{corank}_{\oo}\,H^1_{\FFc^*}(\QQ, T_m^*).$$
 On the other hand, by \cite[Lemma 5.2.15]{mr02}, it follows that
 \be\label{eqn:array}\begin{array}{rcl}
 \textup{rank}_{\oo}\,H^1 _{\FFc}(\QQ, T_m)-\textup{corank}_{\oo}\,H^1_{\FFc^*}(\QQ, T_m^*)&=&\textup{rank}_{\oo}\,T_m^- + \textup{corank}_{\oo}\,H^0(\QQ_p,T_m^*)\\
&=&m\cdot\textup{rank}_{\oo}\,T^- + \textup{corank}_{\oo}\, H^0(\QQ_p,T_m^*),
\end{array}
\ee
where for a $G_\QQ$-module $M$, we write $M^-$ for the $(-1)$-eigenspace for an arbitrary complex conjugation $\mathfrak{c}\in G_\QQ$. The second equality in (\ref{eqn:array}) is because the complex conjugation acts trivially on $\Lambda/\xx^m$. Hence,
 \begin{align*} \textup{length}_{\oo}\,H^1 _{\FFc}(\QQ,T_{k,m})-\textup{length}_{\oo}\,H^1_{\FFc^*}(\QQ,T_{k,m}^*)
&=k\left(\chi(T_m)-\chi(T_m^*)\right)\\
&=k\cdot m\cdot \textup{rank}_{\ZZ_p}T^- + k\cdot\textup{corank}_{\ZZ_p}H^0(\QQ_p,T_m^*)
\end{align*}
By local Tate duality, $H^0(\QQ_p,T_m^*)$ is dual to $H^2(\QQ_p,T_m)$, and
$$H^2(\QQ_p,T_m)=H ^2(\QQ_p,T\otimes\Lambda/\xx^m) \cong H^2(\QQ_p,{T\otimes\LL})/\xx^m H ^2(\QQ_p,{T\otimes\LL})$$
 because $G_{\QQ_p}$ has cohomological dimension 2. However, $H^2(\QQ_p,{T\otimes\LL})=0$ by Lemma \ref{h2-vanishing}, which in turn shows that $\textup{corank}_{\oo}\, H^0(\QQ_p,T_m^*)=0$ for all $m$. We therefore conclude that
 \begin{align}
 \label{align:lower}
 \textup{length}_{\oo}\,H^1_{\FFc}(\QQ,T_{k,m})- \textup{length}_{\oo}\,H^1_{\FFc^*}(\QQ,T_{k,m}^*)
=k\cdot m\cdot\textup{rank}_{\oo}\,T^-.
\end{align}
The equation (\ref{align:lower}) for $m=1$, together with our assumption that $\chi(\overline{T})=1$ imply that $\textup{rank}_{\oo}\,T^-=1$. Proof of Lemma follows.
\end{proof}
\begin{rem}
\label{generic}
For any height one prime $\wp$ of $\Lambda$ which is not exceptional in the sense of~\cite[Definition 5.3.12]{mr02}, one may prove that the core Selmer rank $\XX(T\otimes S_{\wp}, \FFc)$ for the canonical Selmer structure on $T\otimes S_{\wp}$ does not depend on $\wp$ (see \cite[Lemma 5.3.16]{mr02}). Here, $S_{\wp}$ is the integral closure of $\LL/\wp$ inside its field of fractions $\textup{Frac}(\LL/\wp)$. This common value is called the \emph{generic core Selmer rank} of the canonical Selmer structure $\FFc$ on $T\otimes\LL$. In particular, our assumption that the core Selmer rank $\XX(\overline{T},\FFc)=1$, together with the hypothesis \hez\, is equivalent to the assumption that the generic core Selmer rank equals \emph{one}.

\end{rem}
\begin{proof}[Proof of Theorem~\ref{thm1}]
Immediate from Propositions~\ref{lem:upper} and \ref{lem:lower}.
\end{proof}


\section{$\Lambda$-adic Kolyvagin Systems}
\label{chptr:KS}
Many of the arguments in this section are essentially present in~\cite[\S4.3 and Appendix B]{mr02}. We still include every detail for completeness.
\subsection{Kolyvagin Systems over  $R_{\lowercase{k},\lowercase{m}}$}
\label{KS1}
\subsubsection{Generalities}
Let $R$ be a local artinian ring with finite residue field. If $M$ is an $R$-module and $\psi \in \textup{Hom}(M,R)$, we define for any integer $r$ a map $\stackrel{r}{\bigwedge}M \ra \stackrel{r-1}{\bigwedge} M$, also denoted by $\psi$,  by setting
 $$\psi:\, m_1 \wedge \dots \wedge m_r \mapsto \sum _{i=1}^{r}(-1)^{i+1}\psi(m_i)m_1\wedge \dots m_{i-1}\wedge m_{i+1} \wedge \dots\wedge m_r .$$
We also define a map $\stackrel{s}{\bigwedge}\hbox{Hom}(M,R) \ra \hbox{Hom}(\stackrel{r}{\bigwedge}M,\stackrel{r-s}{\bigwedge}M)$ for $s\leq r$, by iterating the construction above: $\psi_1\wedge\dots\wedge\psi_s:=\psi_s\circ\dots\circ\psi_1.$

Lemma~\ref{lemma1} below is one of the main technical tools we utilize in this section. It generalizes~\cite[Lemma B.1]{mr02} to apply with coefficient rings which are artinian (and not necessarily principal artinian, as assumed in loc.cit.). However, Lemma~\ref{lemma1} claims less than~\cite[Lemma B.1]{mr02}, e.g., Lemma~\ref{lemma1} does not determine the exact image of the map $\Psi$ below unless the map $\psi$ defined below is surjective. 
\begin{lemma}
\label{lemma1}
Suppose $M$ is a free $R$-module of rank at least $r+1$ and  $\psi_{1},\dots, \psi_{r} \in \textup{Hom}(M,R)$. Define:
\[
\begin{array}{lll}
\psi&=&\psi_{1}\oplus\dots\oplus\psi_{r}:\,M\lra R^r \\
\Psi&=&\psi_{1}\wedge\dots\wedge\psi_{r}:\, \stackrel{r+1}{\bigwedge} M\lra M
\end{array}
\]
Then
\[
\begin{array}{lrcl}
&\Psi\big(\stackrel{r+1}{\bigwedge} M\big) &\subset& \ker(\psi), \hbox{ and, }\\
&&=&\ker(\psi) \hbox{ if } \psi \hbox{ is surjective.}
\end{array}
\]
\end{lemma}

\begin{proof}

Suppose first that $\psi$ is surjective with kernel $A\subset M$. Then the image of $\psi$ is $R^r$, which is a  projective $R$-module, hence the exact sequence
$$0\lra A \lra M \stackrel{\psi}{\lra} R^r \lra 0$$
 splits. Therefore, there is a submodule $B$ of $M$ which complements $A$, i.e.,  $M=A \oplus B$ and $\psi$ maps $B$ isomorphically onto $R^r$. We have a functorial isomorphism
\be
\label{eqn:id}
\stackrel{r+1}{\bigwedge}M \cong \bigoplus_{p+q=r+1}\left(\stackrel{p}{\bigwedge}A\,\,\otimes\stackrel{q}{\bigwedge}B\right).
\ee
The map $\Psi$, after the identification (\ref{eqn:id}), takes the factor $A\otimes \stackrel{r}{\bigwedge}B$ isomorphically onto $A$ and kills the other summands. This proves the equality when $\psi$ is surjective.

To prove the containment in general, we need to check that the map
$$\psi_s\circ\Psi:\, \stackrel{r+1}{\bigwedge}M \lra R$$
 is identically \emph{zero} for all $s=1,\dots,r$. We will only check this when $s=1$,  and the other cases follow in an identical fashion.

A simple combinatorial argument shows that
\begin{align*}\Psi(m_1\wedge&\dots \wedge m_{r+1}) = (\psi_{1}\wedge\dots\wedge\psi_{r})(m_1 \wedge \dots \wedge m_{r+1})=\\ &\sum_{i=1}^{r+1}\sum_{\sigma}(-1)^{i}\textup{sign}(\sigma)\cdot\psi_1(m_{\sigma(1)})\dots \psi_{i-1}(m_{\sigma(i-1)})\cdot\psi_i(m_{\sigma(i+1)})\dots\psi_r(m_{\sigma(r+1)})\cdot m_i
\end{align*}
where the second sum is over all permutations $\sigma$ 
of $\{1,\dots,i-1,i,i+1,\dots,r+1\}$ such that $\sigma(i)=i$. Then
\begin{align*}
\psi_1\circ\Psi&(m_1\wedge\dots\wedge m_{r+1})=\\
&\sum_{i=1}^{r+1}\sum_{\sigma}(-1)^{i}\textup{sign}(\sigma)\cdot\psi_1(m_{\sigma(1)})\dots \psi_{i-1}(m_{\sigma(i-1)})\cdot\psi_i(m_{\sigma(i+1)})\dots\psi_r(m_{\sigma(r+1)})\cdot \psi_1(m_i).
\end{align*}

We fix $i$. Then the expression  $\psi_1(m_{\sigma(1)})\cdots \psi_{i-1}(m_{\sigma(i-1)})\cdot\psi_i(m_{\sigma(i+1)})\cdots\psi_r(m_{\sigma(r+1)})\cdot \psi_1(m_i)$ appears exactly twice in the final displayed equality above. We now identify these terms and show that they will have opposite signs. This will conclude that $\psi_1\circ\Psi(m_1\wedge\dots\wedge m_{r+1})=0$.

Let $\sigma(1)=j$, and let $$\alpha= \left \{
\begin{array}{ccc}
	 (1,i,i+1,\dots,j)& ,& \hbox{if } j>i  \\
          (1,i,i-1,\dots,j)&,& \hbox{if } j<i
\end{array}
\right.$$
be a cycle in the symmetric group $\mathfrak{S}_{r+1}$. Set $\tau=\sigma \circ \alpha$. Then the term $$\psi_1(m_{\sigma(1)})\dots \psi_{i-1}(m_{\sigma(i-1)})\cdot\psi_i(m_{\sigma(i+1)})\dots\psi_r(m_{\sigma(r+1)})\cdot \psi_1(m_i)$$ appears once for the permutation $\sigma$ itself (with the sign $(-1)^i\textup{sign}(\sigma)$) and for the permutation $\tau$ (with the sign $(-1)^{j}\textup{sign}(\tau)$). Since $\textup{sign}(\tau)=\textup{sign}(\sigma)\textup{sign}(\alpha)=\textup{sign}(\sigma)(-1)^{i-j+1}$, it follows that $(-1)^{j}\textup{sign}(\tau)=(-1)^{i+1}\textup{sign}(\sigma)$. This shows that
 $$\psi_1\circ\Psi(m_1\wedge\dots\wedge m_{r+1})=0.$$ 

\end{proof}
\subsubsection{The lower bound}
\label{sec:lowerbound}
We keep the notation above. Let $\nu(n)$  be the number of prime divisors of $n$. In this section, we obtain (Theorem~\ref{KS}) a lower bound on the size of the $R_{k,m}$-module of Kolyvagin Systems.
\begin{define}
\label{simplicial sheaf}
If $X$ is a graph and $\textup{Mod}_{R_{k,m}}$ is the category of $R_{k,m}$-modules, a \emph{simplicial sheaf} $\mathcal{S}$\newnot{symbol:calS} on $X$ with values in $\textup{Mod}_{R_{k,m}}$ is a rule assigning
\begin{itemize}
\item an $R_{k,m}$-module $\mathcal{S}(v)$ for every vertex $v$ of $X$,
\item an $R_{k,m}$-module $\mathcal{S}(e)$ for every edge $e$ of $X$,
\item an $R_{k,m}$-module homomorphism $\psi_{v}^e:\,\mathcal{S}(v)\ra\mathcal{S}(e)$ whenever the vertex $v$ is an endpoint of the edge $e$.
\end{itemize}
A \emph{global section} of $\mathcal{S}$ is a collection $\{\kappa_v \in \mathcal{S}(v): v \hbox{ is a vertex of } X\}$ such that, for every edge $e=\{v,v^{\prime}\}$ of $X$, we have  $\psi_{v}^e(\kappa_v)=\psi_{v^{\prime}}^e(\kappa_{v^{\prime}})$ in $\mathcal{S}(e)$. We write $\Gamma(\mathcal{S})$ for the $R_{k,m}$-module of global sections of $\mathcal{S}$.
\end{define}

\begin{define}
\label{selmer sheaf}
For the Selmer triple $(T_{k,m},\FFc,\PP_{k,m})$, we define a graph $\XX_{k,m}=\XX(\PP_{k,m})$ by taking the set of vertices of $\XX_{k,m}$ to be $\NN_{k,m}$, and the edges to be $\{n,n\ell\}$ whenever $n,n\ell \in \NN_{k,m}$ (with $\ell$ prime).

The \emph{Selmer sheaf} $\mathcal{H}$\newnot{symbol:calH} is the simplicial sheaf on $\XX_{k,m}$ given as follows. Set $G_n:=\otimes_{\ell | n} \mathbb{F}_{\ell}^\times$. We take
 \begin{itemize}
\item $\mathcal{H}(n) := H^1_{\FFc(n)}(\QQ,T_{k,m}) \otimes G_n$ for $n\in \NN_{k,m}$,
\item  if $e$ is the edge $\{n,n\ell\}$ then $\mathcal{H}(e):= H^1_{s}(\QQ_{\ell},T_{k,m})\otimes G_{n\ell}$.
\end{itemize}
 We define the vertex-to-edge maps to be
\begin{itemize}
\item $\psi_{n\ell}^{e}:\,H^1_{\FFc(n\ell)}(\QQ,T_{k,m}) \otimes G_{n\ell} \ra H^1_{s}(\QQ_{\ell},T_{k,m})\otimes G_{n\ell}$ is localization followed by the projection to the singular cohomology $H^1_s(\QQ_\ell,T_{k,m})$.
\item $\psi_{n}^{e}:\,H^1_{\FFc(n)}(\QQ,T_{k,m}) \otimes G_{n} \ra H^1_{s}(\QQ_{\ell},T_{k,m})\otimes G_{n\ell}$ is the composition of localization at $\ell$ with the finite-singular comparison map $\phi^{\textup{fs}}_{\ell}$.
\end{itemize}
 \end{define}

 A \emph{Kolyvagin system} for the triple $(T_{k,m},\FFc, \PP_{k,m})$ is simply a global section of the Selmer sheaf $\mathcal{H}$.

Recall the definition of the modified Selmer structure $\FFc^n$ on $T_{k,m}$ (Definition~\ref{modified selmer-1}).
\begin{prop}
\label{thm2}
If $n \in \NN_{k,m}$ is a core vertex for $(T_{k,m}, \FFc, \PP_{k,m})$, then the $R_{k,m}$-module $H^1_{\FFc^{\,n}}(\QQ,T_{k,m})$ is free of rank $\nu(n)+1$.
\end{prop}

\begin{proof}
We have the following exact sequences:\\
\begin{equation*}
0 \lra H^1_{\FFc(n)}(\QQ,T_{k,m})\lra H^1_{\FFc^{n}}(\QQ,T_{k,m})\lra \bigoplus_{\ell \mid n}\frac{ H^1(\QQ_{\ell},T_{k,m})}{H^1_{\textup{tr}}(\QQ_{\ell},T_{k,m})}
\end{equation*}

$$
0 \lra H^1_ {(\FFc^{n})^*}(\QQ,T_{k,m}^*)\lra H^1_{\FFc(n)^*}(\QQ,T_{k,m}^*)\lra \bigoplus_{\ell \mid n} H^1_{\textup{tr}}(\QQ_{\ell},T_{k,m}^*)
$$

We have  $H^1_{\FFc(n)^*}(\QQ,T_{k,m}^*)=0$ since $n\in \NN_{k,m}$ is a core vertex, hence the Poitou-Tate global duality theorem (see \cite[Theorem 1.3.5]{r00} and \cite[Theorem 2.3.4]{mr02}) implies that the right-most map in the first sequence is surjective. Thus the sequence
\begin{equation*}
0 \lra H^1_{\FFc(n)}(\QQ,T_{k,m})\lra H^1_{\FFc^{n}}(\QQ,T_{k,m})\lra \bigoplus_{\ell \mid n}\frac{H^1(\QQ_{\ell},T_{k,m})}{H^1_{\textup{tr}}(\QQ_{\ell},T_{k,m})}\lra 0
\end{equation*}\\
is exact. By Theorem~\ref{thm1}, the $R_{k,m}$-module $H^1_{\FFc(n)}(\QQ,T_{k,m})$ is free of rank one, and the sum $\bigoplus_{\ell \mid n}\frac{ H^1(\QQ_{\ell},T_{k,m})}{H^1_{\textup{tr}}(\QQ_{\ell},T_{k,m})}$ is free (hence projective) of rank $\nu(n)$ by Lemma~\ref{lem:ref} and Lemma~\ref{lem:trans}. This completes the proof.
\end{proof}
Fixing a generator of $\textup{Gal}(\QQ(\mu_{\ell})/\QQ)$, we may view (by Lemma~\ref{lem:ref} and Lemma~\ref{lem:trans}) the finite-singular comparison map as an isomorphism
$$H^1_{f}(\QQ_{\ell},T_{k,m}) \stackrel{\phi_{\ell}^{\textup{fs}}}{\lra}H^1_{\textup{s}}(\QQ_{\ell},T_{k,m}) \cong H^1_{\textup{tr}}(\QQ_{\ell},T_{k,m}).$$
 For each $\ell \in \PP_{k,m}$, we fix an isomorphism $\iota_{\ell}:\, H^1_{\textup{tr}}(\QQ_{\ell},T_{k,m}) \ra R_{k,m}.$ Fix a core vertex $n \in \NN_{k,m}$ and order the primes $\ell_{1}\,,\dots,\ell_{\nu(n)}$ dividing $n$ arbitrarily. Let $\textup{loc}_i^{\textup{f}}\,(\hbox{resp., } \textup{loc}_i^{\textup{tr}})$ be the map $H^1(\QQ,T_{k,m})\ra R_{k,m}$ defined as the localization at $\ell_{i}$, followed by the projection onto the finite (resp., the transverse) submodule, an then followed by $\iota_{\ell_i} \circ \phi^{\textup{fs}}_{\ell_i}$ $(\hbox{resp., } \iota _{\ell_i})$\,:
\begin{equation*}
\xymatrix{
H^1(\QQ, T_{k,m})\ar[r]\ar @{-->}[rrddd]_{\textup{loc}_i^{\textup{f}}}&H^1(\QQ_{\ell_{i}}, T_{k,m})\ar[r]^(.46){\textup{proj}}&H^1_{f} (\QQ_{\ell_{i}}, T_{k,m})\ar[d]_{\phi_{\ell_i}^{\textup{fs}}}^{\cong}\\
&&H^1_{\textup{s}}(\QQ_{\ell_i},T_{k,m})\ar[d]^{\cong}\\
&&H^1_{\textup{tr}}(\QQ_{\ell_i},T_{k,m})\ar[d]_{\iota_{\ell_i}}^{\cong} \\
&&R_{k,m}
}
\end{equation*}

and

\begin{equation*}
\xymatrix{
H^1(\QQ, T_{k,m})\ar[r]\ar @{-->}[rrdd]_{\textup{loc}_i^{\textup{tr}}}&H^1(\QQ_{\ell_{i}}, T_{k,m})\ar[r]^{\textup{proj}}&H^1_{\textup{s}} (\QQ_{\ell_{i}}, T_{k,m})\ar[d]^{\cong}\\
&&H^1_{\textup{tr}}(\QQ_{\ell_i},T_{k,m})\ar[d]^{\cong}_{\iota_{\ell_i}}\\
&&R_{k,m}
}
\end{equation*}
\begin{define}
For each $r \mid n$, define
$$\psi ^{(r)}_i= \left \{
\begin{array}{cl}
	 \textup{loc}^{\textup{f}}_i\hbox{ ,}& \hbox{if } \ell_i  \hbox{ divides } r, \\
	 \textup{loc}_i^{\textup{tr}}\hbox{ ,}& \hbox{if } \ell _i \hbox{ does not divide } r
\end{array}
\right. $$
and let
$$
\begin{array}{rlrl}
\psi^{(r)}=&\psi^{(r)}_{1}\oplus\dots\oplus\psi_{\nu(n)}^{(r)}&:&\,H^1_{\FFc^{n}}(\QQ,T_{k,m})\ra R_{k,m}^{\nu(n)}, \\
\Psi^{(r)}=&\psi^{(r)}_{1}\wedge\dots\wedge\psi^{(r)}_{\nu(n)}&:& \,\stackrel{\nu(n)+1}{\bigwedge} H^1_{\FFc^{n}}(\QQ,T_{k,m})\ra H^1_{\FFc^{n}}(\QQ,T_{k,m}).
\end{array}
$$
\end{define}

\begin{prop}
\label{prop3}
Suppose $n \in \NN_{k,m}$ is a core vertex, then for all $r \mid n$,
\[
\begin{array}{rcl}
\Psi^{(r)}\left(\stackrel{\nu(n)+1}{\bigwedge}H^1_{\FFc^{n}}(\QQ,T_{k,m})\right)& \subset&  H^1_{\FFc(r)}(\QQ,T_{k,m})\\

																																		&=& H^1_{\FFc(r)}(\QQ,T_{k,m}) \hbox{, if r is a core vertex.}
\end{array}
\]
\end{prop}

\begin{proof}
For $r\mid n$, the following sequence is exact:
\be\label{eqn:exactstar}0 \lra H^1_{\FFc(r)}(\QQ, T_{k,m}) \lra H^1_{\FFc ^n}(\QQ, T_{k,m}) \stackrel{\psi^{(r)}}{\lra} R_{k,m}^{\nu(n)}\lra \hbox{coker}(\psi^{(r)}) \lra 0.
\ee
If $r$ is a core vertex, then $H^1_{\FFc(r)}(\QQ, T_{k,m})$ is free of rank one by definition, and $H^1_{\FFc^n}(\QQ,T_{k,m})$ is free of rank $\nu(n)+1$ by Proposition~\ref{thm2}. Counting lengths in the exact sequence (\ref{eqn:exactstar}), one checks that $\textup{coker}(\psi^{(r)})=0$, i.e., $\psi^{(r)}$ is surjective. Now Lemma~\ref{lemma1} concludes that
$$\Psi^{(r)}\left(\stackrel{\nu(n)+1}{\bigwedge}H^1_{\FFc^{n}}(\QQ,T_{k,m})\right)=H^1_{\FFc(r)}(\QQ,T_{k,m})$$
 when $r$ is a core vertex. The general containment also follows similarly from Lemma~\ref{lemma1}.
\end{proof}
\begin{define}
Choose $c \in \stackrel{\nu(n)+1}{\bigwedge}H^1_{\FFc^{n}}(\QQ,T_{k,m})$ and for each divisor $r$ of $n$, define
$$\kappa_r:=\kappa_r(n,c)=(-1)^{\nu(r)}\Psi^{(r)}(c) \in H^1_{\FFc(r)}(\QQ,T_{k,m}).$$
\end{define}
\begin{prop}
\label{prop4}
For every $r\ell_i \mid n$, we have $\textup{loc} _{\ell_i}^{\textup{f}}(\kappa_r)=\textup{loc}_{\ell_i}^{\textup{tr}}(\kappa_{r\ell_i})$.
\end{prop}

\begin{proof}
\begin{align*}
\textup{loc}_{\ell_i}^{\textup{f}}(\kappa_r)&=(-1)^{\nu(r)}\,\textup{loc}_{\ell_i}^{\textup{f}}(\Psi^{(r)}(c))\\
													&=(-1)^{\nu(r)}(\Psi^{(r)} \wedge \textup{loc}_{\ell_i}^{\textup{f}})(c)\\
													&=-(-1)^{\nu(r)}(\Psi^{(r\ell_{i})} \wedge\textup{loc}_{\ell_i}^{\textup{tr}})(c)\\
													&=(-1)^{^{\nu(r\ell_i)}}\,\textup{loc}_{\ell_i}^{\textup{tr}}(\Psi^{(r\ell_{i}}(c))\\
													&=\textup{loc}_{\ell_i}^{\textup{tr}}(\kappa_{r\ell_i}).
\end{align*}
where the we obtain third equality by transposing $\textup{loc}_{\ell_i}^{\textup{f}}$ and $\textup{loc}_{\ell_i}^{\hbox{\sss{tr}}}$ which occur in $\Psi^{(r)} \wedge \textup{loc}_{\ell_i}^{\textup{f}}$.
\end{proof}
Thus the collection $\{\kappa_r(n,c)\}$ for $r\mid n$ gives a section of the Selmer sheaf $\mathcal{H}$ to the subgraph $\mathcal{X}_n$ of $\mathcal{X}(\PP_{k,m})$, whose vertices are all the positive divisors $r$ of $n$. We will show that if $n^{\prime} = nd $ is another core vertex, then this section can be extended to a section of $\mathcal{H}$ over the graph $\mathcal{X}_{nd}$.

\begin{rem}
 Note that the section $\kappa(n,c)=\{\kappa_r(n,c): r\mid n\}$ depends on the choice of ordering of primes dividing $n$, but only up to a sign. We will adopt the convention in \cite[Appendix B]{mr02}; namely we will maintain the same ordering of primes dividing $n$ when we extend to $n^{\prime}=nd$, but place them after the \emph{new} primes which divide $d$.
\end{rem}

\begin{lemma}
\label{lemma5}In the following diagram, the image of the horizontal arrow contains the image of the vertical arrow:
$$\xymatrix @C=1in { \stackrel{\nu(n^{\prime})+1}{\bigwedge}H^1_{\FFc^{nd}}(\QQ,T_{k,m}) \ar[r]^{\textup{loc}_{1}^{\textup{tr}} \wedge \dots \wedge\textup{loc}_{\nu(d)}^{\textup{tr}}}&\stackrel{\nu(n)+1}{\bigwedge}H^1_{\FFc^{nd}}(\QQ,T_{k,m}) \\
& \stackrel{\nu(n)+1}{\bigwedge}H^1_{\FFc^{n}}(\QQ,T_{k,m})\ar[u] }
$$
If the image of $c^{\prime} \in \stackrel{\nu(n^{\prime})+1}{\bigwedge}H^1_{\FFc^{nd}}(\QQ,T_{k,m})$ under the horizontal map agrees with the image of $c$ under the vertical map, then the section $\kappa(n^{\prime}, c^{\prime})$ of $\mathcal{X}_{nd}$ extends the section $\kappa(n,c)$ of $\mathcal{X}_n$.
\end{lemma}

\begin{proof}
As in the proof of Theorem~\ref{thm2}, global duality and the vanishing of $H^1_{(\FFc^{n})^*}(\QQ,T_{k,m}^*)$ imply that we have an exact sequence
$$0\lra H^1_{\FFc^{n}}(\QQ,T_{k,m}) \lra H^1_{\FFc^{nd}}(\QQ,T_{k,m}) \stackrel{\textup{loc}_d}{\lra}\bigoplus_{\ell \mid d} \frac{H^1(\QQ_{\ell},T_{k,m})}{H^1_{\textup{loc}}(\QQ_{\ell},T_{k,m})}\lra 0.$$

The right-most term is projective, hence we may choose a free rank-$\nu(d)$ summand $A$, complementary to $H^1_{\FFc^{n}}(\QQ,T_{k,m}) \subset H^1_{\FFc^{nd}}(\QQ,T_{k,m})$. Thus the map $\bigoplus \textup{loc}_i^{\textup{tr}}:\, A \ra R_{k,m} ^{\nu(d)}$ is an isomorphism. This shows that the map $\bigwedge \textup{loc}_i^{\textup{tr}} :\, \bigwedge^{\nu(d)}A \lra  \bigwedge^{\nu(d)}R_{k,m}^{\nu(d)} = R_{k,m}$ and the map
$$\xymatrix @C=.95in{ \bigwedge^{\nu(d)}A \otimes \bigwedge^{\nu(n)+1} H^1_{\FFc^{n}}(\QQ,T_{k,m}) \ar[r]^(.56){{\textup{loc}_{1}^{\textup{tr}} \wedge \dots \wedge \textup{loc}_{\nu(d)}^{\textup{tr}} \otimes\textup{id}}}& \bigwedge^{\nu(n)+1} H^1_{\FFc^{n}}(\QQ,T_{k,m}) }$$
 are isomorphisms as well. Thus, in the commutative diagram
\begin{equation*}
\xymatrix @C=.95in{\bigwedge^{\nu(n^{\prime})+1} H^1_{\FFc^{nd}}(\QQ,T_{k,m}) \ar[r]^(.50){\textup{loc}_{1}^{\textup{tr}} \wedge \dots \wedge \textup{loc}_{\nu(d)}^{\textup{tr}}}  & \bigwedge^{\nu(n)+1} H^1_{\FFc^{nd}}(\QQ,T_{k,m}) \\
\bigwedge^{\nu(d)}A \otimes \bigwedge^{\nu(n)+1} H^1_{\FFc^{n}}(\QQ,T_{k,m}) \ar[r]^(.56){\textup{loc}_{1}^{\textup{tr}} \wedge \dots \wedge \textup{loc}_{\nu(d)}^{\textup{tr}} \otimes \textup{id}}\ar[u]& \ar[u] \bigwedge^{\nu(n)+1} H^1_{\FFc^{n}}(\QQ,T_{k,m})}
\end{equation*}
the lower horizontal map is surjective, and hence the first claim is proved.
To prove the second claim, we observe for $r\mid n$ that
\begin{align*}
\kappa_r (n,c)&=(-1)^{\nu(r)}\left(\psi_{\nu(d)+1}^{(r)} \wedge \dots \wedge \psi_{\nu(nd)}^{(r)}\right)(c)\\
							&=(-1)^{\nu(r)}\left(\psi_{\nu(d)+1}^{(r)} \wedge \dots \wedge \psi_{\nu(nd)}^{(r)}\right)\left(\textup{loc}_{1}^{\textup{tr}} \wedge \dots \wedge \textup{loc}_{\nu(d)}^{\textup{tr}}\right)(c ^{\prime})\\
							&=(-1)^{\nu(r)}\left(\psi_{\nu(d)+1}^{(r)} \wedge \dots \wedge \psi_{\nu(nd)}^{(r)}\right)\left(\psi_{1}^{(r)} \wedge \dots \wedge \psi_{\nu(d)}^{(r)}\right)(c ^{\prime})\\
							&=(-1)^{\nu(r)}\left(\psi_{1}^{(r)} \wedge \dots \wedge \psi_{\nu(nd)}^{(r)}\right)(c ^{\prime})\\
							&=\kappa_r(n^{\prime},c^{\prime}).
\end{align*}
\end{proof}
Let $\Gamma(\mathcal{H})$ denote the \emph{global sections} of the Selmer sheaf $\mathcal{H}$ on the graph $\XX_{k,m}$.
\begin{thm}
\label{KS}
For any core vertex $r \in \NN_{k,m}$, the map $\Gamma(\mathcal{H})\ra \mathcal{H}(r)$ is surjective.
\end{thm}

\begin{proof}
Fixing a generator of $G_r := \bigotimes _{\ell \mid r} \mathbb{F}_{\ell}^{\times}$, we may identify $\mathcal{H}(r)$ with $H^1_{\FFc(r)}(\QQ,T_{k,m})$. Fix an arbitrary $\alpha \in H^1_{\FFc(r)}(\QQ,T_{k,m})$. Using Proposition~\ref{prop3}, one may choose $c_0 \in{\bigwedge}^{\nu(r)+1} H^1_{\FFc^{r}}(\QQ,T_{k,m})$ such that $\kappa_{r}(r,c_0)=\alpha$.

By \cite[Lemma 4.1.9(iii)]{mr02} and Theorem~\ref{thm1},  we may choose a sequence of core vertices $r=r_0 \mid r_1 \mid r_2 \dots $, such that, every $n \in \NN_{k,m}$ divides an $r_i$ for some $i$. By Lemma~\ref{lemma5}, one may choose for each $i>0$ and element  $c_i \in \bigwedge ^{\nu(r_i)+1} H^1_{\FFc^{r_i}}(\QQ,T_{k,m})$ in such a way that the section $\kappa(r_{i+1},c_{i+1})$ of $\mathcal{X}_{n_i+1}$ restricts to $\kappa(r_{i},c_{i})$ on $\mathcal{X}_{n_i}$. We now define $\kappa \in \Gamma(\mathcal{H})$ to be
 $$\kappa:=\{\kappa_n\}_{n \in \NN_{k,m}}= \{\kappa_n(r_i,c_i) \hbox{ for } i \hbox{ chosen sufficiently large}\}_{n \in \NN_{k,m}}.$$
By construction, $\kappa$ maps to $\alpha$ under the map $\Gamma(\mathcal{H})\ra \mathcal{H}(r)$.
\end{proof}
\subsubsection{The upper bound}
In this section, we will obtain an upper bound on the size of  the $R_{k,m}$-module $\textbf{KS}(T_{k,m},\FFc,\PP_{k,m})$.
\begin{thm}
\label{thm:inj}
For any core vertex $r \in \NN_{k,m}$ the map $\Gamma(\mathcal{H})\ra \mathcal{H}(r)$ is injective.
\end{thm}
As an immediate corollary to Theorem~\ref{KS} and Theorem~\ref{thm:inj} we obtain:
\begin{cor}
\label{KolSys}
For any core vertex $r \in \NN_{k,m}$, the map $$\mathbf{KS}(T_{k,m}, \FFc,\PP_{k,m}) \lra H^1_{\FFc(r)}(\QQ,T_{k,m})$$ is an isomorphism. In particular, $\mathbf{KS}(T_{k,m},\FFc,\PP_{k,m})$ is a free $R_{k,m}$-module of rank one.
\end{cor}
The rest of this section is devoted to the proof of Theorem~\ref{thm:inj}. We will be closely following the techniques from \cite[\S4.3]{mr02}, many times repeating the same arguments.

Recall that  $\overline{T}=T/\mm T \cong T_{k,m}/(\mm,\xx)T_{k,m}$ and $\XX_{k,m}$ is the graph defined as in~\S\ref{sec:lowerbound} equipped with the Selmer sheaf $\mathcal{H}$. We define a subgraph $\XX^{0}_{k,m}=\XX^{0}_{k,m}(T_{k,m},\FFc,\PP_{k,m})$ of $\XX_{k,m}$ whose vertices are the core vertices of $\XX_{k,m}$ and whose edges are defined as follows: We join $n$ and $n\ell$ by an edge in $\XX^{0}_{k,m}$ if and only if the localization map $H^1_{\FFc(r)}(\QQ,\overline{T}) \ra H^1_{f}(\QQ_{\ell},\overline{T})$ is \emph{non-zero}.

We now define the sheaf $\mathcal{H}^0$ on $\XX^{0}_{k,m}$ as simply the restriction of the Selmer sheaf $\mathcal{H}$ to $\XX^{0}_{k,m}$.
\begin{define}
\label{locally free of rank one}
A simplicial sheaf $\mathcal{H}$ on a graph $\mathcal{X}$ is \emph{locally free of rank one} if for every vertex $v$ and edge $e=\{v,v^{\prime}\}$ the modules $\mathcal{H}(v)$ and $\mathcal{H}(e)$ are free of rank one and the vertex edge homomorphisms $\psi ^{e}_{v}$ are isomorphisms.
\end{define}

We will prove below that the graph $\XX^{0}_{k,m}$ is connected and $\mathcal{H}^0$ is locally free of rank one. This will be the essential step to prove Theorem~\ref{thm:inj}.

\begin{prop}
\label{connected}
The graph $\XX^{0}_{k,m}$ defined above is connected.
\end{prop}

\begin{proof}
Since the edges of $\XX^{0}_{k,m}$ are defined in terms of $\overline{T}$ (and not $T_{k,m}$ itself),  \cite[Theorem 4.3.12]{mr02} applies.\end{proof}

\begin{prop}
\label{locally free}
The sheaf $\mathcal{H}^0$ is locally free of rank one.
\end{prop}

\begin{proof}
If $n$ is a vertex of $\XX^{0}_{k,m}$, then $n$ is a core vertex, thus $\mathcal{H}^0(n)$ is free of rank one by the definition of a core vertex. Further, $\mathcal{H}^0(n):=H^1_{\textup{s}}(\QQ_{\ell},T_{k,m})\otimes G_{n\ell}$ is free of rank one as well by Lemma~\ref{lem:ref}, since $\ell \in \PP_{k,m}$.

Suppose $e$ is the edge $\{n,n\ell\}$ in $\XX^{0}_{k,m}$. By (\ref{comparison}), we have $H^1_{\FFc(n)}(\QQ, T_{k,m})[\mathcal{M}] \cong H^1_{\FFc(n)}(\QQ, \overline{T})$, and by  Proposition~\ref{prop:cart}, $H^1_{f}(\QQ_{\ell}, T_{k,m})[\mathcal{M}] \cong H^1_{f}(\QQ_{\ell}, \overline{T}).$

Further, by the choice of $n$ and by Lemma~\ref{lem:ref}, both $H^1_{\FFc(n)}(\QQ,T_{k,m})$ and $H^1_{f}(\QQ_{\ell}, T_{k,m})$ are free of rank one over $R_{k,m}$, so the non-triviality of the map $H^1_{\FFc(r)}(\QQ,\overline{T}) \ra H^1_{f}(\QQ_{\ell},\overline{T})$ implies, by Nakayama's lemma, that the map
$$H^1_{\FFc(n)}(\QQ, T_{k,m}) \lra H^1_{f}(\QQ_{\ell}, T_{k,m})$$
 is an isomorphism. By Lemma~\ref{lem:ref}, the map
 $$\mathcal{H}^0(n) \lra \mathcal{H}^0(e):=H^1_{\textup{s}}(\QQ_{\ell},T_{k,m})\otimes G_{n\ell}$$
  is an isomorphism. A similar argument (and making use of Lemma~\ref{lem:trans} this time instead of Lemma~\ref{lem:ref}) shows that the map $\mathcal{H}^0(n\ell) \ra \mathcal{H}^0(e)$ is an isomorphism as well.
\end{proof}

\begin{prop}
\label{vanish1}	
Suppose that $\pmb{\kappa} \in \mathbf{KS}(T_{k,m},\FFc, \PP_{k,m})$ and $\kappa_r=0$ for some core vertex $r$, then $\kappa_n=0$ for all core vertices $n$.
\end{prop}

\begin{proof}
For any core vertex $n$, one can find a path $P=\{r=v_0, v_1, \dots, v_k=n\}$ by Proposition~\ref{connected}. Let $e_i$ denote the edge $\{v_{i},v_{i+1}\}$. Then by Proposition~\ref{locally free}, the map
$$\Psi_P:=(\psi_{v_{k}}^{e_{k-1}})^{-1}\circ \psi_{v_{k-1}}^{e_{k-1}} \circ \dots \circ (\psi_{v_{1}}^{e_{0}})^{-1}\circ \psi_{v_{0}}^{e_{0}}: \,\mathcal{H}(r) \lra \mathcal{H}(n)$$
 is well-defined and is an isomorphism. By the definition of a Kolyvagin system, $\Psi_P(\kappa_r)=\kappa_n$\,. This proves the theorem.
\end{proof}
For a class $c \in H^1(\QQ,M)$ and any prime $\ell$, define $c_{\ell}:=\textup{loc}_\ell(c) \in H^1(\QQ_\ell,T)$ and write $c_{\ell,s}:=\textup{loc}_{\ell}^s(c) \in H^1_s(\QQ_\ell,M)$ for the projection of the class $c_\ell$ to the singular quotient $H^1_s(\QQ_\ell,M)$.

Lemma~\ref{361} and Lemma~\ref{417} below are extensions of \cite[Proposition 3.6.1]{mr02} and \cite[Lemma 4.1.7(iv)]{mr02}, respectively.  
\begin{lemma}
\label{361}
Suppose $c_1\,,c_2 \in H^1(\QQ,T_{k,m})$ and $c_3\,,c_4 \in H^1(\QQ,T_{k,m}^*)$ are all non-zero. For every $u,v\in \ZZ^+$, there is a subset $S\subset\PP_{u,v}$ of positive density such that, for $\ell \in S$, the localizations $(c_i)_{\ell}$ are all non-zero.
\end{lemma}
Mazur and Rubin state~Proposition 3.6.1 of loc.cit. only in the case when the coefficient ring $R$ is principal artinian. One can check that the arguments of~\cite[Page 31]{mr02} go through \emph{verbatim} when the coefficient ring is assumed only to be an artinian ring (not necessarily principal artinian) and give a proof of Lemma~\ref{361}.

Following~\cite{mr02}, define 
$$\lambda(n,M):=\length_{\oo} \,H^1_{\FF}(\QQ,M)$$ 
for  $M=T_{k,m}, T_{k,m}^*, \overline{T}, \overline{T}^*$;  and $\FF=\FFc(n) \hbox{ or } \FFc(n)^*$. 
\begin{lemma}
\label{417}
Suppose $n\ell \in \NN_{k,m}$ and assume that the maps induced from the natural localization maps
$$\xymatrix@R=.1in{
H^1_{\FFc(n)}(\QQ,T_{k,m})[\mathcal{M}] \ar[r]&H^1_{f}(\QQ_{\ell},T_{k,m})  \\
	H^1_{\FFc(n)^*}(\QQ,T_{k,m}^*)[\mathcal{M}] \ar[r]&H^1_{f}(\QQ_{\ell},T_{k,m}^*)
	}$$
	are both non-zero. Then
	\begin{enumerate}
	\item[\textbf{(i)}] $\lambda(n\ell, \overline{T})=\lambda(n, \overline{T}) -1 $.
	\item[\textbf{(ii)}] $\lambda(n\ell, \overline{T}^*)=\lambda(n, \overline{T}^*)-1$.
\end{enumerate}
\end{lemma}
We will include the proof of Lemma~\ref{417} despite following \cite{mr02} quite closely.

\begin{proof}
Since the difference $\lambda(r\,, \overline{T})-\lambda(r\,, \overline{T}^*)$ is independent of $r$ (by~\cite[Corollary 2.3.6]{mr02}), it suffices to check only (ii). There is a commutative diagram
$$\xymatrix{H^1_{\FFc(n)^*}(\QQ,T_{k,m}^*)[\mathcal{M}] \ar[r]^(.54){\textup{loc}_{\ell}} &H^1_{f}(\QQ_{\ell},T_{k,m}^*) [\mathcal{M}] \\
H^1_{\FFc(n)^*}(\QQ,\overline{T}^*) \ar[r]^(.54){\textup{loc}_{\ell}} \ar[u]^{\cong}&H^1_{f}(\QQ_{\ell},\overline{T}^*)\ar[u]^{\cong}
}
$$
where the left vertical isomorphism is~\cite[Lemma 3.5.3]{mr02}. By the assumption of the Lemma, the upper horizontal map is non-zero, therefore the lower horizontal map is non-zero as well. However, $H^1_{f}(\QQ_{\ell},\overline{T}^*)$ is a one-dimensional $\mathbb{F}$-vector space, hence the lower horizontal map is surjective. Similarly, one may prove that the localization map $H^1_{\FFc(n)}(\QQ,\overline{T}) \ra H^1_{f}(\QQ_{\ell},\overline{T})$ is surjective as well. By \cite[Lemma 4.1.7(ii)]{mr02} (which still holds since we are working with the same residual representation $\overline{T}=T_{k,m}/\mathcal{M}\cdot T_{k,m}$), we conclude that
$$H^1_{\FFc(n\ell)^*}(\QQ,\overline{T}^*) =H^1_{\FFc^{\ell}(n)^*}(\QQ,\overline{T}^*) \subset H^1_{\FFc(n)^*}(\QQ,\overline{T}^*).$$
 Thus, we have a short exact sequence
 $$0\lra H^1_{\FFc(n\ell)^*}(\QQ,\overline{T}^*) \lra H^1_{\FFc(n)^*}(\QQ,\overline{T}^*) \lra H^1_{f}(\QQ_{\ell},\overline{T}^*) \lra 0.$$ This, together with the fact that $H^1_{f}(\QQ_{\ell},\overline{T}^*)$ is a one dimensional $\mathbb{F}$-vector space (as $\ell \in \PP_{k,m}$) completes the proof of (ii).
\end{proof}

\begin{prop}
\label{vanish2}
Suppose that $\pmb{\kappa} \in \mathbf{KS}(T_{k,m},\FFc,\PP_{k,m})$ and $\kappa_n=0$ for any core vertex $n$, then $\pmb{\kappa}=0$.
\end{prop}

\begin{proof}
We need to show that $\kappa_r=0$ for every $r \in \NN_{k,m}$. We prove this by induction on $\lambda(r, \overline{T}^*)$ as in \cite{mr02}. By the assumption of Proposition, $\kappa_r=0$ for all $r$ with $\lambda(r, \overline{T}^*)=0$.

Suppose now that $\lambda(r, \overline{T}^*)>0$ and suppose $\kappa_r \neq 0$. Using\footnote{We make use of Lemma~\ref{361} as follows: By our assumption that $\lambda(r, \overline{T}^*)>0$ and using the isomorphism (\ref{comparison}) above along with~\cite[Lemma 3.5.3]{mr02}, it follows that
$$H^1_{\FFc(r)}(\QQ,T_{k,m})[\mathcal{M}]\neq0\neq H^1_{\FFc(r)^*}(\QQ,T_{k,m}^*)[\mathcal{M}].$$
Hence, there are non-zero classes
$$c\in H^1_{\FFc(r)}(\QQ,T_{k,m})[\mathcal{M}] \subset H^1(\QQ,T_{k,m}) \hbox{ and } c^* \in H^1_{\FFc(r)^*}(\QQ,T_{k,m}^*)[\mathcal{M}] \subset H^1(\QQ,T_{k,m}^*).$$
 Now given also that $0\neq\kappa_r\in H^1(\QQ,T_{k,m})$, we invoke Lemma~\ref{361} with non-zero $\kappa_r \hbox{ and } c \in H^1(\QQ,T_{k,m})$ and $c^* \in H^1(\QQ,T_{k,m}^*)$, in order  to find a prime $\ell$ with the desired properties enlisted above.} Lemma~\ref{361}, we may choose a prime $\ell \in \PP_{k,m}$ such that
\begin{enumerate}
	\item[(a)] $(\kappa_r)_{\ell} \neq 0$,
	\item[(b)] The maps
	$$\xymatrix@R=.08in{H^1_{\FFc(r)}(\QQ,T_{k,m})[\mathcal{M}] \ar[r]&H^1_{f}(\QQ_{\ell},T_{k,m}) \\
	H^1_{\FFc(r)^*}(\QQ,T_{k,m}^*)[\mathcal{M}] \ar[r]&H^1_{f}(\QQ_{\ell},T_{k,m}^*)
	}$$
	are both non-zero.
\end{enumerate}
Now Lemma~\ref{417} and (b) above imply that $\lambda(r\ell, \overline{T}^*) < \lambda(r, \overline{T}^*)$, therefore $\kappa_{r\ell}=0$  by the induction hypothesis. However, by the definition of a Kolyvagin system and by (a), it follows that $(\kappa_{r\ell})_{\ell,\textup{s}} \neq 0$. This contradicts with the induction hypothesis that $\kappa_{r\ell}=0$, which shows that $\kappa_r=0$ for all $r \in \NN_{k,m}$.\end{proof}

\begin{proof}[Proof of Theorem~\ref{thm:inj}]
Suppose that $\pmb{\kappa} \in \Gamma(\mathcal{H})$ and that $\kappa_r=0$ for the core vertex $r$ given in the statement of Theorem~\ref{thm:inj}. It follows from Proposition~\ref{vanish1} that $\kappa_n=0$ for any core vertex $n$, therefore $\pmb{\kappa}=0$ by Proposition~\ref{vanish2}.
\end{proof}

\subsection{Kolyvagin Systems for  $T \otimes \Lambda$}
\label{KS(Lambda)}
Using the results from \S\ref{KS1}, we prove (Theorem~\ref{main} below)
$$\overline{\mathbf{KS}}(T\otimes\Lambda,\FFc,\PP)\newnot{symbol:KS}:=\varprojlim_{k,m} \left(\varinjlim_{j} \mathbf{KS}(T_{k,m},\FFc, \PP_{j})\right)$$
 is a free $\Lambda$-module of rank one, under the hypotheses which we set in \S\ref{subsec:hypo}.

\begin{lemma}
\label{lem:j}
For any $j \geq k+m$, the natural restriction map
$$\mathbf{KS}(T_{k,m},\FFc, \PP_{k+m}) \lra \mathbf{KS}(T_{k,m},\FFc, \PP_{j})$$
 is an isomorphism.
\end{lemma}
\begin{proof}
By Corollary~\ref{KolSys} applied with a core vertex $r \in \NN_{j}$ (such $r$ exists by  \cite[Corollary 4.1.9]{mr02} and Theorem~\ref{thm1}), we have isomorphisms
$$\mathbf{KS}(T_{k,m},\FFc, \PP_{k+m}) \stackrel{\sim}{\lra} H^1_{\FFc(r)}(\QQ,T_{k,m}) \stackrel{\sim}{\longleftarrow}\mathbf{KS}(T_{k,m}, \FFc,\PP_{j})$$
 compatible with the restriction map $\mathbf{KS}(T_{k,m}, \FFc,\PP_{k+m}) \ra \mathbf{KS}(T_{k,m},\FFc, \PP_{j})$.
 \end{proof}

\begin{lemma}
\label{lem:surj}
The maps
$$\xymatrix @R=.1in{
H^1_{\FFc(r)}(\QQ,T_{k,m}) \ar[r]&  H^1_{\FFc(r)}(\QQ,T_{k^{\prime},m^{\prime}})\\
\mathbf{KS}(T_{k,m}, \FFc,\PP_{k+m}) \ar[r] & \mathbf{KS}(T_{k^{\prime},m^{\prime}}, \FFc,\PP_{k+m})
}
$$ are surjective for $k\geq k^{\prime}$, $m\geq m^{\prime}$ and for any core vertex $r \in \NN_{k,m}$.
\end{lemma}

\begin{proof}
We will first prove that the second map is surjective assuming the first map is. By Corollary~\ref{KolSys} and Lemma~\ref{lem:j} applied with a core vertex $r\in \NN_{k,m}$ to both $T_{k,m}$ and $T_{k^{\prime},m^{\prime}}$, we obtain the following commutative diagram with vertical isomorphisms:
$$\xymatrix{
\mathbf{KS}(T_{k,m},\FFc, \PP_{k+m}) \ar[r]\ar[d]^{\cong}&  \mathbf{KS}(T_{k^{\prime},m^{\prime}}, \FFc,\PP_{k+m})\ar[d]^{\cong}\\
 H^1_{\FFc(r)}(\QQ,T_{k,m}) \otimes G_{r} \ar @{->>}[r] & H^1_{\FFc(r)}(\QQ,T_{k^{\prime},m^{\prime}}) \otimes G_{r}
}$$
It now follows at once that the upper horizontal map in the diagram is surjective as well.

So it remains to prove that the map $$H^1_{\FFc(r)}(\QQ,T_{k,m}) \lra H^1_{\FFc(r)}(\QQ,T_{k^{\prime},m^{\prime}})$$ is surjective. We have the following commutative diagram, where the vertical isomorphism is obtained from an appropriate version of Lemma~\ref{lem:applycart}:$$\xymatrix @C=.9in @R=.4in{
H^1_{\FFc(r)}(\QQ,T_{k,m}) \ar[r]^{\textup{Reduction}} \ar[rd]^(.52){\varpi^{k-k^{\prime}}}&  H^1_{\FFc(r)}(\QQ,T_{k^{\prime},m})\ar[d]^(.45){[\varpi^{k-k^{\prime}}]}_{\cong}\\
& H^1_{\FFc(r)}(\QQ,T_{k,m})[p^{k^{\prime}}]
}
$$
Since $r\in \NN_{k,m}$ is a core vertex (therefore $H^1_{\FFc(r)}(\QQ,T_{k,m})$ is a free $R_{k,m}$-module of rank one), the map on the diagonal is surjective. This proves that the horizontal map is surjective as well. One shows in a similar manner that the map
$H^1_{\FFc(r)}(\QQ,T_{k^{\prime},m}) \ra H^1_{\FFc(r)}(\QQ,T_{k^{\prime},m^{\prime}})$ is also surjective. We therefore have a commutative diagram
$$\xymatrix @C=.9in @R=.4in{
H^1_{\FFc(r)}(\QQ,T_{k,m}) \ar @{->>} [r] \ar[rd] &  H^1_{\FFc(r)}(\QQ,T_{k^{\prime},m})\ar @{->>}[d]\\
& H^1_{\FFc(r)}(\QQ,T_{k^{\prime},m^{\prime}})
}$$ which shows that the map on the diagonal is surjective, and the proof is complete.
\end{proof}

\begin{thm}
\label{main}
Under our running hypotheses,
\begin{enumerate}
\item[(i)] the $\Lambda$-module $\overline{\mathbf{KS}}(T\otimes\Lambda,\FFc,\PP)$ is free of rank one,
\item[(ii) ]the specialization map $
\overline{\mathbf{KS}}(T\otimes\Lambda,\FFc,\PP) \ra \overline{\mathbf{KS}}(T,\FFc,\PP)$ is surjective.
\end{enumerate}
\end{thm}

\begin{proof}
By Lemma~\ref{lem:j}, it follows that $\varinjlim_j \mathbf{KS}(T_{k,m}, \FFc,\PP_{j})=\mathbf{KS}(T_{k,m},\FFc, \PP_{k+m}).$ Now using Corollary~\ref{KolSys} and Lemma~\ref{lem:surj}, the proof follows.
\end{proof}

\begin{rem}
\label{comparison-KS}
Notice that the collection of ideals $\{(\mm^k,\xx^m)\}_{k,m \in \ZZ^+}$ form a base of neighborhoods at $0 \in\Lambda$. Using this fact, one may see without difficulty that our definition above for the module of $\LL$-adic Kolyvagin systems $\overline{\mathbf{KS}}(T\otimes\Lambda,\FFc,\PP)$ agrees with Mazur and Rubin's definition~\cite[\S 5.3]{mr02} of \emph{generalized} module of Kolyvagin Systems $\overline{\mathbf{KS}}(T\otimes\Lambda,\FF_{\LL},\PP)$. We also recall the module $\KS(T\otimes\Lambda,\FF_{\LL},\PP)$; what Mazur and Rubin define in~\cite[Definition 3.1.3]{mr02} as the module of Kolyvagin systems. By definition, there are natural maps 
$$\KS(T\otimes\Lambda,\FFc,\PP) \lra \KS(T_{k,m},\FFc,\PP_{k,m})$$
for every $k,m \in \ZZ^+$. These in the limit give rise to a homomorphism of $\LL$-modules
\be\label{eqn:comparebarnobar}\KS(T\otimes\Lambda,\FFc,\PP) \lra \overline{\KS}(T\otimes\Lambda,\FFc,\PP).\ee
In this paper, we leave the question whether the map (\ref{eqn:comparebarnobar}) is an isomorphism or not aside, since, as far as the applications of the Kolyvagin system machinery is concerned (i.e., bounding Selmer groups), any element of $\KS(T\otimes\Lambda,\FFc,\PP)$ has the exact same use as an element of the module $\overline{\KS}(T\otimes\Lambda,\FFc,\PP)$. For this reason, when we say $\pmb{\kappa}$ is a $\LL$-adic Kolyvagin system, we mean (by slight abuse) that $\pmb{\kappa}$ is an element of one of the three modules of Kolyvagin systems discussed above.
\end{rem}

\begin{rem}
\label{hopes}
Theorem~\ref{main} says, under the running hypotheses, that $\overline{\KS}(T\otimes\LL,\FFc,\PP)$ is generated by a Kolyvagin system $\pmb{\kappa}$ whose \emph{blind spot} (see~\cite[Definition 3.1.6]{mr02} for a definition) does not contain any prime ideal of $\LL$. Even if \htam\, and \hsez\, fails, but \hez\, holds for $T$, we still expect to have a $\LL$-adic Kolyvagin system for the triple $(T\otimes\LL,\FFc,\PP)$ which does not have the ideal $(\gamma-1)$ in its blind spot. However, we no longer expect that the specialization map
$$\overline{\KS}(T\otimes\LL,\FFc,\PP) \lra \overline{\KS}(T,\FFc,\PP)$$
 to be surjective in this case. Although we are unable to formulate a precise conjecture for the size of the cokernel of this map, we expect that its size will be related to $\#(T^*)^{G_{\QQ_p}}$ and $\#A^{\mathcal{I}_{\ell}}/(A^{\mathcal{I}_{\ell}})_{\textup{div}}$ (for $\ell\neq p$). In fact, using the proof of \cite[Proposition 6.2.6]{mr02}, one can show that
 $$p \mid \#\textup{coker}\left\{\overline{\KS}(T\otimes\LL,\FFc,\PP) \lra \overline{\KS}(T,\FFc,\PP)\right\}$$
  if one of $\#(T^*)^{G_{\QQ_p}}$ and $\#A^{\mathcal{I}_{\ell}}/(A^{\mathcal{I}_{\ell}})_{\textup{div}}$ is greater than one. Furthermore, the author proves in~\cite{kbbtamagawa} that if $p^\alpha|\#A^{\mathcal{I}_{\ell}}/(A^{\mathcal{I}_{\ell}})_{\textup{div}}$, then
$$p^\alpha \mid \#\textup{coker}\left\{\overline{\KS}(T\otimes\LL,\FFc,\PP) \lra \overline{\KS}(T,\FFc,\PP)\right\}.$$
\end{rem}

\section{Applications and Further Discussion}
\label{applications}
In this section, we exhibit several arithmetic implications of Theorem~\ref{main}. First in \S\ref{ex-classical}, we go over well-known examples of Kolyvagin Systems over the cyclotomic Iwasawa algebra and their applications, such as the Kolyvagin Systems derived from cyclotomic unit Euler system and Kato's Euler system. Theorem~\ref{main} does not prove anything new here, except for the fact that the cyclotomic unit Kolyvagin system (resp., Kato's Kolyvagin system) over the cyclotomic Iwasawa algebra lies inside the free rank-one $\LL$-module of $\LL$-adic Kolyvagin systems.

Next in \S\ref{p-adicL-func}, we discuss the relation of the cohomology classes we prove to exist in Theorem~\ref{main} with $p$-adic $L$-functions. Although our construction of  these classes has no reference to the $L$-values, the rigidity of the $\LL$-adic Kolyvagin Systems (i.e.,Theorem~\ref{main}(i)) enables us to set up a connection with $p$-adic $L$-functions via Perrin-Riou and Rubin's (conjectural) Euler systems ({c.f.},~\cite[\S8]{r00}).

At the end in \S\ref{sec:stark-iwasawa}, we apply our results to study the \emph{Rubin-Stark elements}~\cite{ru96} along the cyclotomic $\ZZ_p$-tower. This is the only section that the base field is different from $\QQ$, and we will utilize a version of Theorem~\ref{main} with a Selmer structure different from $\FFc$. We explain how to use these classes (that we prove to exist, and that also independently come (thanks to~\cite{kbb-iwasawa}) from the conjectural Rubin-Stark elements which we assume that they exist) to study the Iwasawa theory of totally real fields. Wiles~\cite{wiles-mainconj} proved the main conjecture in this setting. In \S\ref{sec:stark-iwasawa}, we sketch a strategy (that is developed in full in~\cite{kbb-iwasawa}) to give a proof (conditional on the Rubin-Stark conjectures) of Iwasawa's main conjecture for totally real fields. Along the way, we obtain a result on the local Iwasawa theory of Rubin-Stark elements (see Theorem~\ref{conj-kbb} below). 

 See also~\cite{kbb-stick} for an important arithmetic application of our rigidity result (Theorem~\ref{main}(i)) which we do not include here. In loc.cit., the author  establishes a connection between the \emph{Stickelberger elements} and Rubin-Stark elements using Theorem~\ref{main}(i). More precisely, let $k$ be a totally real field and
let $\chi: G_k \ra \oo^\times$ be a totally odd character of finite
prime-to-$p$ order. Let $\rho_{\textup{cyc}}: G_k \ra \ZZ_p^\times$ be
the cyclotomic character. In~\cite{kbb-stick}, the author constructs a
$\LL$-adic Kolyvagin system for the representation $\oo(\chi)$ (a
similar, yet different construction was given by Kurihara~\cite{kurihara}
prior to this work). Using a formal twisting argument, this Kolyvagin
system gives rise to a Kolyvagin system for the Galois representation
$\rho_{\textup{cyc}}\omega^{-1}\otimes \oo(\chi)=\oo(1) \otimes
\oo(\chi\omega^{-1})$. For the Galois representation $\oo(1) \otimes
\oo(\chi\omega^{-1})$, the author used in~\cite{kbb-iwasawa} the (conjectural)
Rubin-Stark elements to construct a $\LL$-adic Kolyvagin system.
Thanks to Theorem 3.23(i) (slightly enhanced as Theorem 4.10(i)
below), these two $\LL$-adic Kolyvagin systems should therefore differ
from each other at most by multiplication by an element of $\LL$, as
they both live in a free $\LL$-module of rank one. Note that the existence of the Stickelberger element Kolyvagin system constructed in~\cite{kbb-stick} relies only on a special case of Brumer's conjecture, see~\cite{wiles-brumer,kurihara, greither} for a proof of Brumer's conjecture in various cases.
\subsection{Classical Examples}
\label{ex-classical}
Fix once and for all a rational prime $p>2$. In this section we do not claim any new results, except in Proposition~\ref{prop:cycloprimitive} (resp., Proposition~\ref{prop:katoprimitive}), we explain that the Kolyvagin system of cyclotomic units (resp., Kato's Kolyvagin system) generate the cyclic module of $\LL$-adic Kolyvagin systems, under certain hypotheses.
\subsubsection{Cyclotomic Units}
\label{subsub:cyclo}
Let $T=\oo(1)\otimes\rho^{-1}$, where $\rho:G_{\QQ}\ra\oo^{\times}$ is an even character of finite prime-to-$p$ order. Let $L$ be the cyclic extension of $\QQ$ which is cut out by the character $\rho$ and set $\Delta=\textup{Gal}(L/\QQ)$. For any $\oo[\Delta]$-module $A$, let $A^{\rho}$ denote the $\rho$-isotypic component of $A$.

Fix a collection $\{\zeta_n: n \in \ZZ^+\}$ such that $\zeta_n$ is a primitive $n$-th root of unity and $\zeta_{mn}^m=\zeta_n$ for every $m \in \ZZ^+$. Whenever $F$ is a finite abelian extension of $\QQ$ of conductor $f$, we define $c_F$ as the image of $\mathbf{N}_{\QQ(\mu_{fp})/F}\left(\zeta_{fp}-1\right)\in F^\times$ under the Kummer map $F^{\times} \hookrightarrow H^1(F,\oo(1))$. The collections $\{c_F\}_{F}$ is an Euler system in the sense of~\cite[\S II]{r00} for the representation $\oo(1)$, and is called the \emph{cyclotomic unit Euler system}. One can modify these units, as in~\cite[\S IX.6]{r00}, to obtain an Euler system $\mathbf{c}$ for $\oo(1)$ satisfying~\cite[Definition 3.2.2]{mr02}.

By a standard twisting argument ({c.f.}, \cite[Proposition II.4.2]{r00}), the Euler system $\mathbf{c}$ gives an Euler system $\mathbf{c}^{\rho}$\newnot{symbol:crho} for the representation $T=\oo(1)\otimes\rho^{-1}$. By~\cite[Theorem V.3.3]{mr02}, the Euler system $\mathbf{c}^{\rho}$ gives rise to  a $\LL$-adic Kolyvagin system $\pmb{\kappa}^{\rho,\infty}\newnot{symbol:krhoinfty} \in \overline{\KS}(T\otimes\LL,\FFc,\PP)$.

Assume now that $\rho(p)\neq1$ and assume for simplicity that $\rho$ is unramified at $p$. Then it is straightforward to check that $T$ satisfies our hypotheses \hone-\hfour, \htam\, and \hsez; thus Theorem~\ref{main} applies to conclude that the $\LL$-module $\overline{\KS}(T\otimes\LL,\FFc,\PP)$ is free of rank one.  
\begin{prop}
\label{prop:cycloprimitive}
The $\LL$-adic Kolyvagin system $\pmb{\kappa}^{\rho,\infty}$ is $\LL$-primitive \textup{(}in the sense of~\cite[Definition 5.3.9]{mr02}\textup{)} and it generates the free $\LL$-module $\overline{\KS}(T\otimes\LL,\FFc,\PP)$.
\end{prop}
\begin{proof}
For any sub-quotient $M$ of $T\otimes\LL$, we simply write $\overline{\KS}(M)$ instead of $\overline{\KS}(M,\FFc,\PP)$. Let $\pmb{\kappa}^{\rho}$\newnot{symbol:krho} denote the image of $\pmb{\kappa}^{\rho,\infty}$ under the specialization map $\overline{\KS}(T\otimes\LL) \ra \overline{\KS}(T)$. By~\cite[Remark 6.1.8]{mr02}, the Kolyvagin system $\pmb{\kappa}^{\rho}$ is primitive, i.e., its image $\overline{\pmb{\kappa}^{\rho}}$ under the map
$$\overline{\KS}(T) \lra {\KS}(T/\mm T)$$
 is non-zero. This proves that the image of $\pmb{\kappa}^{\rho,\infty}$  under the map 
$\overline{\KS}({T\otimes\LL}) \ra \overline{\KS}({T}\otimes\LL/\frak{p})$
 is non-zero for any height-one prime $\frak{p} \subset \LL$; as we have a commutative diagram
 $$\xymatrix@C=.05pt@R=.5pt{\pmb{\kappa}^{\rho,\infty}\ar@{{|}->}[dd] &\in& \overline{\KS}({T}\otimes\LL)\ar[dd]\ar[rrrrd]&&&&\\
 &&&&&&\overline{\KS}({T}\otimes\LL/\frak{p})\ar[lllld]\\
 \overline{\pmb{\kappa}^{\rho}}&\in&{\KS}({T}/\mm T)&&&& 
 }$$ and $\overline{\pmb{\kappa}^{\rho}}\neq0$.
\end{proof}
See \cite[Theorem 5.3.10]{mr02} for the standard application of the $\LL$-adic Kolyvagin system $\pmb{\kappa}^{\rho,\infty}$. 
\subsubsection{Kato's Euler System}
\label{subsub:kato}
Fix an elliptic curve $E_{/\QQ}$,  assume that $p>3$ and $\oo=\ZZ_p$. Let $T=T_p(E)$ be the $p$-adic Tate module. Suppose also that the $p\hbox{-adic}$ representation 
$$\rho_E:G_{\QQ} \ra \textup{Aut}(E[p^{\infty}])\cong\textup{GL}_2(\ZZ_p)$$
 is surjective. Let $N$ be the conductor of $E$, and let $N_1$ (resp., $N_2$) be the product of primes $\ell|N$ such that $E$ has split (resp., non-split) multiplicative reduction at $\ell$. Kato~\cite{ka1} has constructed an Euler system for $(T_p(E),\PP^{\prime})$, where $\PP^{\prime}$ is the set of rational primes not dividing $NpDD^{\prime}$ with two auxiliary positive integers $D$ and $D^{\prime}$ used in Kato's construction. Set $\PP^{(0)}=\PP^{\prime}\cap\PP$ (recall the definition of $\PP$ from \S\ref{primes} above) and $\PP^{(0)}_{k,m}=\PP^{\prime}\cap\PP_{k,m}$. Using \cite[Theorem 5.3.3]{mr02}, Kato's Euler system gives rise to a $\LL$-adic Kolyvagin system $\pmb{\kappa}^{\textup{Kato},\infty}\newnot{symbol:katoinfty} \in \overline{\KS}(T\otimes\LL,\FFc,\PP^{(0)})$. Let $\pmb{\kappa}^{\textup{Kato}}$\newnot{symbol:kato} denote the image of $\pmb{\kappa}^{\textup{Kato},\infty}$ under the map $\overline{\KS}(T\otimes\LL,\FFc,\PP^{(0)}) \ra \overline{\KS}(T,\FFc,\PP^{(0)})$.

\begin{prop}
\label{prop:katoprimitive}
Assume that $E$ has good reduction at $p$, and also that,
\begin{itemize}
\item[(E0)] $L(E,1)\neq0$, where $L(E,s)$ is the Hasse-Weil $L$-function attached to $E$,
\item[(E1)] \hsez\, holds: $E(\QQ_p)[p]=0$,
\item[(E2)] \htam\, holds: $p$ does not divide the Tamagawa factors $c_{\ell}$ 
 for any prime $\ell\neq p$, 
\item[(E3)] $\displaystyle p\nmid \prod_{\ell|N_1}(\ell-1)\prod_{\ell|N_2}(\ell+1)$,
\item[(E4)] the $p$-part of the Birch and Swinnerton-Dyer conjecture holds for $E$.
\end{itemize}
Then $\pmb{\kappa}^{\textup{Kato},\infty}$ is $\LL$-primitive and generates the free $\LL$-module $\overline{\KS}(T\otimes\LL,\FFc,\PP^{(0)})$.
\end{prop}
\begin{proof}
Let $L_N (E, s)$ be the $L$-function with the Euler factors at primes dividing the conductor $N$ of $E$ removed, and let $\Omega$ be a fundamental period. As in the proof of Proposition~\ref{prop:cycloprimitive}, the assertion that $\pmb{\kappa}^{\textup{Kato},\infty}$ is $\LL$-primitive will follow once we verify that the Kolyvagin system $\pmb{\kappa}^{\textup{Kato}}$ is primitive. This, however, follows at once from \cite[Theorem~6.2.4 and Corollary~5.2.13(ii)]{mr02}, along with the observation that we have 
$$\textup{ord}_p\left(\frac{L(E,1)}{\Omega}\right)=\textup{ord}_p \left(\frac{L_N(E,1)}{\Omega}\right)$$
 thanks to (E3). 
 
 We also note that,  thanks to our assumptions on $\rho_E$, the hypotheses \hone-\hthree\, hold for $T=T_p(E)$, and \hfour\, holds since we assumed $p>3$. Furthermore,  the core Selmer rank $\chi(T,\FFc)$ of the Selmer structure $\FFc$ on $T$ equals to $\textup{rank}_{\ZZ_p}\,T^{-}=1$, where $T^-$ is the $-1$-eigenspace for any complex conjugation.  Hence, thanks to (E1) and (E2), Theorem~\ref{main} applies to conclude that the $\LL$-module $\overline{\KS}(T\otimes\LL,\FFc,\PP^{(0)})$ is free of rank one. 
\end{proof}


\begin{rem}
\label{rem:e12fails}
\begin{itemize}
\item[(i)] We note that the assumption that $E$ has good reduction at $p$ is not necessary for Theorem~\ref{main} to hold; we only need this assumption to use Kato's calculations with his Euler system: We prove in Theorem~\ref{main} above that the $\LL$-module $\overline{\KS}(T\otimes\LL,\FFc,\PP^{(0)})$ is free of rank one even when only (E1) and (E2) holds.
\item[(ii)] If (E1) or (E2) fails, we do not expect the $\LL$-adic Kolyvagin system $\pmb{\kappa}^{\textup{Kato},\infty}$ to be $\LL$-primitive; see \cite[Proposition 6.2.6]{mr02} and \cite{kbbtamagawa}.
\end{itemize}
\end{rem}
\subsection{$\Lambda$-adic Kolyvagin Systems and $\lowercase{p}$-adic $L$-functions}
\label{p-adicL-func}
Examples we discuss in the previous section suggests that the $\LL$-adic Kolyvagin Systems we prove to exist should relate to the $p$-adic $L$-functions. In this section, we explore this in much greater generality and establish a link between our $\LL$-adic Kolyvagin systems and Perrin-Riou's conjectural $p$-adic $L$-functions.

To motivate, we revisit the case $T=\ZZ_p(1)$ and the Euler system of cyclotomic units. See~\cite[\S VIII.5]{r00} or~\cite{pr-kubota} for a detailed discussion of  what follows in this paragraph. Let $U_n$ be the local units inside $\QQ_{p,n}$, the unique extension of $\QQ_p$ of degree $p^n$ inside the cyclotomic $\ZZ_p$-extension of $\QQ_p$; and let $\mathcal{O}_n$ be the ring of integers of $\QQ_{p,n}$. There is an isomorphism
$$\exp^*:\,H^1_{\textup{Iw}}(\QQ_p,T) \stackrel{\sim}{\lra} D_{W}(T)^{\psi=1}$$
 (see Theorem~A.\ref{cohomology-iwasawa-p-tilde} and Theorem~A.\ref{cohomology-iwasawa-p} below), where $W$ is the prime-to-$p$ part of $\textup{Gal}(\QQ_p(\mu_p^{\infty})/\QQ_p)$; and $D_W(T)^{\psi=1}$ is the $W$-fixed part of $D(T)^{\psi=1}$, and finally, $D(T)$ is Fontaine's $(\varphi,\Gamma)$-module attached to $T$. We then have a commutative diagram (which is essentially  the explicit reciprocity law of Iwasawa~\cite{iwasawa-2} in this setting)
$$\xymatrix{\varprojlim_n U_n \ar[rr]^{\frak{K}}\ar[rd]_{u \mapsto\frac{\partial f_u}{f_u}(\pmb{\pi})}&&H^1_{\textup{Iw}}(\QQ_p,T)\ar[ld]^{\exp^*}\\
&D_W(T)^{\psi=1}&
}$$
Here, $f_u$ is the Coleman's power series attached to $u \in \varprojlim_n U_n$; $\pmb{\pi}=\{\pi_n\} \in \varprojlim_n \mathcal{O}_n$ and $\pi_n$ is a distinguished uniformizer of $\mathcal{O}_n$ (which we will not define here); and $\frak{K}$ is the Kummer map. It turns out that the image of $\{\mathbf{c}^{\textup{cycl}}_{\QQ_n}\}_{n} \in \varprojlim_n U_n$ under $\textup{exp}^*\circ \frak{K}$, where $\mathbf{c}^{\textup{cycl}}$ is the cyclotomic unit of Euler system, is the measure whose Amice transform gives rise to the Kubota-Leopoldt $p$-adic zeta function.

This example is part of a big conjectural picture (due to Perrin-Riou~\cite{pr2}). For a general $\ZZ_p[[G_{\QQ}]]$-representation $T$, we conjecturally have the following diagram:
\[
\begin{array}{cll}
  \left\{
\begin{array}{cll}
  \textup{Leading terms of $\LL$-adic}\\
 \textup{Kolyvagin Systems for } T\otimes\LL
\end{array}
\right\}\lra&H^1(\QQ,T\otimes\LL)
  \stackrel{\textup{exp}^*}{\lra}D_{W}(T)^{\psi=1}\stackrel{\frak{A}}{\lra}\left\{\begin{array}{c} p\textup{-adic }\\ L\textup{-functions} \end{array}\right\}
\end{array}
\]
Here, $\frak{A}$ is the Amice transform and the leading term of a $\LL$-adic Kolyvagin system $\pmb{\kappa}^{\infty}$ is the term $\kappa^{\infty}_1 \in \varprojlim _{k,m}H^1(\QQ,T_{k,m})\cong H^1(\QQ,T\otimes\LL)$.

In~\cite[\S VIII]{r00}, Rubin proposes to go the other way in the diagram above: To construct an Euler system for $T$ starting from a collection of $p$-adic $L$-functions associated to $T$ over various abelian extensions of $\QQ$. In what follows, we will closely follow Rubin's exposition from~\cite{r00}.

Suppose until the end of this section that 
\begin{itemize}
\item $T$ is a $G_{\QQ}$-stable lattice inside the $p$-adic realization $V=T\otimes \QQ_p$ of a \emph{motivic structure} (in the sense of~\cite[\S III]{FPR91}) and that $V$ is crystalline at $p$,
\item $\oo=\ZZ_p$.
\end{itemize}
 Assume further that the \emph{generic core Selmer rank} of $\FFc$ on $T\otimes\LL$ is one. Throughout this section, a different ring of periods will be utilized than the one mentioned in Appendix~\ref{fontaine}, which was also constructed by Fontaine, called the crystalline period ring and denoted by $\mathbb{B}_{\textup{cris}}$. 

 Let $D_{\textup{cris}}(V)=(\mathbb{B}_{\textup{cris}}\otimes_{\QQ_p}V)^{G_{\QQ_p}}$ denote Fontaine's filtered vector space attached to $V$ and let $V^*=\textup{Hom}(V,\QQ_p(1))$. If $F$ is an abelian extension of $\QQ$ which is unramified at $p$, we also let $D_{\textup{cris}}^F(V)=D_{\textup{cris}}(\textup{Ind}_{F/\QQ}V)$. Let $\mathbb{H}$ denote the extended Iwasawa algebra defined by Perrin-Riou~\cite[\S1]{pr}, and let $\mathbb{K}=\textup{Frac}(\mathbb{H})$ denote the field of fractions of $\mathbb{H}$. In~\cite{pr}, Perrin-Riou constructs a $\ZZ_p[[\textup{Gal}(F\QQ_{\infty})/\QQ]]$-module homomorphism (what she calls an expanded logarithm, and which is closely related to $\exp^*$ of Theorem~A.\ref{cohomology-iwasawa-p-tilde} and Theorem~A.\ref{cohomology-iwasawa-p})
 $$\mathcal{L}_F:\,\bigoplus_{v\mid p} \varprojlim_{n}H^1((F\QQ_n)_v,T)\lra \mathbb{K}\otimes D_{\textup{cris}}^F(V).$$
 Furthermore, Perrin-Riou and Rubin conjecture that for every extension $E$  of $\QQ_p$ and every character $\chi:G_{\QQ}\ra E^{\times}$ of finite order, unramified at $p$, and $r \in \ZZ^+$ that is divisible by the conductor of $\chi$, there is an element
 $$\mathbf{L}^{(p)}_r(T\otimes\chi) \in \mathbb{K}\otimes D_{\textup{cris}}(V^*\otimes\chi^{-1})$$
  which is characterized by a certain interpolation property ({c.f.},~\cite[\S4.2]{pr2}). We will not give any details on $\mathbf{L}^{(p)}_r(T\otimes\chi)$ and refer the reader to~\cite{pr2}, except for the following rough version of the interpolation property which it conjecturally satisfies for characters $\rho$  of $\Gamma$ of finite order and for sufficiently large positive integers $k$:
\[
\begin{array}{c}
  \langle\chi_{\textup{cycl}}\rangle^k\rho(\mathbf{L}^{(p)}_r(T\otimes\chi))= (\hbox{$p$-Euler factor})\times \frac{L_r(V\otimes\chi\omega^k\rho^{-1},-k)}{\textup{archimedean period}}\times(\hbox{$p$-adic period}).
\end{array}
\]
Here, $\omega$ is the Teichm\"uller character and $\langle\chi_{\textup{cycl}}\rangle$ is the character given by
$$\langle\chi_{\textup{cycl}}\rangle:=\omega^{-1}\chi_{\textup{cycl}}:G_{\QQ}\twoheadrightarrow \Gamma=\textup{Gal}(\QQ_{\infty}/\QQ),$$
and $L_r(V\otimes\chi\omega^k\rho^{-1},s)$ is the conjectural complex $L$-function of $V\otimes\chi\omega^k\rho^{-1}$ with the Euler factors at primes dividing $r$ removed.

Let $\QQ(\mu_{r})^+$ be the maximal real subfield of $\QQ(\mu_r)$ and write $\Delta_r=\textup{Gal}(\QQ(\mu_{r})^+/\QQ)$. For every character $\chi:\Delta_r \ra E^{\times}$ (where $E$ is a finite extension of $\QQ_p$ as above), let $\epsilon_{\chi} \in E[\Delta_r]$ be the idempotent associated with $\chi$.

For $f \in \mathbb{K}$, let $f^{\iota}$ denote the image of $f$ under the involution induced by $\gamma\mapsto\gamma^{-1}$ for $\gamma \in\textup{Gal}(\QQ_{\infty}(\mu_r)^+/\QQ)$. Define also $ H^1_{\infty}(\QQ(\mu_r)^+,T):=\varprojlim_n H^1(\QQ_n(\mu_r)^+,T)$.

\begin{conj}\textup{\cite[Conjecture VIII.2.6]{r00}}\label{rubin-conj}
 Assume $r \in \ZZ^+$ is prime to $p$, and $T$ is as above. Suppose $\alpha\in\ZZ_p[[G_{\QQ}]]$ annihilates $H^0(\QQ_{\infty}(\mu_r),A)$.

 Then there is an element $\mathbf{\xiv}_r\newnot{symbol:xi}=\mathbf{\xiv}_r^{(\alpha)} \in H^1_{\infty}(\QQ(\mu_r)^+,T)$ such that for every character $\chi$ as above
 $$\epsilon_{\chi}\mathcal{L}_{\QQ(\mu_r)^+}(\mathbf{\xiv}_r)=\chi(\alpha)\mathbf{L}_r^{(p)}(T^*\otimes\chi)^{\iota},$$
 where $\chi(\alpha)$ is the image of $\alpha$ under the composition
 $$\ZZ_p[[G_Q]]\lra\LL\otimes\ZZ_p[\Delta_r] \stackrel{1\otimes\chi}{\lra}\LL\otimes E \lra \mathbb{K}\otimes E.$$
 \end{conj}

 Let $N$ be the product of all rational primes where $T$ is ramified. It follows from~\cite[ Lemma IV.2.5(i)]{r00} that
 $$H^0(\QQ_{\infty}(\mu_r),T)=H^0(\QQ_{\infty},T)\,\, \hbox{    and   }\,\,H^0(\QQ_{\infty}(\mu_r),A)=H^0(\QQ_{\infty},A), $$
  if $r$ is prime to $Np$. Fix an element $\alpha \in \ZZ_p[[G_{\QQ}]]$ which annihilates $H^0(\QQ_{\infty},A)$ (therefore also $H^0(\QQ_{\infty}(\mu_r),A)$ for every $r \in \ZZ^+$ prime to $Np$). Assume that the weak Leopoldt conjecture ({c.f.},~\cite[\S1.3]{pr2}) holds for $T^*$, and $H^0(\QQ_{\infty},T)=0$.

 For every $r \in \ZZ^+$ prime to $Np$, let 
 $$\xiv_r=\{\xi_{n,r}\} \in H^1_{\infty}(\QQ(\mu_r)^+,T)\hbox{ with }\xi_{n,r} \in  H^1(\QQ_n(\mu_r)^+,T)$$ be an element that satisfies Rubin's conjecture above. Then Rubin shows:
 \begin{prop}\textup{\cite[Corollary VIII.3.2]{r00}}\label{rubin-euler}
 The collection $$\{\xi_{n,r} \in H^1(\QQ_n(\mu_r)^+,T):n\geq0,\hbox{ $r$ prime to $Np$}\}$$
  is an Euler system for $(T,\QQ_{\infty}\QQ^{\textup{ab},Np,+},Np)$ in the sense of \cite[Definition II.1.1 and Remark II.1.3]{r00}, where $\QQ^{\textup{ab},Np,+}$ is the maximal real subfield of the maximal abelian extension of $\QQ$ unramified outside $Np$.
 \end{prop}
\begin{rem}
 Assume in this Remark that $T$ satisfies the hypotheses \hone-\hfour, \htam\, and \hsez, in addition to the hypotheses of of the Conjecture above. We also assume that $\chi(T,\FFc)=1$. It therefore follows from Theorem~\ref{main} that the $\LL$-module $\overline{\KS}(T\otimes\LL,\FFc,\PP)$ is free of rank one, and in fact it is generated by a $\LL$-primitive Kolyvagin system $\pmb{\kappa}^{\infty}$. In this paragraph, we explain how $\pmb{\kappa}^{\infty}$ we prove to exist relates to the conjectural collection $\{\xi_{n,r}\}$ of Perrin-Riou and Rubin. Theorem 5.3.3 of~\cite{mr02} applied on the conjectural Euler system $\{\xi_{n,r}\}$ of Proposition~\ref{rubin-euler} gives rise to a $\LL$-adic Kolyvagin system $\pmb{\kappa}^{\textup{PR}}\newnot{symbol:PR} \in \overline{\KS}(T\otimes\LL,\FFc,\PP)$. As $\pmb{\kappa}^{\infty}$ generates the $\LL$-module $\overline{\KS}(T\otimes\LL,\FFc,\PP)$, there is a $\lambda \in \LL$ \symbolfootnote[2]{Under the assumptions of Theorem~\ref{main}, we expect $\pmb{\kappa}^{\textup{PR}}$ to be also $\LL$-primitive (just as the cyclotomic unit Kolyvagin system $\pmb{\kappa}^{\rho,\infty}$ in~\S\ref{subsub:cyclo} and Kato's Kolyvagin system $\pmb{\kappa}^{\textup{Kato},\infty}$ in~\S\ref{subsub:kato} above are, under suitable assumptions), so we in fact expect that $\lambda \in \LL^\times$ under the running assumptions.} such that 
 $$\pmb{\kappa}^{\textup{PR}}=\lambda \cdot \pmb{\kappa}^{\infty}.$$ 
 This means the $\LL$-adic Kolyvagin system $\pmb{\kappa}^{\infty}$ that we prove to exist in Theorem~\ref{main} is linked to the  $p$-adic $L$-functions, as it is related to the conjectural classes of Perrin-Riou and Rubin.

 We also note that Theorem~\ref{main} proves a consequence of the conjectures of Perrin-Riou and Rubin we mentioned above, namely the existence of $\LL$-adic Kolyvagin Systems. In this sense, these conjectures borrow evidence (although quite weak) from Theorem~\ref{main}.
\end{rem}

\subsection{Iwasawa theory of Rubin-Stark elements and Kolyvagin Systems}
\label{sec:stark-iwasawa}
Let $F$\newnot{symbol:F} be a totally real number field of degree $r$ over $\QQ$ and fix an algebraic closure $\overline{F}$ of $F$. For any rational prime $\ell$ and a $G_F$-module $M$, write $H^{*}(F_\ell,M)$ for the semi-local cohomology group at $\ell$.

In this section, our base field will be different from $\QQ$, however our results from Sections~\ref{sec:Core} and~\ref{chptr:KS} will apply (almost) verbatim.

Let $\rho$ be a totally even character of $G_F$ (i.e., $\rho$ is trivial on all complex conjugations inside $G_F$) into $\oo^{\times}$ that has finite prime-to-$p$ order, and let $L$ be the fixed field of $\hbox{ker}(\rho)$.  We define $f_{\rho}$ to be the conductor of $\rho$, and set $\Delta:=\hbox{Gal}(L/F)$. We assume that $(p, f_{\chi})=1$ and $\rho(\wp)\neq1$ for any prime $\wp$ of $F$ lying above $p$. Assume also for notational simplicity that $p$ is unramified in $F/\QQ$. Let $T$ be  the $G_F$-representation $\oo(1)\otimes\rho^{-1}$.
\begin{rem}
\label{hypo holds}
$T$ clearly satisfies the hypotheses \hone-\hfour\, ($\QQ$ replaced by $F$). Observe that the following versions of the hypotheses \htam\, and \hsez\, also hold for $T$:

(\htamF)\,\,\,\,\,\,\,\,\,\,\,\,\, $A ^{\mathcal{I}_{F_{\lambda}}}$ is divisible for every $\lambda \nmid p$.

(\hsezF)\,\,\,\,\, $H^0(F_{\wp},T^*) =0$ for  $\wp\mid p$.
\end{rem}
Let $F_{\infty}$ denote the cyclotomic $\ZZ_p$-extension of $F$, and $\Gamma=\Gal(F_{\infty}/F)$. Since we assumed  $p$ is unramified in $F/\QQ$, note that $F_{\infty}/F$ is then totally ramified at all primes $\wp$ over $p$. Let $F_{\wp}$ denote the completion of $F$ at $\wp$, and let  $F_{\wp,\infty}$ denote the cyclotomic $\ZZ_p$-extension of $F_{\wp}$. We may identify $\Gal(F_{\wp,\infty}/F_{\wp})$ with $\Gamma$ for all $\wp|p$ and henceforth $\Gamma$ will stand for any of these Galois groups. Let $\LL=\oo[[\Gamma]]$ as usual.
\subsubsection{Modified Selmer structures on $T\otimes\LL$}
\label{modified selmer}
In~\cite{kbbstark}, the author modifies the classical local conditions at $p$ in order to obtain a Selmer structure $\FF_{\mathcal{L}}$ on $T$ (see~\cite[\S1]{kbbstark}). The objective of this section is to lift the Selmer structure $\FF_{\mathcal{L}}$ to a Selmer structure on $T\otimes\LL$.

Write $H^1_{\textup{Iw}}(F_{\wp},T):=\varprojlim_n H^1(F_{\wp,n},T)$, where $F_{\wp,n}$ denotes the unique subfield of $F_{\wp,\infty}$ which has degree $p^n$ over $F_{\wp}$. Using  Shapiro's Lemma, one may canonically identify $H^i_{\textup{Iw}}(F_{\wp},T)$ with $H^i(F_{\wp},T\otimes\LL)$  for all $i \in \ZZ^+$ ({c.f.}, ~\cite[Lemma 5.3.1]{mr02}). Let
$$H^i_{\textup{Iw}}(F_{p},T):=\oplus_{\wp|p}H^i_{\textup{Iw}}(F_{\wp},T) \,\hbox{  and  }\, H^i(F_{p},T\otimes\LL):=\oplus_{\wp|p} H^i(F_{\wp},T\otimes\LL).$$

As in Appendix~\ref{fontaine}, set $H_K:=\Gal(\overline{K}/K_{\infty})$ for any local field $K$. Recall that $T_m=T\otimes\LL/(\xx^m)$.
\begin{prop}
\label{structure-at-p}
Suppose $\rho(\wp)\neq1$ for any prime $\wp$ of $F$ above $p$. Then:
\begin{enumerate}
\item[\textbf{(i)}] The $\LL$-module $H^1(F_{p},T\otimes\LL)=H^1_{\textup{Iw}}(F_{p},T)$ is free of rank $r$.
\item[\textbf{(ii)}] The map $H^1(F_{p},T\otimes\LL) \ra H^1(F_{p},T_m)$ is surjective for every $m \in \ZZ^+$.
\end{enumerate}
\end{prop}
\begin{proof}
Since we assumed $\rho(\wp)\neq1$ for any $\wp|p$, it follows that $H^0(H_{F_{\wp}}, T)=0$, and thus (i) follows  from Theorem~A.\ref{cohomology-iwasawa-p}. By our assumption on $\rho$, it follows that $H^0(F_{\wp},T^*)=0$, hence, by the proof of Lemma~\ref{h2-vanishing}, also that $H^2(F_{\wp},T\otimes\LL)=0$. But
$$\textup{coker}\{H^1(F_{p},T\otimes\LL) \ra H^1(F_{p},T_m)\}=H^2(F_{p},T\otimes\LL)[\xx^m],$$
  hence (ii) follows.
\end{proof}

We recall the definition of the Selmer structure $\FF_{\LL}$ of~\cite[ Definition 5.3.2]{mr02}. We fix a set $\Sigma(\FF_{\LL})$ of places of $F$ which contains all the places above $p$ and infinite places of $F$, as well as the places at which $T$ is ramified. We then let  $H^1_{\FF_{\LL}}(F_{v},T\otimes\LL)=H^1(F_{v},T\otimes\LL)$  for all $v\in\Sigma(\FF_{\LL})$. By \cite[Lemma 5.3.1]{mr02},  this definition of $\FF_\LL$ does not depend on the choice of $\Sigma(\FF_{\LL})$. It also follows from the proofs of Proposition~\ref{prop:cart} and Proposition~\ref{cart-at-p} that the propagated Selmer structure $\FF_{\LL}$ on the quotients $\{T_m\}$ coincides with $\FFc$, as long as the hypotheses \htamF\, and \hsezF\, hold for $T$.

Fix a free, rank one $\LL$-direct summand $\LLL\newnot{symbol:LLL}\subset H^1(F_{p},T\otimes\LL)$. By Proposition~\ref{structure-at-p}, this also fixes a free, rank one  $\LL/(\xx^m)$-direct summand $\mathcal{L}_m\subset H^1(F_{p},T_m)$. When $m=1$, we denote $\mathcal{L}_1$ simply by $\mathcal{L}$\newnot{symbol:Ll}.

\begin{define}
\label{modified selmer structure}
The $\LLL$-\emph{modified Selmer structure} $\FF_{\LLL}$\newnot{symbol:FFLL} on $T\otimes\LL$ is given by
\begin{itemize}
\item $\Sigma(\FF_{\LLL})=\Sigma(\FF_{\LL})$,
\item $H^1_{\FF_{\LLL}}(F_{v},T\otimes\LL)=H^1_{\FF_{\LL}}(F_{v},T\otimes\LL)$, for $v\nmid p$,
\item $H^1_{\FF_{\LLL}}(F_p,T\otimes\LL)= \LLL\subset H^1(F_p,T\otimes\LL)$.
\end{itemize}
\end{define}
The induced Selmer structure on the quotients $\mathcal{T}_0=\{T_{k,m}\}$ will also be denoted by $\FF_{\LLL}$ (except for the induced Selmer structure on $T$, which we will denote by $\FF_\mathcal{L}$\newnot{symbol:FFLl} for the sake of consistency with~\cite{kbbstark}). Note that the local conditions on $T_{k,m}$ at the primes $\lambda\nmid p$ determined by $\FF_{\LLL}$ coincide with the local conditions determined by $\FF_\LL$, and therefore also by $\FFc$ since  \htamF\, is true.
\begin{prop}
\label{prop:cart-L}
Under the assumptions of Proposition~\ref{structure-at-p},
\begin{enumerate}
\item[\textbf{(i)}] the Selmer structure $\FF_{\LLL}$ is cartesian on $\mathcal{T}_0$ in the sense of Definition~\ref{iwasawa-cartesian},
\item[\textbf{(ii)}] the core Selmer rank $\chi(T,\FF_\mathcal{L})$ of the Selmer structure $\FF_{\mathcal{L}}$ on $T$  is one.
\end{enumerate}
\end{prop}
\begin{proof}
(ii) is ~\cite[Proposition 1.8]{kbbstark}.

Since $\FFc$ and $\FF_{\LLL}$ determine the same local conditions at $v\nmid p$, the Selmer structure $\FF_{\LLL}$ is cartesian at $v\nmid p$ by Proposition~\ref{prop:cart}. Therefore, it remains to check that $\FF_{\LLL}$ is cartesian on $\mathcal{T}_0$ at the prime $p$.

The property $\mathbf{C1}$ holds by the definition of $\FF_{\LLL}$ on $\mathcal{T}_0$, and property $\mathbf{C.3}$ follows easily from~\cite[Lemma 3.7.1]{mr02} (which applies since $H^1(F_p,T_m)/\Ll_m$ is a free $\LL/(\xx^m)$-module for every $m$).

We now verify $\mathbf{C2}$. Let $\Ll_{k,m}$ be the image of $\Ll_m$ under the reduction map
$$H^1(F_p,T_m) \lra H^1(F_p,T_{k,m}).$$
 It is easy to see that the $R_{k,m}$-module $\Ll_{k,m}$ (resp., $H^1(F_p,T_{k,m})/\Ll_{k,m}$) is free of rank one (resp., of rank $r-1$). We need to check that the map
 $$\xymatrix@C=.55in{H^1(F_p,T_{k,m})/\Ll_{k,m} \ar[r]^(.48){[\xx^{M-m}]}& H^1(F_p,T_{k,M})/\Ll_{k,M} }$$
 is injective for all $M\geq m$. But this is evident, since the quotient $H^1(F_p,T_{k,m})/\Ll_{k,m}$ (resp., $H^1(F_p,T_{k,M})/\Ll_{k,M}$) is a free $R_{k,m}$-module (resp., a free $R_{k,M}$-module) of rank $r-1$.
\end{proof}
As a corollary of Proposition~\ref{prop:cart-L} and a trivial extension of Theorem~\ref{main} we obtain:
\begin{thm}
\label{main-stark}
Under the assumptions of Proposition~\ref{structure-at-p},
\begin{itemize}
\item[(i)] the $\LL$-module of $\LL$-adic Kolyvagin Systems $\overline{\KS}(T\otimes\LL,\FF_{\LLL},\PP)$  is free of rank one.
\item[(ii)] the map $\overline{\KS}(T\otimes\LL,\FF_{\LLL},\PP) \ra \overline{\KS}(T,\FF_{\mathcal{L}},\PP)$ is surjective.
\end{itemize}
\end{thm}
\begin{proof}
This is exactly Theorem~\ref{main} with $T=\oo(1)\otimes\rho^{-1}$ and the base field $\QQ$ replaced by $F$. The proof of Theorem~\ref{main} generalizes verbatim to the case when the base field is $F$, under the running hypotheses.
\end{proof}
\subsubsection{Rubin-Stark elements and $\LL$-adic Kolyvagin Systems}
Fix a $\LL$-line $\LLL$ as above. Fix also a finite set $S$ of places of $F$ that does \emph{not} contain any prime above $p$, but contains all the infinite places $S_{\infty}$, all primes $\lambda$ that divide the conductor $f_{\rho}$ of $\rho$. Assume that $|S| \geq r+1$ (where we recall that $r=[F:\QQ]$). Let
$$\kk_0=\{F_n(\tau): \tau \hbox{ is an integral ideal of } F \hbox{ which is prime to } f_{\rho}p\,;\,\, n \in\ZZ_{\geq 0}\}$$
 and let $\kk=\{L \cdot F_n(\tau): F_n(\tau) \in \kk_0\}$. Here, $F_n(\tau)=F_nF(\tau)$, and $F(\tau)$ denotes the maximal pro-$p$ extension of $F$ inside the ray class field of $F$ modulo $\tau$. For each $K \in\kk$, let $S_K=S \cup \{\hbox{places of $F$ at which $K$ is ramified}\}$ be another set of places of $F$. Let $O_{K,S_K}^{\times}$ denote the $S_K$ units of $K$, and $\Delta_K$ ({resp.}, $\delta_K$) denote $\hbox{Gal}(K/k)$ ({resp.}, $|\hbox{Gal}(K/k)|$).  Conjecture $B^{\prime}$ of \cite{ru96} predicts the existence of certain elements
 $$\varepsilon_{K,S_K}\newnot{symbol:ers} \in \Lambda_{K,S_K} \subset \frac{1}{\delta_K}{\wedge^r} O_{K,S_K}^{\times} \subset  \frac{1}{\delta_K}{\wedge^r} H^1(K,\ZZ_p(1))$$
  where $\Lambda_{K,S_K}$ is defined as in \cite[\S 2.1]{ru96}, the exterior products are  taken in the category of $\ZZ_p[\Delta_K]$-modules, and the final inclusion is induced by Kummer theory. Throughout this section, we assume the Rubin-Stark conjecture~\cite[Conjecture $B^{\prime}$]{ru96}.

Using elements of $\varprojlim_{K \in \kk} \bigwedge^{r-1}\hbox{Hom}_{\ZZ_p[\Delta_K]}(H^1(K,\ZZ_p(1)), \ZZ_p[\Delta_K])$  (or the images of the elements of
$$\varprojlim_{K \in \kk} \bigwedge^{r-1}\hbox{Hom}_{\ZZ_p[\Delta_K]}(H^1(K_p,\ZZ_p(1)), \ZZ_p[\Delta_K])$$
 under the canonical map
$$  \varprojlim_{K \in \kk} \bigwedge^{r-1}\hbox{Hom}_{\ZZ_p[\Delta_K]}(H^1(K_p,\ZZ_p(1)), \ZZ_p[\Delta_K])\lra \varprojlim_{K \in \kk} \bigwedge^{r-1}\hbox{Hom}_{\ZZ_p[\Delta_K]}(H^1(K,\ZZ_p(1)), \ZZ_p[\Delta_K])
 $$
 induced from localization at $p$) and the Rubin-Stark elements above, Rubin shows \cite[\S6 ]{ru96} (see also~\cite{pr-es}) how to obtain an Euler system for the $G_F$-representation $\ZZ_p(1)$. One may then obtain an Euler system for $T=\oo(1)\otimes\rho^{-1}$ using a standard twisting argument ({c.f.}, \cite[ \S II.4]{r00} and \cite[\S2.2]{kbbstark}).

In~\cite{kbbstark}, the author explicitly determines a generator for the $\oo$-module $\overline{\KS}(T,\FF_{\mathcal{L}},\PP)$ (which is a free $\oo$-module of rank one by~\cite[Theorem 5.2.10]{mr02}) using an Euler system of Rubin-Stark elements. Let $\pmb{\kappa}^{\textup{Stark}}\newnot{symbol:krs} \in \overline{\KS}(T,\FF_{\mathcal{L}},\PP)$ denote this generator constructed in~\cite{kbbstark} and denoted there by $\{\kappa_{\eta}^{\Phi_0}\}$, where $\Phi_0=\{\phi_0^{(\tau)}\}\in\varprojlim_{\tau} \bigwedge^{r-1}\hbox{Hom}_{\oo[\Delta_{\tau}]}(H^1(F(\tau)_p,T), \oo[\Delta_{\tau}])$ is a distinguished element which plays an important role in loc.cit. Here $\Delta_{\tau}=\Delta_{F(\tau)}$. See~\cite[\S2.3-\S2.4]{kbbstark} for details. We further recall from~\cite{kbbstark} that $\kappa_1^{\Phi_0}=\phi_0^{(1)}(\varepsilon_{L,S_L}^{\chi})$, where $\varepsilon_{L,S_L}^{\rho}$ is the $\rho$-part of the Rubin-Stark element $\varepsilon_{L,S_L}$.

Let $\Gamma_n=\Gal(F_n/F)$ and let $\gamma$ be a fixed generator of $\Gamma$. Let $\mathbb{L}_n$ be the image of $\LLL$ under the (surjective) homomorphism
$$H^1(F_p,T\otimes\LL) \lra H^1(F_p,T\otimes\LL/(\gamma^{p^n}-1)) \cong H^1(F_{n,p},T),$$
 where the last isomorphism follows from Shapiro's Lemma. Using the arguments of~\cite[\S2.3]{kbbstark}, one may choose an element
 $$\Phi^{\infty}=\{\Phi^{(n)}\} \in \varprojlim_{n} \bigwedge^{r-1}\hbox{Hom}_{\oo[\Gamma_n]}(H^1(F_{n,p},T), \oo[\Gamma_n]) $$
with the following properties:
\begin{itemize}
\item $\Phi^{(n)}$ maps $\bigwedge^{r}H^1(F_{n,p},T)$ isomorphically onto $\mathbb{L}_n$ for all $n \in \ZZ^+$ (and therefore $\Phi^{\infty}$ maps $\bigwedge^{r}H^1(F_p,T\otimes\LL)$ isomorphically onto $\LLL$).
\item $\Phi^{(0)}=\phi_0^{(1)}$ in  $\bigwedge^{r-1}\hbox{Hom}_{\oo}(H^1(F_{p},T), \oo)$.
\end{itemize}
Let $L_n=L\cdot F_n$ and let $\varepsilon_{L_n,S_{L_n}}^{\rho}$ be the $\rho$-part of the Rubin-Stark element. Set
$$\varepsilon_{F_{\infty}}^{\textup{Stark}}\newnot{symbol:rsinfty}=\{\varepsilon_{L_n,S_{L_n}}^{\rho}\} \in \varprojlim\bigwedge^{r}H^1(F_n,T)$$
 and $\mathbf{c}^{\textup{Stark}_\infty}_{1}=\Phi^{\infty}(\varepsilon_{F_{\infty}}^{\textup{Stark}}):= \{\Phi^{(n)}(\varepsilon_{L_n,S_{L_n}}^{\rho})\}_n \in\varprojlim_n H^1(F_n,T) = H^1(F,T\otimes\LL)$. Note that, since $\Phi^{(0)}=\phi_0^{(1)}$ by definition, it follows that
 $$\Phi^{(0)}(\varepsilon_{L,S_L}^\rho)=\phi_0^{(1)}(\varepsilon_{L,S_L}^\rho)=\kappa_{1}^{\Phi_0},$$
  i.e., $\kappa_{1}^{\Phi_0}$ is the image of $\mathbf{c}^{\textup{Stark}_\infty}_{1}$ under the projection $H^1(F,T\otimes\LL) \ra H^1(F,T).$

Using methods of Perrin-Riou~\cite{pr-es}, one may show that there is a $\LL$-adic Kolyvagin system
$\pmb{\kappa}^{\textup{pr},\infty}\newnot{symbol:prrs} \in \overline{\KS}(T\otimes\LL,\FFc,\PP)$ such that $$\kappa^{\textup{pr},\infty}_1=\mathbf{c}^{\textup{Stark}_\infty}_{1}\,\,\,\hbox{ and }\,\,\, \pmb{\kappa}^{\textup{pr},\infty} \equiv \pmb{\kappa}^{\textup{Stark}} \hbox{ mod } (\gamma-1),$$ The trouble is that $\pmb{\kappa}^{\textup{pr},\infty}$ does not necessarily belong to $\overline{\KS}(T\otimes\LL,\FF_{\LLL},\PP)$.

On the other hand, Theorem~\ref{main-stark} proves the existence of a $\LL$-adic Kolyvagin system 
$$\pmb{\kappa}^{\infty} \in \overline{\KS}(T\otimes\LL,\FF_{\LLL},\PP)$$
 such that $\pmb{\kappa}^{\infty} \equiv \pmb{\kappa}^{\textup{Stark}} \hbox{ mod }(\gamma-1)$. Let us pause at this point and briefly summarize the situation:
\begin{itemize}
\item The Kolyvagin system $\pmb{\kappa}^{\textup{pr},\infty}$ is constructed in terms of the Rubin-Stark elements $\varepsilon_{F_\infty}^{\textup{Stark}}$, however, $\pmb{\kappa}^{\textup{pr},\infty}$ does not necessarily live in $\overline{\KS}(T\otimes\LL,\FF_{\LLL},\PP)$, but only in $\overline{\KS}(T\otimes\LL,\FFc,\PP)$.
\item The Kolyvagin system $\pmb{\kappa}^{\infty} \in \overline{\KS}(T\otimes\LL,\FF_{\LLL},\PP)$ abstractly lifts the Kolyvagin system $\pmb{\kappa}^{\textup{Stark}}\in \overline{\KS}(T,\FF_{\Ll},\PP)$ which is constructed in~\cite{kbbstark}, but it is not a priori related to the Rubin-Stark elements $\varepsilon_{F_\infty}^{\textup{Stark}}$.
\end{itemize}

It requires to extend the methods of~\cite{kbbstark} to conclude that there is a $\LL$-primitive Kolyvagin system $\pmb{\kappa}^{\textup{Stark},\infty} \in \overline{\KS}(T\otimes\LL,\FF_{\LLL},\PP)$ for which $\kappa^{\textup{Stark},\infty}_1=\mathbf{c}^{\textup{Stark}_\infty}_{1} \in H^1(F,T\otimes\LL)$. This is the subject of a forthcoming paper~\cite{kbb-iwasawa}.

We now illustrate how one utilizes  $\pmb{\kappa}^{\infty} \in \overline{\KS}(T\otimes\LL,\FF_{\LLL},\PP)$ which our Theorem~\ref{main-stark} proves to exist. We assume Leopoldt's conjecture for every $L_n/L$ until the end of this section. This in particular implies that $H^1_{\FF_{\LLL}}(F,T\otimes\LL)$ injects into $\LLL$. We identify $\kappa^{\infty}_1 \in H^1_{\FF_{\LLL}}(F,T\otimes\LL)$  with its image inside $ \LLL$ under this injection.

For any $\oo$-module $M$, let $M^{\vee}:=\textup{Hom}(M,\Phi/\oo)$ denote its Pontryagin dual.
\begin{prop}\label{exact-seq-stark}
The following sequence of $\LL$-modules is exact:
$$ 0\lra \frac{H^1_{\FF_{\LLL}}(F,T\otimes\LL)}{\LL\cdot\kappa^{\infty}_1}\lra \frac{\LLL}{\LL\cdot\kappa^{\infty}_1} \lra
H^1_{\FF_{\textup{str}}^*}(F,(T\otimes\LL)^*)^{\vee} \lra H^1_{\FF_{\LLL}^*}\left(F,(T\otimes\LL)^*\right)^{\vee}\lra 0$$
Here $\FF_{\textup{str}}$ \textup{(}resp., $\FF_{\textup{str}}^*$\textup{)} is the Selmer structure $\FF_{\LL}$ \textup{(}resp., $\FF_{\LL}^*$\textup{)}, but local conditions at $p$ are replaced by the strict \textup{(}resp., relaxed\textup{)} conditions in the sense of~\cite[Definition 1.1.6]{mr02} and Definition~\ref{modified selmer-1} above.
\end{prop}
\begin{proof}
 First injection follows from Leopoldt's conjecture, and the rest of the sequence is exact by class field theory ({c.f.},~\cite[Theorems I.7.3 and Theorem III.2.10]{r00}, \cite[\S III.1.7]{deshalit} and \cite[Proposition 2.12]{kbb-iwasawa}).
\end{proof}

Since $\pmb{\kappa}^{\infty}$ maps to $\pmb{\kappa}^{\textup{Stark}}\in  \overline{\KS}(T,\FFl,\PP)$ under the map
$$\overline{\KS}(T\otimes\LL,\FF_{\LLL},\PP) \lra \overline{\KS}(T,\FFl,\PP)$$
 by its construction, and since $\pmb{\kappa}^{\textup{Stark}}$ is primitive (see the argument following the proof of~\cite[Theorem 3.11]{kbbstark}), it follows that $\pmb{\kappa}^{\infty}$ is $\LL$-primitive. We therefore conclude using~\cite[Theorem 5.3.10]{mr02} that
 $$\textup{char}\left(H^1_{\FF_{\LLL}}(F,T\otimes\LL)/\LL\cdot\kappa^{\infty}_1\right)=\textup{char}\left(H^1_{\FF_{\LLL}^*}(F,(T\otimes\LL)^*)^{\vee}\right).$$
  This together with Proposition~\ref{exact-seq-stark} gives:
\begin{cor}
\label{first reduction} 
Suppose $\rho(\wp)\neq1$ for any prime $\wp$ of $F$ above $p$. If Leopoldt's conjecture holds for fields $L_n$ \textup{(}with $n\in\ZZ_{\geq0}$\textup{)}, then
$$\textup{char}\left(\LLL/\LL\cdot\kappa^\infty_1\right)=\textup{char}\left((H^1_{\FF_{\textup{str}}^*}(F,(T\otimes\LL)^*)^{\vee}\right).$$
\end{cor}
\begin{rem}
\label{rem:dontneedRS}
Attentive reader will notice that we do not need to know the existence of Rubin-Stark elements for the proofs of Proposition~\ref{exact-seq-stark} and Corollary~\ref{first reduction}, but only the existence of a $\LL$-primitive $\LL$-adic Kolyvagin system $\pmb{\kappa}^\infty$, which we proved above in~Theorem~\ref{main-stark}.
\end{rem}
 Assume that one could choose the Kolyvagin system $\pmb{\kappa}^{\infty}$ so that $\kappa^{\infty}_1=\mathbf{c}^{\textup{Stark}_\infty}_{1} \in H^1(F,T\otimes\LL)$. We prove in~\cite{kbb-iwasawa} that this in fact is possible.

Since $\Phi^{\infty}$ maps $\bigwedge^rH^1(F_p,T\otimes\LL)$ isomorphically onto $\LLL$ (and maps the element $\varepsilon_{F_{\infty}}^{\textup{Stark}} \in \bigwedge^rH^1(F_p,T\otimes\LL)$ to $\mathbf{c}^{\textup{Stark}_\infty}_{1} \in \LLL\subset H^1(F_p,T\otimes\LL)$), it follows that
$$\bigwedge^rH^1(F_p,T\otimes\LL)/\LL\cdot\varepsilon_{F_{\infty}}^{\textup{Stark}} \cong \LLL/\LL\cdot\mathbf{c}^{\textup{Stark}_\infty}_{1}.$$
\begin{cor}
\label{second reduction}
Assume that the Rubin-Stark element $\varepsilon_{F_{\infty}}^{\textup{Stark}}$ above exists. Then, under the assumptions of Corollary~\ref{first reduction}, 
$$\textup{char}\left(\bigwedge^rH^1(F_p,T\otimes\LL)/\LL\cdot\varepsilon_{F_{\infty}}^{\textup{Stark}}\right)=\textup{char}\left(H^1_{\FF_{\textup{str}}^*}(F,(T\otimes\LL)^*)^{\vee}\right).$$
\end{cor}

Using the explicit description of the Galois cohomology groups in question (see \cite[\S I.6.3]{r00}), one may identify  $H^1_{\FF_{\textup{str}}^*}(k,\mathbb{T}^*)^{\vee}$ with $\Gal(M_\infty/L_\infty)^{\rho}$, where $M_\infty$ is the maximal abelian $p$-extension of $L_\infty$ which is unramified outside the primes above $p$. This is the Iwasawa module which is involved in the formulation of  the ``main conjectures" in this setting. Let $\mathcal{L}_{F}^{\rho}$\newnot{symbol:LrhoDR} denote (an appropriate normalization of) the Deligne-Ribet~\cite{deligne-ribet} $p$-adic $L$-function attached to the character $\rho$. As a consequence of the work of Wiles~\cite{wiles-mainconj}, we therefore deduce:
\begin{thm}
\label{conj-kbb}
Under the assumptions of Corollary~\ref{second reduction}, 
$$\mathcal{L}_{F}^{\rho}\cdot\LL=\textup{char}\left(\bigwedge^rH^1(F_p,T\otimes\LL)/\LL\cdot\varepsilon_{F_{\infty}}^{\textup{Stark}}\right).$$
\end{thm}

\begin{rem}
Although Corollary~\ref{second reduction} relies on the Rubin-Stark conjectures, the existence of the $\LL$-adic Kolyvagin systems derived from them is unconditional and is proved in Theorem~\ref{main-stark} above.
\end{rem}
\begin{rem}
Theorem~\ref{conj-kbb} generalizes (assuming the truth of the Rubin-Stark conjecture) a classical theorem of Iwasawa~\cite{iwasawa-1}, which states when $r=1$ (namely, when $k-\QQ$) that the characteristic ideal of the local units modulo the cyclotomic units is generated by the Kubota-Leopoldt $p$-adic $L$-function.  
\end{rem}
\section*{Acknowledgements} This paper is based on the author's Ph.D. thesis. The author would like to express his gratitude to his Ph.D. adviser Karl Rubin, for suggesting the problem, for many insightful discussions, his careful revision of this work and many comments to improve the exposition of this article. The author wishes to thank the mathematics department of Stanford University where he conducted his graduate studies. He also thanks the anonymous referees for numerous corrections, remarks and suggestions to improve the exposition of this article.
\appendix
\section{Local  Iwasawa theory via $(\varphi,\Gamma)$-modules}
\label{fontaine}

In this Appendix, we give an overview of certain results due to Benois, Colmez, Fontaine, Herr, and Perrin-Riou. We use these results to determine the structure of the semi-local cohomology groups $H^1(k_n(\tau)_p,{T})$. This  could have been achieved without appealing to the theory of $(\varphi,\Gamma)$-modules, however, the approach via $(\varphi,\Gamma)$-modules adds more perspective to local Iwasawa theory.

Throughout this Appendix, let $K$ denote a finite extension of $\QQ_p$ and set $\tilde{K}_n:=K(\mu_{p^n})$, $\tk_{\infty}:=\bigcup_{n}\tk_n$.  Define the Galois groups $\tilde{H}_K:=\textup{Gal}(\overline{K}/\tk_{\infty})$ and $\tGamma_K:=G_K/\tH_K=\textup{Gal}(\tk_{\infty}/K)$. Let $\tgamma$ be a topological generator of the pro-cyclic group $\tGamma_K$, and let $\tlambda_K:=\ZZ_p[[\tGamma_K]]$. Let $\tgamma_n$ be a fixed topological generator of $\textup{Gal}(\tk_{\infty}/\tk_n):=\tGamma^{(n)}$ for $n \in \ZZ^+$, which is chosen in a way that $\tgamma_n=\tgamma_1^{p^{\alpha_n}}$, where $\alpha_n \in \ZZ^+$ is such that $[\tk_n:\tk]=p^{\alpha_n}$.

 Let $K_n$ be the maximal $p$-extension of $K$ inside $\tk_n$, and let $\bigcup_n K_n=:K_{\infty} \subset \tk_{\infty}$ be the cyclotomic $\ZZ_p$-extension of $K$. We set $\Gamma_K:=\textup{Gal}(K_{\infty}/K)$ and $\Lambda_K:=\ZZ_p[[\Gamma_K]]$. Note then that
 $$\tGamma_K=W\times \Gamma_K,\hbox{ and that } \tlambda_K=\ZZ_p[W]\otimes_{\ZZ_p}\Lambda_K,$$
 where $W$ is a finite group whose order is prime to $p$. (In fact $W$ can be identified by $\textup{Gal}(K(\mu_p)/K)$.) Let $\gamma$ denote the restriction of $\tgamma$ to $K_{\infty}$; so that the element $\gamma$ is a topological generator of $\Gamma_K$. Let $\gamma_n$ denote the image of $\tgamma_n$ under the natural isomorphism
 $$\textup{Gal}(\tk_{\infty}/\tk_n) \cong \textup{Gal}(K_{\infty}/K_n),$$
 and set $H_K:=\textup{Gal}(\overline{K}/K_{\infty})$ (so that $H_K/\tH_K\cong W$).

 In~\cite{fontaine-phi-gamma}, Fontaine\footnote{However, one should be cautious as Fontaine uses $K_n$ for our $\tk_n$, etc. For example, his $\Gamma_K$ is our $\tGamma_K$.} introduces the notion of a $(\varphi,\Gamma)$-module over a certain period ring which he denotes by $\fontaineO$ (and which is the ring of integers of the field $\widehat{\varepsilon^{\textup{nr}}}$). We set $\mathcal{O}_{\varepsilon(K)}:=(\mathcal{O}_{\widehat{\varepsilon^{\textup{nr}}}})^{H_K}$. We will not include a  detailed discussion of these objects here, we refer the reader to~\cite[A.3.1-3.2]{fontaine-phi-gamma} for the definitions and the basic properties of these rings. Briefly, a $(\varphi,\Gamma)$-module over $\fontaineOK$ is a finitely generated $\fontaineOK$-module with semi-linear continuous and commuting actions of $\varphi$ and $\Gamma:=\tGamma_K$. A $(\varphi,\Gamma)$-module $D$ over $\fontaineOK$ is called \emph{\'etale} if $\varphi(D)$ generates $D$ as an $\fontaineOK$-module.

 Using his theory, Fontaine established an equivalence between the category of $\ZZ_p$-representations of the absolute Galois group $G_K$ of $K$, and the category of \'etale $(\varphi,\Gamma)$-modules over $\fontaineOK$. This equivalence is given by
 $$\begin{array}{rcl}
 T&\longmapsto &D(T):=(\fontaineO \otimes_{\ZZ_p}T)^{H_K}\\
 (\fontaineO \otimes_{\fontaineOK}D)^{\varphi=1}=:T(D)&\longmapsfrom&D
 \end{array}
 $$
 See \cite[A.1.2.4-1.2.6 ]{fontaine-phi-gamma} for details.

Suppose $T$ is any $\ZZ_p[[G_K]]$-module which is free of finite rank over $\ZZ_p$. In~\cite{herr-1}, Herr makes use of the theory of $(\varphi,\Gamma)$-modules to compute the Galois cohomology groups $H^*(K,T)$. One of the benefits of his approach is that the complex he constructs with cohomology $H^*(K,T)$ is quite explicit. This allows one to compute certain local Galois cohomology groups of $p$-adic fields. In~\cite{herr-2}, Herr gives a proof of the local Tate duality, where the local pairing (c.f.,~\cite[\S5]{herr-2})  is explicitly defined in terms of the residues of the differential forms on $\fontaineOK$. The rest of this Appendix is a survey of Herr's results and their applications \cite{benois, colmez} to Iwasawa theory.

In Fontaine's theory of $(\varphi,\Gamma)$-modules, there is an important operator\footnote{This definition makes sense because: \begin{itemize}
\item  $\textup{Tr}_{\widehat{\varepsilon^{\textup{nr}}}/\varphi(\widehat{\varepsilon^{\textup{nr}}})}(\fontaineO) \subset p\fontaineO$,
\item $\varphi$ is injective.
\end{itemize}
}
$$\psi: \fontaineO \lra \fontaineO,$$
$$\psi(x):= \frac{1}{p}\varphi^{-1}(\textup{Tr}_{\widehat{\varepsilon^{\textup{nr}}}/\varphi(\widehat{\varepsilon^{\textup{nr}}})}(x)),$$
which is crucial for what follows. The map $\psi$ is a left inverse of $\varphi$, and its action on $\fontaineO$ commutes with the action of $G_K$. It induces an operator (which we still denote by $\psi$)
$$\psi: D(T) \lra D(T)$$ for any $G_K$-representation $T$.

Let $C_{\psi,\tgamma}$ be the complex
$$\xymatrix{
C_{\psi,\tgamma}: \,0 \ar[r] &D(T) \ar[rr]^(.4){(\psi-1,\tgamma-1)} &&D(T) \oplus D(T) \ar[rr]^(.57){(\tgamma-1) \ominus (\psi-1)}&&D(T)\ar[r] &0
}$$
The main result of~\cite{herr-1} is the following:

\begin{Athm}
\label{herr-main}
The complex $C_{\psi,\tgamma}$ computes the $G_K$-cohomology of $T$:
\begin{enumerate}
\item[\textbf{(i)}] $H^0(K,T)\cong D(T)^{\psi=1, \tgamma=1}$,
\item[\textbf{(ii)}]$H^2(K,T) \cong D(T)/(\psi-1,\tgamma-1)$,
\item[\textbf{(iii)}] There is an exact sequence
$$0 \lra \frac{D(T)^{\psi=1}}{\tgamma-1} \lra H^1(K,T)\lra \left(\frac{D(T)}{\psi-1}\right)^{\tgamma=1} \lra 0.$$
\end{enumerate}
All the isomorphisms and maps that appear above are functorial in $T$ and $K$.
\end{Athm}

\begin{Adefine}
\label{iwasawa-p}
Let
$$H^1_{\tilde{\textup{Iw}}}(K,T):=\varprojlim_{n}H^1(\tk_n,T) \hbox{ and, }$$
$$H^1_{\textup{Iw}}(K,T):=\varprojlim_{n}H^1(K_n,T),$$
 where the inverse limits are taken with respect to the corestriction maps.
\end{Adefine}

\begin{Arem}
\label{K-tilde-vs-K}
Since the order of $W$ is prime to $p$, it follows that
$$H^1_{\textup{Iw}}(K,T) \stackrel{\sim}{\lra}H^1_{\tilde{\textup{Iw}}}(K,T)^W$$
 by the Hochschild-Serre spectral sequence.
\end{Arem}
We now determine the structure of $H^1_{\textup{Iw}}(K,T)$ using Theorem~A.\ref{herr-main}.
\begin{Aprop}
\label{nth-layer-p}
Define $\tau_n:=1+\tgamma_{n-1}+ \dots + \tgamma_{n-1}^{p-1} \in \ZZ_p[[\tGamma_K]]$. There is a commutative diagram with exact rows:
$$\xymatrix@R=.2in{C_{\psi,\tgamma_n}(\tk_n, T): 0 \ar @{.>}[d]^{\tau_n^{*}} \ar[r] &D(T) \ar[r]\ar[d]^{\tau_n}& D(T) \oplus D(T) \ar[r] \ar[d]^{\tau_n \oplus \textup{id}}& D(T) \ar[r]\ar[d]^{\textup{id}}& 0\\
C_{\psi,\tgamma_{n-1}}(\tk_{n-1}, T): 0 \ar[r] &D(T) \ar[r]& D(T) \oplus D(T) \ar[r]& D(T) \ar[r]& 0
}$$
Furthermore, the map induced from the morphism $\tau_n^*$ on the cohomology of $C_{\psi,\tgamma_n}(\tk_n, T)$ coincides with the corestriction map under Herr's identification $H^*(C_{\psi,\tgamma_n}(\tk_n, T)) \cong H^*(\tk_n,T)$ of Theorem~A.\ref{herr-main}.
\end{Aprop}

\begin{proof}
This follows from the fact that $\tau_n^*$ is a cohomological functor and induces $\textup{Tr}_{\tk_n/\tk_{n-1}}$ on $H^0$, hence it induces corestrictions on $H^i$.
\end{proof}

Using Proposition~A.\ref{nth-layer-p}, one may compute $H^*_{\tilde{\textup{Iw}}}(K,T)$:
\begin{Athm}
\label{cohomology-iwasawa-p-tilde}
\begin{enumerate}
\item[\textbf{(i)}] $H^i_{\tilde{\textup{Iw}}}(K,T)=0$ if $i \neq 1,2$,
\item[\textbf{(ii)}] $H^1_{\tilde{\textup{Iw}}}(K,T)\stackrel{\sim}{\lra} D(T)^{\psi=1}$,
\item[\textbf{(iii)}] $H^2_{\tilde{\textup{Iw}}}(K,T) \stackrel{\sim}{\lra} D(T)/(\psi-1)$.
\end{enumerate}

\end{Athm}
See~\cite[\S{II.3}]{colmez} for a proof of this theorem.
\begin{Arem}
\label{rem:p-adic L-function}
The isomorphism $\exp^*: H^1_{\tilde{\textup{Iw}}}(K,T)\stackrel{\sim}{\lra} D(T)^{\psi=1}$ of Theorem~A.\ref{cohomology-iwasawa-p-tilde}(ii) can be considered as a vast generalization of Coleman's map~\cite{coleman}. The isomorphism  $\exp^*$ conjecturally gives rise to the (conjectural) $p$-adic $L$-function attached to $T$. This viewpoint we gain is one of the important benefits of using the theory of $(\varphi,\Gamma)$-modules to compute Galois cohomology.
\end{Arem}

Let $\mathcal{C}(T):=(\varphi-1)D(T)^{\psi=1}$. Since $\psi$ is a left inverse of $\varphi$, it follows that
$$\ker\{D(T)^{\psi=1} \stackrel{\varphi-1}{\lra} \mathcal{C}(T)\}=D(T)^{\varphi=1},$$
 hence we have an exact sequence:
\be
\label{phi-psi-sequence}
0 \lra D(T)^{\varphi=1} \lra D(T)^{\psi=1} \stackrel{\varphi-1}{\lra}\mathcal{C}(T) \lra 0.
\ee

Using techniques from the theory of $(\varphi,\Gamma)$-modules, one can determine the structure of $\mathcal{C}(T)$:
\begin{Aprop}
\label{free-C}
The  $\tilde{\Lambda}_K$-module $\mathcal{C}(T)$ is free of rank $[K:\QQ_p] \cdot\textup{rank}_{\ZZ_p}T$.
\end{Aprop}

One may also check that $D(T)^{\varphi=1}\cong T^{\tH_K}$. In particular, $D(T)^{\varphi=1}$ is finitely generated over $\ZZ_p$, hence is a torsion $\ZZ_p[[\tGamma_K]]$-module. Thus, it follows from Proposition~A.\ref{free-C} and Theorem~A.\ref{cohomology-iwasawa-p-tilde} that $D(T)^{\varphi=1}=H^1_{\tilde{\textup{Iw}}}(K,T)_{\textup{tors}}$, the torsion submodule of $H^1_{\tilde{\textup{Iw}}}(K,T)$.

If we now take the $W$-invariance of the exact sequence~(\ref{phi-psi-sequence}) (and use the fact that taking $W$-invariance is an exact functor) and  apply Remark~A.\ref{K-tilde-vs-K} along with Theorem~A.\ref{cohomology-iwasawa-p-tilde} and Proposition~A.\ref{free-C}, we see that:

\begin{Athm}\textup{(Perrin-Riou, see~\cite{colmez} for the proof we sketch in this Appendix)}
\label{cohomology-iwasawa-p}
\begin{enumerate}
\item[\textbf{(i)}] For the $\LL_K$-torsion submodule $H^1_{\textup{Iw}}(K,T)_{\textup{tors}}$ of $H^1_{\textup{Iw}}(K,T)$, we have $H^1_{\textup{Iw}}(K,T)_{\textup{tors}}\cong T^{H_K}$.
\item[\textbf{(ii)}] The  $\LL_K$-module $H^1_{\textup{Iw}}(K,T)/H^1_{\textup{Iw}}(K,T)_{\textup{tors}}$ is free of rank $[K:\QQ_p]\cdot \textup{rank}_{\ZZ_p}T$.
\end{enumerate}
\end{Athm}

\bibliographystyle{alpha}
\bibliography{biblio-tez}

\section*{List of Symbols\hfill} 
\begin{tabbing}
$\chi(T,\FF)\,\,\,\,\,\,\,$~~~~~~~~~~~~~\=\parbox{5in}{Core Selmer rank\dotfill \pageref{symbol:coreselmer}}\\
\addsymbol \oo: {Coefficient ring}{symbol:coeff}
\addsymbol \Phi: {Field of fractions of $\oo$}{symbol:ff}
\addsymbol \mm: {Maximal ideal of $\oo$}{symbol:maxidealO}
\addsymbol \varpi: {Uniformizer of $\oo$}{symbol:uniformizerO}
\addsymbol T: {Galois representation}{symbol:T}
\addsymbol T^*: {Dual representation}{symbol:T*}
\addsymbol A: {The $p$-divisible module attached to $T$}{symbol:A}
\addsymbol \QQ_\infty: {Cyclotomic $\ZZ_p$-extension of $\QQ$}{symbol:qinfty}
\addsymbol G_K: {Galois group of a fixed separable closure of a field $K$}{symbol:Gal}
\addsymbol R_{k,m}: {A collection of dimension-two artinian rings}{symbol:rkm}
\addsymbol V: {Vector space attached to $T$}{symbol:V}
\addsymbol \mathcal{D}_\ell: {A fixed decomposition group at $\ell$}{symbol:Dec}
\addsymbol \mathcal{I}_\ell: {Inertia subgroup}{symbol:Inert}
\addsymbol \textup{Fr}_\ell: {Frobenius element}{symbol:frob}
\addsymbol \QQ_{\ell}^{\textup{unr}}: {Maximal unramified extension of $\QQ_\ell$}{symbol:maxunr}
\addsymbol H^1_f(\QQ_\ell,\cdot): {Finite local condition}{symbol:finite}
\addsymbol H^1_{\textup{unr}}(\QQ_\ell,\cdot): {Unramified local condition}{symbol:unr}
\addsymbol \Gamma: {}{symbol:Gamma}
\addsymbol \LL: {Cyclotomic Iwasawa algebra}{symbol:LL}
\addsymbol T_f: {}{symbol:Tf}
\addsymbol T_{k,f}: {}{symbol:Tkf}
\addsymbol \FF: {Selmer structure}{symbol:selmerstr}
\addsymbol T_{k,m}: {}{symbol:Tkm}
\addsymbol \FFc: {Canonical Selmer structure}{symbol:FFc}
\addsymbol H^1_{\textup{tr}}(\QQ_\ell,\cdot): {Transverse local condition}{symbol:trans}
\addsymbol \FF^{a}_b(c): {Selmer structure $\FF$ modified at primes dividing $abc$}{symbol:modsel}
\addsymbol \PP_{k,m}: {Kolyvagin primes for $T_{k,m}$}{symbol:pkm}
\addsymbol \PP: {Kolyvagin primes}{symbol:primes}
\addsymbol \FF^*: {Dual Selmer structure}{symbol:FF*}
\addsymbol \mathcal{S}: {Simplicial sheaf}{symbol:calS}
\addsymbol \mathcal{H}: {Selmer sheaf}{symbol:calH}
\addsymbol \overline{\KS}: {$\LL$-module $\LL$-adic Kolyvagin systems}{symbol:KS}
\addsymbol \mathbf{c}^\rho: {Cyclotomic unit Euler system}{symbol:crho}
\addsymbol \pmb{\kappa}^{\rho,\infty}: {Cyclotomic unit $\LL$-adic Kolyvagin system}{symbol:krhoinfty}
\addsymbol \pmb{\kappa}^{\rho}: {Cyclotomic unit Kolyvagin system}{symbol:krho}
\addsymbol \pmb{\kappa}^{\textup{Kato},\infty}: {Kato's $\LL$-adic Kolyvagin system}{symbol:katoinfty}
\addsymbol \pmb{\kappa}^{\textup{Kato}}: {Kato's Kolyvagin system}{symbol:kato}
\addsymbol \mathbf{\xiv}_r: {Perrin-Riou's conjectural cohomology class}{symbol:xi}
\addsymbol \pmb{\kappa}^{\textup{PR}}: {$\LL$-adic Kolyvagin system derived from Perrin-Riou's conjectural classes}{symbol:PR}
\addsymbol F: {A fixed totally real number fields}{symbol:F}
\addsymbol \LLL: {A $\LL$-line inside the local cohomology at $p$}{symbol:LLL}
\addsymbol \Ll: {An $\oo$-line inside the local cohomology at $p$}{symbol:Ll}
\addsymbol \FF_{\LLL}: {Modified Selmer structure on $T\otimes\LL$}{symbol:FFLL}
\addsymbol \FF_{\Ll}: {Modified Selmer structure on $T$}{symbol:FFLl}
\addsymbol \varepsilon_{K,S_K}: {The Rubin-Stark element}{symbol:ers}
\addsymbol \pmb{\kappa}^{\textup{Stark}}: {Kolyvagin system derived from Rubin-Stark elements}{symbol:krs}
\addsymbol  \varepsilon_{F_\infty}^{\textup{Stark}}: {The collection of Rubin-Stark elements along the cyclotomic tower}{symbol:rsinfty}
\addsymbol \pmb{\kappa}^{\textup{pr},\infty}: {$\LL$-adic Kolyvagin system for $\FFc$ obtained from Rubin-Stark elements}{symbol:prrs}
\addsymbol \Ll_F^\rho: {Deligne-Ribet $p$-adic $L$-function attached to $\rho$}{symbol:LrhoDR}

\end{tabbing}
\end{document}